\documentclass[10pt]{amsart}
\usepackage{amssymb}
\usepackage{amscd}
\usepackage[all]{xy}
\usepackage[colorlinks=true, linkcolor=red,
urlcolor=blue]{hyperref}
\newcounter{TmpEnumi}

\numberwithin{equation}{section}

\def\today{\number\day\space\ifcase\month\or  
 January\or February\or
   March\or April\or May\or June\or  
    July\or August\or September\or
   October\or November\or December\fi\ 
     \number\year}

\theoremstyle{definition}
\newtheorem{thm}{Theorem}[section]
\newtheorem{lem}[thm]{Lemma}
\newtheorem{prp}[thm]{Proposition}
\newtheorem{dfn}[thm]{Definition}
\newtheorem{cor}[thm]{Corollary}

\newtheorem{rmk}[thm]{Remark}
\newtheorem{ntn}[thm]{Notation}
\newtheorem{eme}[thm]{Example}

\newtheorem{qst}[thm]{Question}

\newcommand{\beq}{\begin{equation}}
\newcommand{\eeq}{\end{equation}}
\newcommand{\beqr}{\begin{eqnarray*}}
\newcommand{\eeqr}{\end{eqnarray*}}
\newcommand{\bal}{\begin{align*}}
\newcommand{\eal}{\end{align*}}
\newcommand{\bei}{\begin{itemize}}
\newcommand{\eei}{\end{itemize}}
\newcommand{\limi}[1]{\lim_{{#1} \to \infty}}

\newcommand{\af}{\alpha}
\newcommand{\bt}{\beta}

\newcommand{\dt}{\delta}
\newcommand{\ep}{\varepsilon}

\newcommand{\ld}{\lambda}

\newcommand{\ph}{\varphi}

\newcommand{\K}{\mathcal{K}}

\newcommand{\N}{{\mathbb{Z}}_{> 0}}
\newcommand{\Nz}{{\mathbb{Z}}_{\geq 0}}

\newcommand{\Ped}{{\operatorname{Ped}}}
\newcommand{\Cu}{{\operatorname{Cu}}}

\newcommand{\id}{{\operatorname{id}}}

\newcommand{\spec}{{\operatorname{sp}}}

\newcommand{\card}{{\operatorname{card}}}
\newcommand{\Aut}{{\operatorname{Aut}}}

\newcommand{\EQT}{{\operatorname{EQT}}}
\newcommand{\QT}{{\operatorname{QT}}}

\newcommand{\ET}{{\operatorname{ET}}}

\newcommand{\Dom}{{\operatorname{Dom}}}
\newcommand{\rc}{{\operatorname{rc}}}

\newcommand{\cF}{{\mathcal{F}}}

\newcommand{\cK}{{\mathcal{K}}}

\newcommand{\cM}{{\mathcal{M}}}

\newcommand{\cW}{{\mathcal{W}}}

\newcommand{\cZ}{{\mathcal{Z}}}

\newcommand{\dirlim}{\varinjlim}

\newcommand{\andeqn}{\qquad {\mbox{and}} \qquad}

\newcommand{\ca}{C*-algebra}

\newcommand{\hm}{homomorphism}

\newcommand{\hsa}{hereditary C*-subalgebra}

\newcommand{\cfn}{continuous function}

\newcommand{\I}{\infty}

\newcommand{\Lem}[1]{Lemma~\ref{#1}}
\newcommand{\Def}[1]{Definition~\ref{#1}}

\title[The relative radius of comparison]{The relative radius of comparison of the crossed product of a non-unital C*-algebra by a finite group }   

\date{\today}

\author{M. Ali Asadi-Vasfi}
\address{ M. Ali Asadi-Vasfi, 
The Fields Institute, 222 College Street, Toronto,
ON M5T 3J1 Canada
}

\author{George A. Elliott}
\address{George A. Elliott, 
Department of Mathematics, University of Toronto, Toronto, ON M5S 2E4 Canada
}



\begin{document}

\maketitle
\begin{abstract}
In this paper, we prove results on the relative radius of comparison of  C*-algebras and their crossed products, focusing on the non-unital setting.
More precisely,
let  $A$ be a   stably finite simple non-type-I (not necessarily unital) C*-algebra, 
let $G$ be a finite group, and
let $\alpha \colon G \to {\operatorname{Aut}} (A)$ be an action
 which has the weak tracial Rokhlin property.   
 Let $a$ be a non-zero positive element 
in
$A^\alpha \otimes \mathcal{K}$.
Then we show that 
the radius of comparison of 
$\Cu (A^\alpha)$ relative to  $[a]$
is bounded above by
the radius of comparison of 
$\Cu (A)$ relative to $[a]$. 
If further $A$ is exact and $a$ is 
in the Pedersen ideal  of
$A^\alpha \otimes \mathcal{K}$, then 
the radius of comparison of 
$\Cu (A\rtimes_{\alpha} G)$
relative to  $[a]$ is equal to
its radius of comparison relative to 
 $[p\cdot a]$, scaled by
 $1/|G|$, where $p$ is the averaging projection in the multiplier algebra of 
$(A \otimes \mathcal{K}) \rtimes_{\alpha \otimes \id} G$. Moreover,
the radius of comparison of 
$\Cu (A\rtimes_{\alpha} G)$
relative to  $[a]$ is bounded above by  $1/|G|$  times 
the radius of comparison of $\Cu (A)$
relative to  $[a]$. 
We also prove that
the inclusion of $A^{\alpha}$ in $A$
induces an isomorphism from the purely positive part of the
Cuntz semigroup ${\operatorname{Cu}} (A^{\alpha})$ to the fixed point
of the purely positive part of 
${\operatorname{Cu}} (A)$. 
An important consequence of our results is that they apply to non-unital C*-algebras and give new insights into comparison theory of C*-algebras and their crossed products.  
\end{abstract}

\tableofcontents
\section{Introduction}
Over the last decade, major advances in comparison theory for C*-algebras have culminated.  This theory extends the Murray-von Neumann comparison theory for factors. Blackadar
in \cite{B10} called it  ``\emph{the Fundamental Comparability Question}". 
The associated ideas, in particular the radius of comparison, play a very important role in the Elliott program for the classification of C*-algebras \cite{CE08, ET08, Tom08}.
The radius of comparison of a unital stably finite C*-algebra $A$, 
 denoted by $\operatorname{rc} (A)$,  is a numerical invariant introduced by Andrew Toms in Section~6 of \cite{Tom06}  to distinguish  non-isomorphic simple separable unital AH algebras that share the same Elliott invariant (see Section~5  of \cite{Tom09} for additional applications of the radius of comparison).
 This invariant 
 is based on the Cuntz semigroup and the quasitraces on $A$. Namely, it measures the failure of quasitraces to predict comparison of positive elements.
 It is 
well known that the radius of comparison of $C(X)$ for a compact space $X$ is bounded above by $\frac{1}{2} \cdot \dim(X)$
and bounded below by 
$\max \{ 
0, \tfrac{1}{2} \cdot (\dim (X) -7) \}$ (see \cite{EN13,Ph23}).
 This justifies why the radius of comparison can be viewed as a non-commutative analogue of the covering dimension in topology.   
 We refer to \cite{AA20} for 
 estimates on
 the radius of comparison 
 of $C(X) \otimes A$, where 
 $X$ is a compact space and 
$A$ is a unital stably finite C*-algebra.

Even during the Elliott classification program, which has now successfully classified simple separable amenable 
 C*-algebras 
 (see \cite{GLN20, KR95, Ph00, TWW17}), using invariants related to K-theory, attention began to  shift to non-classifiable C*-algebras—those with a non-zero radius of comparison. 
 For instance, for a compact space $X$,
 the radius of comparison of
 $C(X) \rtimes_{h} \mathbb{Z}$, where $h$
 a minimal homeomorphism
of $X$, was
 studied in 
  \cite{HPTInPre} and later in \cite{Ni22}, and when 
$h$ is a subshift, in \cite{GK10}.
  Additionally,
the radius of comparison of 
 $C(X) \rtimes_{T} G$, where 
 $G$ is a countable amenable group
 and 
 $T$ is an action of $G$ on $X$, has been investigated 
in  the work of Hirshberg and Phillips~\cite{HP22}. 
  
In the light of these examples involving  topological dynamics, one may ask for an
   understanding of the relationship between the radius of comparison of the crossed product, the fixed point algebra, and the underlying 
C*-algebra when we have a C*-dynamical system. This question is very interesting and important even for finite group actions on unital C*-algebras. 
For instance, 
for an action 
$\alpha \colon G \to \Aut(A)$ of a finite group $G$ 
on a stably finite simple infinite-dimensional unital C*-algebra $A$, it is shown in 
\cite{AGP19} that 
if $\alpha$ has the weak tracial Rokhlin property in the sense of Definition~3.2 of \cite{AGP19}, originally introduced in \cite{ArchThes, HO13}, then
the radius of comparison of the crossed product $A \rtimes_{\alpha} G$, the fixed point algebra 
$A^{\alpha}$, and the radius of comparison of $A$ satisfy
\begin{equation}\label{Result1}
{\operatorname{rc}} (A^{\alpha}) \leq {\operatorname{rc}} (A)
\qquad\text{ and }\qquad
{\operatorname{rc}} \big( A \rtimes_{\alpha} G \big)
 \leq \frac{1}{\card (G)} \cdot {\operatorname{rc}} (A).
\end{equation}
Also, in \cite{ASV23}, it is proved  that if $\alpha$ is weakly tracially strictly approximately  inner in the sense of Definition~4.1 of \cite{ASV23}, and 
$A \rtimes_{\alpha} G$ is simple, then
\begin{equation}\label{AppInner}
\operatorname{rc} (A) \leq   \operatorname{rc} \big( A \rtimes_{\alpha} G \big) \leq \operatorname{rc} (A^{\alpha}).
\end{equation}
In particular, if $G$ is nontrivial, then both 
(\ref{Result1}) and (\ref{AppInner}) cannot hold at the same time unless all the radii of comparison are zero. 
 
Less is known about the radius of comparison of crossed products of simple non-classifiable 
C*-algebras by non-finite groups, as these can behave unpredictably. For example, in \cite{H24}, Hirshberg has constructed a simple unital separable AH-algebra 
$A$ with $\operatorname{rc} (A)>0$ such that $A$ admits a properly outer automorphism 
$\alpha$ and
$\operatorname{rc} (A \rtimes_{\alpha} \mathbb{Z})=0$. 

An important question, proposed in \cite{AGP19}, is whether the  result in
(\ref{Result1}) above is still valid in the setting of non-unital C*-algebras (see the discussion after Question~7.5 of \cite{AGP19}). Complications in answering this question include the additional complexity associated with the definition of the
weak tracial Rokhlin property for non-unital C*-algebras
and, also, what to substitute for the conventional definition of the radius of comparison. 
An appropriate version of the weak tracial Rokhlin property for 
non-unital simple C*-algebras
was introduced in Definition~3.1 of \cite{FG20}. 
Additionally, the definition of the
radius of comparison was extended to non-unital C*-algebras
by
Blackadar, Robert, Tikuisis, Toms, and Winter
 in Definition~ 3.3.2 of \cite{BRTW12}, using more algebraic terms. The idea is to use the Cuntz class of a positive full element in $A \otimes \cK$ instead of the Cuntz class of the unit of $A$. For a C*-algebra $A$ and a full element $a$ in $(A \otimes \cK)_+$, this version of the radius of comparison,
  is called  ``\emph{the radius of comparison of $\Cu (A)$  relative to $[a]_A$}'' or sometimes simply ``\emph{the relative radius of comparison of $A$}'' as $\Cu(A)$ is determined by $A$ (see 
Definition~\ref{DefRcCuAlgebraic}
for an algebraic version of it
and Definition~\ref{DefRcCu}
for a non-algebraic one). It is denoted by
$\rc \big(\Cu (A), \ [ a ]_A \big)$ or sometimes $\rc \big(\Cu (A),\ [ a ] \big)$ when there is no confusion related to the Cuntz class of $a$. 
 It is known that for a unital 
C*-algebra $A$ with all quotients stably finite, we have 
$\rc \big(\Cu (A), \ [ 1_A ]_A \big) = \operatorname{rc} (A)$ (see Proposition~3.2.3 of \cite{BRTW12}).

In this paper, we prove the following fact. Let  $A$ be a
 stably finite simple  non-type-I (not necessarily unital)  C*-algebra and let $\alpha \colon G \to \Aut (A)$ be an action of a finite group $G$ on $A$ with the weak tracial Rokhlin property. Let $a \in (A^{\alpha} \otimes \K)_{+} \setminus \{0\}$. 
Let $\iota \colon A^\alpha \to A$ and 
$\kappa \colon A \to A\rtimes_{\alpha} G$ 
denote the inclusion maps. First, we prove that 
for purely positive elements (recalled in Definition~\ref{D_9421_Pure}) of the fixed point algebra $A^\alpha$,  Cuntz comparison over $A$ and Cuntz comparison over $A^\alpha$ are equivalent (see Proposition~\ref{W.T.R.P.inject}). This allows us to establish some relations among the relative radius of comparison of the crossed product by $\alpha$, the fixed point algebra $A^\alpha$, and the underlying C*-algebra $A$,  as well as for their Cuntz semigroups. Namely, we show in Theorem~\ref{RcofFixedPoint} that
\begin{equation}\label{EQ1.2024.09.10}
\rc \big(\Cu (A^\alpha),  \ [a] \big)  \leq \rc \big(\Cu(A),\ [\iota ( a )] \big).
\end{equation}
This generalizes Theorem~4.1 of \cite{AGP19}.
 If further, $A$ is exact, 
 $a \in \Ped(A^\alpha \otimes \mathcal{K}) \setminus \{0\}$,
 and 
$p= \frac{1}{\card (G)} \cdot \sum_{g} u_g$ in 
$\cM \big((A \otimes \mathcal{K}) \rtimes_{\alpha \otimes \id} G\big)$, then we prove in Proposition~\ref{RCandCardinal} and Theorem~\ref{MainThmRC}(\ref{MainThmRC.1}) that
 \begin{align*}
  \frac{1}{\card (G)}  \cdot 
  \rc  \Big(
  \Cu (A \rtimes_{\alpha} G), \  [p \cdot (\kappa \circ \iota) (a)]
   \Big) 
  &=
  \rc  \Big(\Cu(A \rtimes_{\alpha} G), \
  (\kappa \circ \iota) (a)  \Big)  
  \\&=  \frac{1}{\card (G)} \cdot
  \rc \big(\Cu (A^\alpha),\  [a]\big).
  \end{align*}
 By this and (\ref{EQ1.2024.09.10}), we show in Theorem~\ref{MainThmRC}(\ref{MainThmRC.2}) that
\[
\rc \Big(\Cu (A \rtimes_{\alpha} G), \
  [ (\kappa \circ \iota) (a)]  \Big) 
  \leq 
  \frac{1}{\card (G)}  \cdot 
  \rc \big(
  \Cu (A), \ [\iota ( a )]
  \big).
  \] 
  Further, in close analogy with comparison theory, we prove a  stronger result regarding the Cuntz semigroup of the fixed point algebra. Namely, we show 
  in Theorem~\ref{Thm.05.22.2019} that the inclusion map 
   $\iota \colon A^\alpha \to A $ induces an isomorphism of ordered semigroups 
$\Cu (\iota)$ from the purely positive part of the Cuntz semigroup of the fixed point algebra together with zero,
i.e., $\Cu_+ (A^{\alpha}) \cup  \{0\}$,
 to the 
$\alpha$-invariant elements of the purely positive part of the Cuntz semigroup of $A$
 together with  zero, i.e., 
$\Cu_+ (A)^{\alpha}  \cup  \{0\}$.
One of the most useful and important consequences of these results is that we can apply them to a broader class of C*-algebras, more precisely, the non-unital ones. 
This opens the door to  analyzing structural properties for a wider class of C*-algebras and deepening the understanding of their Cuntz semigroups and associated comparison theory as these computations are often considered to be too complicated for C*-algebras without strict comparison. 
Our work also has points of contact with the work of Hirshberg and Phillips \cite{HP24E} where they define 
the local radius of comparison function on the positive cone of the $K_0$-group. 
  \subsection*{Organization of the paper}
   Section~\ref{SecPreliminaries} is devoted to necessary preliminaries and reviewing certain known theorems on the relative radius of comparison based on  
  the Cuntz semigroup, extended 2-quasitraces, and the Pedersen ideal for easy reference and to ensure that the reader is familiar the essential terminology. Also, some of the material in this section is new
or at least cannot be found in the literature. 
  Section \ref{Cu(FixedPoint)} focuses on both the injectivity and surjectivity of the map
  $\Cu (\iota) \colon \Cu (A^{\alpha}) \to \Cu (A)^{\alpha}$ on the purely positive part of the domain when $A$ is not necessarily unital. We conclude this section with some results on $\Cu (A \rtimes_\alpha G)$ and its purely positive part when $\alpha$ has the (non-unital) weak tracial Rokhlin property. In Section~\ref{ReRc}, we prove our results related to the relative radius of comparison of the crossed product and its relation with the fixed point algebra and the underlying algebra. Finally, in Section~\ref{Example}, we present some examples to show the application of our theorems. 
\section{Preliminaries}\label{SecPreliminaries}
In this section, we will  describe the contents of the paper in more detail. More precisely, we begin with setting up  notation and defining the key ingredients in our results, such as the Cuntz semigroup, functionals and extended quasitraces, and the relative radius of comparison, as the main statements of our main theorems are not trivial
and require a few technical preliminaries.
 Also, we review the main  theorems in comparison theory that we need  for the
 convenience of reader.

\begin{ntn}
We use the following notation.
If $A$ is a \ca,  we write 
$A_{\mathrm{sa}}$
for the set of selfadjoint elements
of $A$
and 
$A_{+}$ for the
set of positive elements of $A$. 
For $a \in  A_{\mathrm{sa}}$, we write $a= a_+ - a_{-}$ as its unique decomposition with
$a_+, a_{-}\geq 0$ and 
$a_+ a_{-}=0$. 
 We denote by $\cM (A)$ the multiplier algebra, the maximal unital extension of 
$A$ in which $A$ is an essential ideal. 
Further, we denote by $A^+$ the minimal unitization of a C*-algebra $A$.
\end{ntn}

We recall the definition of stable finiteness. 
\begin{dfn}
A unital C*-algebra 
$A$ is called \emph{finite} if $1_A$ is finite. 
 A non-unital C*-algebra $A$ is considered finite if its minimal unitization is finite. We say that 
$A$ is \emph{stably finite} if the matrix algebra 
$M_n(A)$ is finite for every $n \in \N$. 
\end{dfn}

\begin{rmk}\label{FiniteProj}
A C*-algebra which is  stably finite in the above sense contains no infinite projections. But the converse is not true in general (see V.2.2.14 of \cite{BlaBo}). It is true for simple C*-algebras (see V.2.3.6 of \cite{BlaBo}). 
\end{rmk}

Part (\ref{Cuntz_def_property_a}) 
of the following definition is originally from~\cite{Cun78}. 
Since we need to use Cuntz subequivalence
with respect to different subalgebras, we include $A$ in the notation to avoid confusion.
\subsection{Cuntz subequivalence}\label{Sec_CuSub}
Let $A$ be a \ca.
\begin{enumerate}
\item\label{Cuntz_def_property_a}
For $a, b \in (A \otimes \cK)_{+}$,
we say that $a$ is {\emph{Cuntz subequivalent to~$b$ in~$A$}},
written $a \precsim_{A} b$,
if there exists a sequence $(w_n)_{n = 1}^{\infty}$ in $A \otimes \cK$
such that
\[
\limi{n} w_n b w_n^* = a.
\]
We say that $a$ and $b$ are {\emph{Cuntz equivalent in~$A$}},
written $a \sim_{A} b$,
if $a \precsim_{A} b$ and $b \precsim_{A} a$.
This relation is an equivalence relation,
and we write $ [a]_A$ for the equivalence class of~$a$.
We define the positively ordered monoid
\[
\Cu (A) = (A \otimes \cK)_{+} / \sim_A,
\]
together with the commutative semigroup operation
$[a]_A +  [b]_A
 = [a \oplus b]_A$ coming from 
  an isomorphism $M_2(\cK) \to \cK$
and the partial order
$[a]_A \leq [b]_A$
if $a \precsim_{A} b$.
We write $0$ for~$[0]_A$.
\item\label{Cuntz_def_property_e}
Let $A$ and $B$ be C*-algebras
and let $\ph \colon A \to B$ be a \hm.
We use the same letter for the induced maps
$M_n (A) \to M_n (B)$
for $n \in \N$ and
$A \otimes \cK \to B \otimes \cK$.
We define
 $\Cu (\ph) \colon \Cu (A) \to \Cu (B)$
by $[a]_A \mapsto [\ph (a)]_B$
for  $a \in (A \otimes \cK)_{+}$.
\end{enumerate}
\begin{rmk}
We use the usual identifications 
\[
A \subset A \otimes M_n  \subset A \otimes \cK.
\]
If $a, b \in A_+$ and $a \precsim_A b$, then there is a sequence $(z_n)_{n=1}^{\infty}$ in $A$ such that 
\[
\lim_{n \to \infty} z_n b z_n^* = a.
\]
\end{rmk}
\begin{dfn}
Let $A$ be a \ca,
let $a \in A_{+}$,
and let $\ep \geq  0$.
Let $f \colon [0, \infty) \to
[0, \infty)$  be the function
$f (t) = \max (0, \, t - \ep) = (t - \ep)_{+}$.
Then, by functional calculus, define $(a - \ep)_{+} = f (a)$.
\end{dfn}
Part (\ref{PhiB.Lem.18.4.11}) of the following lemma is called 
\emph{Rørdam’s lemma} and is Proposition 2.4  of~\cite{Ror92}.
Part  (\ref{KR00.lem.22}) 
 is Lemma~2.2(iii) of~\cite{KR00}
 or Proposition 2.18 of \cite{Thi17}. 
 This shows that 
Cuntz comparison remains invariant when passing to a hereditary sub-C*-algebra.
 Parts (\ref{PhiB.Lem.18.4.8}) and (\ref{PhiB.Lem.18.4.10.a}) 
  are taken from Lemma 2.5 of ~\cite{KR00}. Part (\ref{PhiB.Lem.18.4.15})
  easily follows from (\ref{PhiB.Lem.18.4.8}) and (\ref{PhiB.Lem.18.4.10.a}). 
   For (\ref{PhiB.Lem.18.4.6}), we refer to  the discussion after Definition~2.3 of~\cite{KR00} and Proposition 2.3(ii) of \cite{ERS11}.
  Part (\ref{PhiB.Lem.18.4.15}) is Lemma 1.5 of \cite{Ph14}.
  Part (\ref{LemdivibyNorm}) follows from
  (\ref{PhiB.Lem.18.4.10.a}) and the fact that $\| y a y^* -
 y \left(a - \tfrac{\ep}{\|y\|^2}\right)_+ y^* \|<\ep$.
 Part (\ref{Item.largesub.lem1.7}) is Lemma~1.7 of \cite{Ph14}.
\begin{lem} 
\label{PhiB.Lem.18.4}
 Let $A$ be a \ca .
 \begin{enumerate}
 \item\label{PhiB.Lem.18.4.11}
 Let $a, b \in A_{+}$.
 Then the following relations are equivalent:
 \begin{enumerate}
 \item\label{PhiB.Lem.18.4.11.a}
 $a \precsim_A b$.
 \item\label{PhiB.Lem.18.4.11.b}
 $(a - \ep)_{+} \precsim_A b$ for all $\ep > 0$.
 \item\label{PhiB.Lem.18.4.11.c}
 For every $\ep > 0$ there is $\dt > 0$ such that
 $(a - \ep)_{+} \precsim_A (b - \dt)_{+}$.
 \end{enumerate}
 \item\label{KR00.lem.22}
Let $B$ be a hereditary C*-subalgebra of $A$ and  let $a, b\in B_+$. Then 
$a \precsim_B b$ if and only if $a \precsim_A b$.
 \item\label{PhiB.Lem.18.4.8}
 Let $a \in A_{+}$ and let $\ep_1, \ep_2 > 0$.
 Then
 \[
 \big( ( a - \ep_1)_{+} - \ep_2 \big)_{+}
  = \big( a - ( \ep_1 + \ep_2 ) \big)_{+}.
 \] 
  \item\label{PhiB.Lem.18.4.10}
 Let $a, b \in A_{+}$ and let $\ep>0$.
 If $\| a - b \| < \ep$, then:
 \begin{enumerate}
 \item\label{PhiB.Lem.18.4.10.a}
  $(a - \ep)_{+} \precsim_A b$.
\item\label{Item_9420_LgSb_1_6}
For any $\ld > 0$,
we have $(a - \ld - \ep)_{+} \precsim_A (b - \ld)_{+}$.
\end{enumerate}
 \item\label{PhiB.Lem.18.4.6}
 Let $c \in A$ and let $\lambda \geq 0$.
 Then 
 \[
 (c^* c - \lambda)_{+} \sim_A (c c^* - \lambda)_{+}.
 \] 
 \item\label{PhiB.Lem.18.4.15}
 Let $a, b \in A$ be positive,
 and let $\af, \bt \geq 0$.
 Then
 \[
 \big( (a + b ) - (\af + \bt) \big)_{+}
    \precsim_A (a - \af)_{+} + (b - \bt)_{+}
    \precsim_A (a - \af)_{+} \oplus (b - \bt)_{+}.
 \] 
 \item\label{LemdivibyNorm}
 Let $y \in A \setminus \{0\}$ and let
$a \in A_+$. Then for any $\ep> 0$,
\[
(y a y^* - \ep)_+ 
\precsim_A
 y \left(a - \tfrac{\ep}{\|y\|^2}\right)_+ y^*.
\]
\item\label{Item.largesub.lem1.7}
Let $\ep > 0$. Let $a, b \in A$ satisfy $0 \leq a \leq b$. 
Then 
$(a - \ep)_+ \precsim_A (b - \ep)_+$.
 \end{enumerate}
 \end{lem}

\subsection{The category $\Cu$}
The following definition is from \cite{CowEllIva08CuInv}.  (See Theorem~1 on page 163 of \cite{CowEllIva08CuInv}.)
\begin{dfn}
Let $(P, \leq)$ be a partially ordered set. 
Let $x, y \in P$. We say that $x$ is \emph{compactly contained} in $y$, denoted by $x \ll y$, if for any increasing sequence 
$(y_n)_{n \in \N}$ in $P$ with supremum $y$, we have
$x \leq y_{n_0}$ for some $n_0 \in \N$.
A \emph{$\Cu$-semigroup} is a positively ordered abelian monoid~$S$ with the following properties:
\begin{enumerate}
\item[(O1)]
Every increasing sequence in $S$ has a supremum.
\item[(O2)]
For every element $x$ in $S$, there exists a sequence $(x_n)_{n \in \N}$ in $S$ such that $x_0 \ll x_1 \ll x_2 \ll \cdots$ and such that $x = \sup_n x_n$.  
\item[(O3)] 
For every $x', x, y', y  \in S$ satisfying $x' \ll x$ and $y' \ll y$, we have $x' + y' \ll x + y$.
\item[(O4)]
For every increasing sequences $(x_n)_{n \in \N}$ and $(y_n)_{n \in \N}$ in $S$, we have 
\[
\sup_n (x_n+y_n) = \sup_nx_n + \sup_ny_n.
\]
\end{enumerate} 
\end{dfn}
It is shown in \cite{CowEllIva08CuInv} that the Cuntz semigroup $\Cu(A)$ of any C*-algebra is a $\Cu$-semigroup. 
(See Theorem~1 on page 170 of \cite{CowEllIva08CuInv}.) 
Other axioms, such as O5, O6 (Section 4 of \cite{Rob13}), and O7, have been defined but are not needed here.
The Cuntz semigroup is often complicated to compute in particular for C*-algebras which are not simple or do not have strict comparison. The reader is encouraged to see
\cite{APT18, APT11, GP24} for  many aspects of the Cuntz semigroup
and  \cite{Rob13Cu(X)}  
for the Cuntz semigroup of $C(X)$ when the space $X$ has dimension at most two. 
\begin{dfn}\label{Cu-morphismDef}
Let $S$ and $T$ be $\Cu$-semigroups and let
  $\varphi \colon S\to T$ be a monoid morphism.  We say that
 $\varphi$ is a \emph{$\Cu$-morphism} if
the following conditions hold:
\begin{enumerate}
\item\label{Cu-morphismDef.a}
For every $x, y \in S$ with 
$x \leq y$, we have 
$\varphi (x) \leq \varphi (y)$. 
\item\label{Cu-morphismDef.b}
For every $x, y \in S$ with 
$x \ll y$, we have 
$\varphi (x) \ll \varphi (y)$. 
\item\label{Cu-morphismDef.c}
For every sequence $(x_n)_{n \in \N}$ in $S$ with $x_1 \leq x_2 \leq x_3 \leq \cdots$, we have 
\[
\varphi (\sup_n x_n) = \sup_n \varphi (x_n).
\]  
\end{enumerate}  
\end{dfn}
\begin{dfn}
Let $S$ be a $\Cu$-semigroup and let $T \subseteq S$ be a submonoid of $S$. We say that $T$ is \emph{sub-$\Cu$-semigroup} if 
\begin{enumerate}
\item 
$T$ is a $\Cu$-semigroup with the inherited order.
\item
The inclusion map $T \to S$ is a $\Cu$-morphism. 
\end{enumerate}
\end{dfn}
Recall that a \emph{completely positive map} $\varphi$ from 
a C*-algebra $A$ to $B$ is said to have 
\emph{ order zero} if $\varphi(a)\varphi(b) = 0$ for all 
$a, b \in A_{\mathrm{sa}}$ satisfying $ab = 0$. 

The following remark follows from Proposition~3.4 of \cite{BTZ19}.  It is also Proposition~3.22 and Proposition~3.23 of \cite{Thi17}.
\begin{rmk}\label{Cucpc}
Let $A$, $B$ be C*-algebras and let
  $\varphi \colon A \to B$ 
  be a completely positive contractive  order zero map. 
  Then:
  \begin{enumerate}
  \item\label{Cucpc.1}
  $\varphi$
  naturally induces a map 
  $\Cu(\varphi) \colon \Cu(A) \to \Cu(B)$ which preserves addition, order, the zero element and the suprema of increasing sequences. But,  this map does not preserve  compact containment in general.
  \item\label{Cucpc.2}
If, furthermore, $\varphi$ is a
homomorphism, then $\Cu(\varphi)$ also preserves  compact containment, and therefore, $\Cu(\varphi)$ is a $\Cu$-morphism.
\end{enumerate}
\end{rmk}
\begin{ntn}
Let $\alpha \colon G \to \Aut (A)$ be an action of a discrete group $G$ on a C*-algebra $A$.
\begin{enumerate}
\item
We let $A\rtimes_{\alpha} G$ denote the crossed product.
For $g \in G$, we let $u_g$ denote the element of  $A\rtimes_{\alpha} G$ represented by the
function from $G$ to $A$ which takes the value $1$ at $g$ and $0$ at the other elements of $G$ if $A$ is unital. If $A$ is not unital,  we
extend the action to an action 
$\alpha^{+} \colon G \to \Aut(A^+)$ on the
unitization $A^+$ of $A$ by 
$\alpha_{g}^+(a + \lambda \cdot 1) = 
\alpha_g(a) + \lambda \cdot 1$.
Then write $u_g$ as above.
It is clear that the products $a u_g$, with $a \in A$, are in $A\rtimes_{\alpha} G$. 
\item
The fixed point algebra, denoted by $A^{\alpha}$, is
\[
A^{\alpha} = \big\{ a \in A \colon
\alpha_g (a) = a \mbox{ for all } g \in G \big\}.
\]
\item
We denote by $\Cu (A)^\alpha$ the fixed point semigroup, given by   
 \[
 \Cu (A)^\alpha 
 = \{ x \in \Cu (A) \colon \Cu(\alpha_g) (x) = x \text{ for all } g\in G \}.
\]
 \end{enumerate}
\end{ntn} 
$\Cu (A)^\alpha$ is a submonoid of 
$\Cu(A)$ that is closed under passing to suprema of increasing sequences. It is  not known in general whether 
$\Cu (A)^\alpha$ is a sub-$\Cu$-semigroup of $\Cu(A)$. However, if $\alpha\colon G \to \Aut(A)$ is an action of a finite group $G$ on a  stably finite simple  non-type-I (not necessarily unital)  C*-algebra $A$ with the weak tracial Rokhlin property, then 
$\Cu (A)^\alpha$
a sub-$\Cu$-semigroup of $\Cu(A)$. See Proposition~\ref{CuCrosFix}(\ref{CuCrosFix.2}).
\subsection{Functionals and extended
2-quasitraces}
\begin{dfn}
Let $S$ be a semigroup in the category $\Cu$.
A \emph{functional}
on $S$ is a map 
$\lambda \colon S \to [0, \infty]$ such that: 
\begin{enumerate}
\item 
$\lambda (0)= 0$ and 
$\lambda (x + y) = \lambda (x) + \lambda (y)$ for all 
$x , y \in S$. 
\item
For any $x, y \in S$ with $x \leq y$, we have 
$\lambda (x ) \leq \lambda (y)$.
\item
For any sequence $(x_n)_{n=1}^{\infty}$ in $S$ with 
$x_n \leq x_{n+1}$ for all $n \in \N$,
we have 
\[
\lambda (\sup_n x_n) = \sup_n \lambda (x_n).
\]
\end{enumerate}
\end{dfn}
We use $\cF(S)$ to denote the functionals on $S$.
\begin{dfn}
Let $A$ be a C*-algebra. 
Let $\tau \colon (A \otimes \cK)_+ \to [0, \infty]$ be a function. 
We say that $\tau$ is a \emph{extended 2quasitrace}  if the following conditions hold:
\begin{enumerate}
\item
$\tau(x x^*) = \tau(x^* x)$ for all $x \in A \otimes \cK$. 
\item
$\tau (0) = 0$ and 
$\tau(t\cdot a) = t\cdot \tau(a)$ for all $a \in (A \otimes \cK)_+$ and all $t \in (0, \infty)$. 
\item 
$\tau(a+b) = \tau(a)+\tau(b)$ for all
 $a, b\in (A \otimes \cK)_+$ with  $ab=ba$. 
\end{enumerate}
\end{dfn}
We denote by 
$\EQT_2 (A)$ the set of the lower semicontinuous extended 2-quasitraces on $(A \otimes \cK)_+$.
We denote by 
$\ET (A)$ the set of the lower semicontinuous extended traces on $(A \otimes \cK)_+$.
It follows from  Remark  4.5 of \cite{ERS11} that if $A$ is exact then
$\EQT_2 (A) = \ET(A)$. Also, it follows from Corollary 6.10 of \cite{Thi17} that extended quasitraces are order-preserving, i.e., if $a, b \in (A\otimes \cK)_+$ with 
$a \leq b$, then $\tau (a) \leq \tau (b)$ for $\tau \in \EQT_2 (A)$. 

In the following definition, the notation 
$\Dom_{1/2} (\tau)$ is defined in Definition~2.22 of \cite{BK04}. 
\begin{dfn}\label{Domdef}
Let $A$ be a C*-algebra and 
let $\tau \in \EQT_2 (A)$. Set 
\[
\Dom_{1/2} (\tau)= \{ a \in A \otimes \cK \colon \tau (a a^*) <\infty \}.
\] 
\begin{enumerate}
\item\label{Domdef.a}
We say that $\tau$ is \emph{bounded} if 
$\Dom_{1/2} (\tau)= (A \otimes \cK)_+$.
\item\label{Domdef.b}
We say that $A$ is \emph{densely defined} if 
 $\Dom_{1/2} (\tau)$ is dense in 
 $(A \otimes \cK)_+$. 
\end{enumerate}
\end{dfn}

Recall the following trivial
extended 2-quasitraces,
$\tau_{\infty}, \tau_{0} \colon  (A \otimes \cK)_+ \to [0, \infty]$:
\[
\tau_{\infty} (a) = 
\begin{cases}
             0  & \text{if } a=0 \\
             \infty  & \text{if } \mbox a\neq  0
       \end{cases}
       \qquad
      \mbox{ and }
        \qquad 
       \tau_{0} (a) = 0 \quad
       \mbox{for $a \in (A \otimes \cK)_+$}.
\]

Let $\tau \in \EQT_2 (A)$ be bounded.
Set 
\[
N(\tau) = \sup \{ \tau (a) \colon a \in (A \otimes \cK)_+ \mbox{ and } \| a \| \leq 1 \}.
\] 
Then it follows from Remark~2.27(V) of \cite{BK04} that
\begin{enumerate}
\item 
$N(\tau) < \infty$,
\item
$\tau (a) \leq N(\tau) \cdot \| a \|$ for all $a \in (A \otimes \cK)_+$, and 
\item
$| \tau (a) - \tau(b) | \leq N(\tau)
\cdot
 \| a - b\|$ for all $a, b\in (A \otimes \cK)_+$. 
\end{enumerate}
So $\tau$ extends
uniquely to a uniformly continuous map
 $\tau_{\mathrm{e}} \colon A \otimes \cK \to \mathbb{C}$ 
given by 

\[
\tau_{\mathrm{e}} (a + i b) = \tau (a_+) - \tau (a_{-}) + i \cdot \tau (b_+) - i \cdot \tau (b_{-}),
\]
for $a, b \in (A\otimes \cK)_{\mathrm{sa}}$.
 It is clear that
 $\tau_{\mathrm{e}}$ is linear on every commutative C*-subalgebra of $A \otimes \cK$. By abuse of notation, we usually write $\tau$ instead of $\tau_{\mathrm{e}}$. 
 
 It is known that a simple C*-algebra $A \otimes \cK$
is stably finite if and only if there exists a faithful densely defined 
$\tau$ in $\EQT_2 (A)$ (see Remark~3.7 of \cite{BK04}). 
\begin{dfn}
Let $A$ be a \ca. 
For every $\tau \in  \EQT_2 (A)$, we define a map
$d_{\tau} \colon  (A \otimes \cK)_+ \to  [0, \infty]$ by the following formula:
\[
d_{\tau} (a) = \lim_{n \to \infty} \tau (a^{1/n}).
\]
\end{dfn}
It follows from Proposition 
\ref{Pro.4.2.ESR11} that 
for every $\tau \in \EQT_2 (A)$, and 
every $a, b \in (A \otimes \cK)_+$ with $a \precsim_A b$, we have $d_{\tau} (a) \leq d_\tau (b)$. 

The following statement is  Proposition 4.2 of  \cite{ERS11}; briefly, the association 
$\tau \to  d_\tau$ defines a bijection between $\EQT_2(A)$
and $\cF(\Cu(A))$, extending the work of Blackadar and Handelman in \cite{BH82}.
\begin{prp}\label{Pro.4.2.ESR11}
Let $A$ be a \ca.
\begin{enumerate}
\item
Given $\tau \in \EQT_2 (A)$. The function 
$\lambda_{\tau} ([a])= \sup_n \tau (a^{1/n})$, for $a \in (A \otimes \cK)_+$,
is well defined and gives us 
 a functional $\Cu (A)$, i.e., an element of  $\cF(\Cu(A))$.
\item
Given $\lambda \in \cF(\Cu(A))$. The function 
\[
\tau_{\lambda} (a)= \int_{0}^{\infty} \lambda \left([ (a - t)_+ ] \right) dt,
\quad
\mbox{for $a \in (A \otimes \cK)_+$,}
\]
is a lower semicontinuous extended 2-quasitrace, i.e., an element in $\EQT_2 (A)$.
\item
The map $\theta \colon \EQT_2 (A) \to \cF (\Cu (A))$, defined by $\tau \to \lambda_{\tau}$,
is bijective.
\end{enumerate}
\end{prp}
\subsection{The relative radius of comparison}
In this paper, we focus on 
 the relative radius of comparison, namely
the radius of comparison relative to a particular full element in the Cuntz semigroup of a C*-algebra.  
 This means we are examining how the comparison properties behave in relation to that specific element.  
\begin{dfn}
Let $S$ be a semigroup in the category $\Cu$
and let $x \in S$.
 We denote by $\infty \cdot x$ the supremum $\sup_n n x$. We say that $x$ is \emph{full} if
$\infty \cdot x$ is the largest element of $S$. 
\end{dfn}
 \begin{dfn}
Let  $A$ be  a \ca{} and let $a \in \cM (A)$.
The
ideal of $A$ generated by $a$, denoted by $\mathrm{Span}( AaA)$, is
\[
\mathrm{Span}(AaA)= \Bigg\{\sum_{j=1}^{m} a_j a c_j \colon m\in \N \text{ and } a_j, c_j \in
 A  \text{ for all } j=1, 2, \ldots ,m\Bigg\}.
\]
 We say $a$ is \emph{full} in $A$ if the closed two-sided
ideal of $A$ generated by $a$ is all of $A$ (i.e., $\overline{\mathrm{Span}(AaA)}=A$). 
A corner of $A$ is a C*-subalgebra of the form $pAp$ where $p$ is a projection
in $\cM (A)$. A corner is said to be full if $\overline{\mathrm{Span}(ApA)} = A$.
\end{dfn}
Let $A$ be a C*-algebra and let $a \in (A \otimes \cK)_+$.
Then it is easy to see that $a$ is full in $A \otimes \cK$
if and only if $[ a ]_A$ is full. 
Also, it is clear that if the C*-algebra $A$ is simple, then every non-zero element is full. Our focus in this paper is on simple C*-algebras, so  every non-zero element is full. 

The following definition is contained in Definition~3.3.2 of \cite{BRTW12}. 
\begin{dfn}\label{DefRcCuAlgebraic}
Let $S$ be a semigroup in the category $\Cu$ and let $e \in S$ be  full.
\begin{enumerate}
\item
Let $r\in (0, \infty)$. We say that 
$S$  has \emph{$r$-comparison relative to $e$}, if whenever $x, y \in S$
satisfy 
$\lambda ( x ) + r \cdot \lambda (e)
 \leq
\lambda( y )$
for all $\lambda \in \cF (S)$,
then $x \leq y$. 
\item
 The \emph{radius of comparison of $S$ relative to
$e$}, denoted by
 $\rc (S, e )$, is 
 \[
 \rc (S, e)= \inf \left\{r \in (0, \infty) \colon 
S~ \mbox{has $r$-comparison relative to $e$ } 
\right\}.
 \]
 If there is no such $r$, then $ \rc (S, e)=\infty$.
\end{enumerate}
\end{dfn}

By Proposition~\ref{Pro.4.2.ESR11}, Definition~\ref{DefRcCu} is equivalent the following definition. 
\begin{dfn}
\label{DefRcCu}
Let $A$ be a \ca{} and let $a \in (A \otimes \K)_+$ be  full.
\begin{enumerate}
\item
Let $r\in (0, \infty)$. We say that 
$\Cu (A)$ has \emph{$r$-comparison relative to $[ a ]_A$}, if whenever $x, y \in (A \otimes \K)_+$
satisfy 
$d_\tau ( x ) + r \cdot d_\tau (a)
 \leq
d_\tau ( y )$
for all $\tau \in \EQT_2(A)$,
then $x \precsim_{A} y$. 
\item
The \emph{radius of comparison of $\Cu (A)$ relative to
$[ a ]_A$}, denoted by
 $\rc \big(\Cu (A), \ [ a ]_A \big)$, is 
 \[
 \rc \big(\Cu (A), \ [ a ]_A \big)= \inf \Big\{r \in (0, \infty) \colon 
\Cu (A) ~\mbox{has $r$-comparison relative to $[ a ]_A$} 
\Big\}.
 \]
\end{enumerate}
\end{dfn}

In Definition~3.3.2 of \cite{BRTW12},
the notation used for the radius of comparison of 
$\Cu (A)$ relative to $[ a ]$  is $r_{A, a}$ but we use 
$\rc \big(\Cu (A), \ [ a ]_A \big)$ here or sometimes 
$\rc \big(\Cu (A),\ [ a ] \big)$ when there is no confusion regarding the Cuntz class of $a$. 
 It is also shown in Proposition~3.2.3 of \cite{BRTW12} that for a unital 
C*-algebra $A$ with all quotients stably finite, 
$\rc \big(\Cu (A), [ 1_A ]\big)$ is equal to 
$\operatorname{rc} (A)$, which is the conventional radius of comparison defined in \cite{Tom06}.
We caution the reader that while the conventional radius of
comparison, $\operatorname{rc} (A)$, changes (as the unit does) when transitioning to a matrix algebra over $A$ (see Proposition~6.2(ii) of \cite{Tom06}), the relative radius of comparison does not.
Also, for comparison properties in the category $\Cu$ which are more algebraic, we refer to \cite{BP18}. 
\subsection{The Pedersen ideal and extended 2-quasitraces }
\begin{rmk}\label{PedIdeal}
It is shown in \cite{Ped66} that for every C*-algebra $A$, there is a dense two-sided ideal of $A$, denoted by $\Ped (A)$, 
 which is contained in every dense two-sided ideal of $A$. 
 $\Ped(A)$ is called the \emph{Pedersen ideal of $A$}.  
 Also, it follows from Proposition 5.6.2 of \cite{PedBo} that  for every $a \in A$, we have:
 \begin{enumerate}
 \item\label{PedIdeal.a}
 $a^* a \in  \Ped(A) 
 \Longleftrightarrow
 aa^* \in \Ped(A) 
\Longleftrightarrow 
 a \in \Ped(A).$
 \item\label{PedIdeal.b}
 If $a \in \Ped (A)$, then 
  $\overline{aAa} \subseteq \Ped (A)$. In particular, $C^*(a) \subseteq \Ped (A)$. 
 \end{enumerate}
 See more details about the proof of part (\ref{PedIdeal.a}) in
Proposition 3.4 of \cite{IvKu19}
and part (\ref{PedIdeal.b}) in II.5.2.4. of \cite{BlaBo}. 
 \end{rmk}
The following lemma provides connections between 
$\EQT_2 (A)$,  $\Ped (A)$, and $\Dom_{ \tfrac{1}{2} } (\tau)$. 
\begin{lem}\label{DomLem}
Let $A$ be a C*-algebra and let $\tau \in \EQT_2 (A)$. 
Then: 
\begin{enumerate}
\item\label{DomLem.1}
$\Dom_{ \tfrac{1}{2} } (\tau)= \{ 0\}$ if and only if 
$\tau = \tau_{\infty}$.  
\item\label{DomLem.2}
$\Dom_{ \tfrac{1}{2} } (\tau_0)= A$. 
\item\label{DomLem.3}
For $\tau \in \EQT_2 (A)$, the set $\Dom_{ \tfrac{1}{2} } (\tau)$ is a two-sided $*$-ideal of $A$.
\item\label{DomLem.4}
If $A$ is a simple C*-algebra, then, for every $\tau \in \EQT_2 (A) \setminus 
\{\tau_{\infty}\}$, we have 
\[
\Ped (A \otimes \cK) \subseteq \Dom_{ \tfrac{1}{2} } (\tau).
\] 
\item\label{DomLem.5}
If $A$ is a simple C*-algebra, then
every $\tau \in \EQT_2 (A) \setminus 
\{\tau_{\infty}\}$ is densely defined. 

\end{enumerate}
\end{lem}
\begin{proof}
Parts (\ref{DomLem.1}) and (\ref{DomLem.2}) follow from the definition.   
Part (\ref{DomLem.3}) follows from Proposition 2.24 of \cite{BK04}. 

To prove (\ref{DomLem.4}), let  $\tau \in \EQT_2 (A) \setminus 
\{\tau_{\infty}\}$. It follows from 
(\ref{DomLem.1}) that $ \Dom_{ \tfrac{1}{2} } (\tau) \neq \{ 0 \}$. Since $A$ is simple, it follows from (\ref{DomLem.3}) that $\overline{ \Dom_{ \tfrac{1}{2} } (\tau)} = A \otimes \cK$. Therefore, using the fact that
 $\Ped(A \otimes \cK)$ 
 is the smallest dense ideal of $A \otimes \cK$, we get 
$\Ped (A \otimes \cK) \subseteq \Dom_{ \tfrac{1}{2}} (\tau).$

Part (\ref{DomLem.5}) easily follows from simplicity of $A$ and (\ref{DomLem.1}) and (\ref{DomLem.3}). 

\end{proof}
The following lemma is used in the proof of Theorem~\ref{TrcaeofProj}. 
\begin{lem}\label{BonPed}
Let $A$ be a simple C*-algebra, let $\tau \in \EQT_2(A)  \setminus 
\{\tau_{\infty}\}$, and
 let $a \in \Ped (A \otimes \cK)$.
Then the restriction of $\tau$ to 
the positive part of $\overline{a (A \otimes \cK) a}$ is bounded. 
\end{lem}
\begin{proof}
Using  Remark \ref{PedIdeal}(\ref{PedIdeal.b}) at the first step and using Lemma \ref{DomLem}(\ref{DomLem.4}) at the second step, we get, for all $\tau \in \EQT_2(A)  \setminus 
\{\tau_{\infty}\}$, 
\[
\overline{a (A \otimes \cK) a} \subseteq \Ped (A \otimes \cK)
\subseteq \Dom_{ \tfrac{1}{2} } (\tau).
\] 
Now, the result follows from the definition of 
$\Dom_{ \tfrac{1}{2} } (\tau)$. 
\end{proof}
%
%
\section{$\Cu (\iota) \colon \Cu_+ (A^{\alpha}) \cup \{ 0 \}\to \Cu_+ (A)^{\alpha} \cup \{ 0 \}$ is an isomorphism}
\label{Cu(FixedPoint)}
In the this section, we will try to demystify  Cuntz comparison over $A$ and $A^{\alpha}$
when $A$ is a simple infinite-dimensional (not necessarily unital) C*-algebra and
  $\alpha$ is an action of a finite group $G$ on $A$ with the  weak
tracial Rokhlin property. 
The definition of the weak tracial Rokhlin property for non-unital C*-algebras was given in \cite{FG20}). The definition of the radius of comparison for non-unital C*-algebras was given in \cite{BRTW12}. Our objective is  to determine the validity of the results presented in \cite{AGP19} for non-unital C*-algebras.
This involves some technical difficulties arising from the definition of the weak tracial Rokhlin property  when dealing with non-unital C*-algebras.
\begin{rmk}\label{Nonempty_fixed_point}
Let $A$ be a non-zero \ca{} and 
let $\alpha \colon G \to \Aut (A)$ be an action of 
a finite group $G$ on $A$.
So, there exists $a \in A_+ \setminus \{0\}$.
Since $a \leq \sum_{g \in G } \alpha_g (a)$, it follows that  
$\sum_{g \in G} \alpha_g (a) \in ( A^{\alpha} )_+ \setminus \{0\}$.
In other words, $A^\alpha$ is a non-zero C*-subalgebra of $A$.
\end{rmk}
%
The weak tracial Rokhlin property in the case of non-unital C*-algebras is defined as follows. 
\begin{dfn}\label{W.T.R.P.def.non}
Let $G$ be a finite group,
 let $A$ be a simple
not necessarily unital \ca,
 and let  $\alpha \colon  G \to  \Aut (A)$  be an action of
$G$ on  $A$. We say that  $ \alpha $ has the
 \emph{weak tracial Rokhlin property} if for every $\ep > 0$, 
every finite set $F \subset A$, and every pair of positive 
elements $x, y \in  A$ with $\| x \| = 1$, 
there exist orthogonal positive contractions
$f_g \in G$ in $A$, for $g \in G$, such that, 
 with $f = \sum_{g \in G} f_g$, the following conditions hold:
\begin{enumerate}
\item\label{Def.w.t.r.p.a.non}
$\| a f_g - f_g a \| < \ep$ for all $g \in G $ and all $a \in F$.
\item\label{Def.w.t.r.p.b.non}
$\| \alpha_{g}( f_h ) - f_{gh} \|  < \ep$ for all $g , h\in G$.
\item\label{Def.w.t.r.p.c.non}
$(y^2-  y f y - \ep)_+ \precsim_A x$.
\item\label{Def.w.t.r.p.d.non}
$\|  f x f  \| > 1 - \ep$.
\end{enumerate}
\end{dfn}
\begin{rmk}\label{Rmk2024.08.27}
Adopt the notation from the above definition. 
 It is shown in Proposition 3.2 of \cite{FG20} that
 finite group actions with the weak tracial Rokhlin property are  pointwise outer, and so $A$ is not elementary. Thus, by  Corollary IV.1.2.6 of \cite{BlaBo}, $A$ is not of type I.
Sometimes, we highlight these properties like being  
non-type I or  infinite-dimensional for emphasis.
\end{rmk}
In the following lemma,  we show that the element $f$ in Definition \ref{W.T.R.P.def.non} can be required to belong to the fixed point algebra $A^\alpha$, but  this results in replacing orthogonality of the elements $f_g$ by approximate orthogonality. 
This lemma  generalizes  Lemma 1.11 of \cite{AV21} and 
Lemma~3.3 of \cite{AGP19}, building on the same basic ideas. 
\begin{lem}\label{invariant.contractions}
Let $G$ be a finite group,
 let $A$ be a  simple  \ca,
 and let  $\alpha \colon  G \to  \Aut (A)$  be an action of
$G$ on  $A$.  Then  $\alpha$ has the
 weak tracial Rokhlin property in the sense of \Def{W.T.R.P.def.non} if and only if
 the following condition holds.
 
 For every $\ep > 0$, 
every finite set $F \subset A$, and every pair of positive
elements $x, y  \in  A$ with $\| x \| = 1$, 
there exist  positive contractions
$f_g \in G$ in $A$, for $g \in G$, such that, 
with $f = \sum_{g \in G} f_g$, the following conditions hold:
 \begin{enumerate}
 \item \label{W.T.R.P.66}
$\|f_{g} f_{h}\| <  \ep$ for all $g, h \in G$ with $g \neq h$.
\item \label{W.T.R.P.11}
$\| a f_g - f_g a \| < \ep$ for all $g \in G $ and all $a \in F$.
\item \label{W.T.R.P.22}
$\| \alpha_{g}( f_h ) - f_{gh} \|  < \ep$ for all $g , h\in G$.
\item \label{W.T.R.P.33}
$(y^2 - y f y - \ep)_+ \precsim_A x$.
\item \label{W.T.R.P.44}
 $\|  f x f  \| > 1 - \ep$. 
\item \label{W.T.R.P.55}
 $f \in A^{\alpha} $ and $\| f \| = 1$.
\setcounter{TmpEnumi}{\value{enumi}}
\end{enumerate}
\end{lem}
\begin{proof}
Set $n=\card (G) \in \N$.
To prove the forward implication, 
assume that $\alpha$ has the weak tracial Rokhlin property. 
We may assume $\ep < \frac{1}{2}$.
Let $F_0 \subset A$ be a finite set, and 
let $x, y \in A_{+}$ with $\| x \|= 1$.
Set 
\[
F= F_0 \cup \{y \}, 
\quad
M= \max \{\| a \| \colon a \in F\},
\quad
\mbox{and} 
\quad
\ep'=
\frac{\ep}{6 n \left[ 1+ 2n (2 + 4 \,n) \right]  (M + 1)^2 }.
\]
Applying \Def{W.T.R.P.def.non}  with $F$, $x$, and $y$ as given, and with $\ep'$ in place of $\ep$, 
we get orthogonal positive contractions
$d_g \in G$ in $A$, for $g \in G$, such that, 
 with $d = \sum_{g \in G} d_g$, the following conditions hold:
 \begin{enumerate}
\setcounter{enumi}{\value{TmpEnumi}}
\item \label{W.T.R.P.111}
$\| a d_g - d_g a \| < \ep'$ for all $g \in G $ and all $a \in F$.
\item \label{W.T.R.P.222}
$\| \alpha_{g}( d_h ) - d_{gh} \|  < \ep'$ for all $g , h\in G$.
\item \label{W.T.R.P.333}
$(y^2 - y d y - \ep)_+ \precsim_A x$.
\item \label{W.T.R.P.444}
$\|  d x d  \| > 1 - \ep'$.
\setcounter{TmpEnumi}{\value{enumi}}
\end{enumerate}
Now,  we define 

\[
r=\left\| \sum_{g \in G} \alpha_{g} (d_{1}) \right\|,
\qquad
f_{g} = \frac{1}{r} \alpha_{g} (d_1)
\quad
\mbox{for $g \in G$},
\qquad
\mbox{
and
}
\qquad
f= \sum_{g \in G} f_g.
\]
Clearly, $\| f_g\| \leq 1$ for all $g \in G$, $\| f \|=1$ , and $f \in A^{\alpha}$.
Now we claim that 
\begin{enumerate}
\setcounter{enumi}{\value{TmpEnumi}}
\item \label{W.T.R.P.111'}
$| r-1 | < 2 n \ep' $ and $r > \frac{1}{2}$,
\item\label{W.T.R.P.222'}
$\| f_g - d_g \| < \big(2 + 4 n\big) \ep'$ for all $g \in G$.
\setcounter{TmpEnumi}{\value{enumi}}
\end{enumerate}

To prove (\ref{W.T.R.P.111'}), first, by (\ref{W.T.R.P.444}),  we have 
\[
1 - \ep' < \| d x d\| \leq \|d\|.
\]
Also, by (\ref{W.T.R.P.222}), we have 
\begin{equation*}
\Bigg\|\sum_{g \in G} \alpha_{g} (d_1) - d \Bigg\|=
\Bigg \|\sum_{g \in G} \alpha_{g} (d_1) - \sum_{g \in G} d_g \Bigg \|
\leq \sum_{g \in G} \| \alpha_{g} (d_1) -  d_g\| 
< n\ep'.
\end{equation*}
This relation  implies that
\[
r - \| d \| < n\ep'
\andeqn
\| d \| - r < n \ep'.
\]
Then we have 
\begin{equation}\label{1.3.19.11}
r< \| d \| +n \ep' \leq 1+ n \ep' \leq 1+ 2 \,\card (G) \, \ep',
\end{equation}
and 
\begin{align}\label{1.3.19.22}
r
&> \| d \| - n \ep' 
>
 1-\ep' -n \ep' 
\\ \notag
&= 1-\ep' (1 + n) 
\geq 1 - 2 n\ep'
\\\notag
&>
 1 - \ep > \frac{1}{2}.
\end{align}
Using (\ref{1.3.19.11}) and  (\ref{1.3.19.22}), we get
\[ 
| r - 1 | < 2 n \ep',
\]
and so we have proved (\ref{W.T.R.P.111'}). 
Now, we prove (\ref{W.T.R.P.222'}).  We use (\ref{W.T.R.P.222}) at the second step and use
(\ref{W.T.R.P.111'}) at the fourth step to get
\begin{align*}
\|  f_g - d_g \| 
&=
\Big \| \frac{1}{r}\alpha_{g}( d_1) - d_g  \Big\| 
\\ 
&\leq  
\Big \| \frac{1}{r}\alpha_{g}( d_1) - \frac{1}{r} d_g \Big  \| 
+ \Big  \| \frac{1}{r} d_g -d_g \Big \|
\\ 
&\leq
\frac{1}{r} \ep' + \frac{\| d_g \|}{r} | r-1 |
\\
&<
2 \ep' + 4 n\ep'
= \big(2+4n \big) \ep'.
\end{align*}
This completes the proof of (\ref{W.T.R.P.222'}). 

Now, we prove (\ref{W.T.R.P.66}).  We compute, 
using (\ref{W.T.R.P.222'}) at the third step,
\begin{align*}
\| f_g f_h\| 
&\leq
 \| f_g f_h - f_g d_h\| + \| f_g d_h - d_g d_h\| 
\\
&\leq
\| f_g \| \cdot \| f_h -  d_h\| + \| f_g  - d_g \| \cdot \|d_h\| 
\\
&<
2 \big( 2+4 n \big) \ep'
< \ep.
\end{align*}

To prove  (\ref{W.T.R.P.11}), we compute, 
using (\ref{W.T.R.P.111}) at the second step 
and  (\ref{W.T.R.P.222'}) at the third step,
\begin{align*}
\| a f_g - f_g a \| 
&\leq 
\| a f_g -a  d_g  \| + \| a d_g - d_g a \| + \|  d_g a  - f_g a\|
\\
&<
2 \| a \| \cdot \| d_g - f_g\| + \ep'
\\
&<
\left(2 M (2+4 n ) + 1\right)\ep' < \ep.
\end{align*}

To prove (\ref{W.T.R.P.22}), we get, 
using (\ref{W.T.R.P.222})  
and  (\ref{W.T.R.P.222'}) at the second step,
\begin{align*}
\| \alpha_{g}( f_h) - f_{gh} \| 
&\leq
\| \alpha_{g}( f_h ) - \alpha_{g}( d_h ) \|  
+ 
 \|  \alpha_{g}( d_h ) - d_{gh} \| + \|  d_{gh} - f_{gh}  \| 
\\
&<
(2+4 n ) \ep' + \ep' +  (2+4n ) \ep' < \ep.
\end{align*}

To prove  (\ref{W.T.R.P.33}),  we  estimate, 
using  (\ref{W.T.R.P.222'}) at the fourth step,
\begin{align*}
&
\| (y^2 - y f y ) - (y^2 - y d y - \ep')_+ \| 
\\\notag
&\hspace*{4 em} {\mbox{}} \leq
\| (y^2 - y f y ) - (y^2 - y d y ) \| +
\| (y^2 - y d y ) - (y^2 - y d y - \ep')_+ \|
\\\notag
&\hspace*{ 4 em} {\mbox{}}\leq
\| y \| \cdot \| d - f \| \cdot \| y \| + \ep'
\\\notag
&\hspace*{4 em} {\mbox{}}\leq  
\sum_{g \in G } \| y \| \cdot  \|  f_g -  d_g \| \cdot \| y \| + \ep'
\\\notag
&\hspace*{4 em} {\mbox{}} <
\card (G) M^2 (2+4n ) \ep' + \ep' < \ep.
\end{align*}
Therefore, 
using this and Lemma~\ref{PhiB.Lem.18.4}(\ref{PhiB.Lem.18.4.10.a}) at the first step,  
and 
using (\ref{W.T.R.P.333}) at the second step, we get
\[
(y^2 - y f y - \ep)_+ \precsim_{A} (y^2 - y d y - \ep')_+  \precsim_{A} x.
\]

To prove (\ref{W.T.R.P.44}), we  estimate, 
using (\ref{W.T.R.P.222'}) at the last step,
\begin{align*}
\| d x d - f x f \| 
&\leq \| d x d - d  x f \| + \| d  x f - f x f \| 
\\
& \leq
\| d x \| \cdot \| d -  f \| + \| d   - f \| \cdot \| x f \|
\\
&<
2n (2+4 n ) \ep'.
\end{align*}
Therefore, using (\ref{W.T.R.P.444}) at the second step,
\begin{align*}
\| f x f \| 
&>
 \| d x d \| -2 n (2+4 n ) \ep'
\\
&>
1 - \ep' -2 n (2+4n) \ep' 
\\
&>
1- \ep.
\end{align*}
This completes the proof of the forward implication.

To  prove the reverse implication, assume that $\alpha$ satisfies the conditions (\ref{W.T.R.P.66}) to (\ref{W.T.R.P.55}) in the lemma hold.
We must to show that $\alpha$ has the weak tracial Rokhlin property.
 Let $\ep>0$, 
let $F_0 \subset A$ be a finite set, and let $x, y \in A_+$  with $\| x \| = 1$. 
Set 
\[
F = F_0 \cup \{y\},
\quad
M= \max \{ \| a \| \colon a \in F\},
\quad
\mbox{and}
\quad
\ep' = \frac{\ep}{6 (M+1)^2 n}.
\]
Using Lemma~2.5.12 of \cite{LnBook}
  with $\ep'$ in place of $\ep$ 
and 
with  $n$ as given, we can choose $\dt>0$ such that the following conditions hold:
\begin{enumerate}
\setcounter{enumi}{\value{TmpEnumi}}
\item\label{Item.a_9417_dtlessep.n}
$\dt < \frac{\ep}{2}$.
\item\label{Item.b_9417_Prod.n}
If  $a_1, \ldots, a_n \in A_+$ 
with $\| a_j \|\leq 1$ for $j=1, \ldots, n$ such that 
$\| a_j a_k\| < \dt$ when $j \neq k$, then there are 
$b_1, \ldots, b_n \in A_+$ such that 
$b_j b_k = 0$ when $j \neq k$, $\| b_j \| \leq 1$, and 
$\| a_j - b_j \| < \ep$
for $j=1, \ldots, n$.
\setcounter{TmpEnumi}{\value{enumi}}
\end{enumerate}
Using the condition in the lemma with $\dt$ in place of $\ep$
with the finite set $F$, and $x, y \in A_+$ as given, 
we can find  positive contractions
$f_g \in G$ in $A$ for $g \in G$ such that, 
with $f = \sum_{g \in G} f_g$ , the following conditions hold:
 \begin{enumerate}
 \setcounter{enumi}{\value{TmpEnumi}}
 \item \label{W.T.R.P.66.n}
$\|f_{g} f_{h}\| <  \dt$ for all $g, h \in G$ with $g \neq h$.
\item \label{W.T.R.P.11.n}
$\| a f_g - f_g a \| < \dt$ for all $g \in G $ and all $a \in F$.
\item \label{W.T.R.P.22.n}
$\| \alpha_{g}( f_h ) - f_{gh} \|  < \dt$ for all $g , h\in G$.
\item \label{W.T.R.P.33.n}
$(y^2 - y f y - \dt)_+ \precsim_A x$.
\item \label{W.T.R.P.44.n}
 $\|  f x f  \| > 1 - \dt$. 
\item \label{W.T.R.P.55.n}
 $f \in A^{\alpha} $ and $\| f \| = 1$.
\setcounter{TmpEnumi}{\value{enumi}}
\end{enumerate}
Using (\ref{Item.b_9417_Prod.n}) and  (\ref{W.T.R.P.66.n}), 
we can find $c_g \in A_+$ for $g \in G$ such that:
 \begin{enumerate}
 \setcounter{enumi}{\value{TmpEnumi}}
\item\label{Lin.application1}
$\| c_g \| \leq 1$ for all $g \in G$.
\item\label{Lin.application2}
$c_g c_h = 0$ for $g \neq h$.
\item\label{Lin.application3}
$\| c_g - f_g \| < \ep'$ for all $g \in G$.
\setcounter{TmpEnumi}{\value{enumi}}
\end{enumerate}
Set $c= \sum_{g \in G} c_g$. Now we claim that:
\begin{enumerate}
 \setcounter{enumi}{\value{TmpEnumi}}
\item \label{W.T.R.P.11.nn}
$\| a c_g - c_g a \| < \ep$ for all $g \in G $ and all $a \in F$.
\item \label{W.T.R.P.22.nn}
$\| \alpha_{g}( c_h ) - c_{gh} \|  < \ep$ for all $g , h\in G$.
\item \label{W.T.R.P.33.nn}
$(y^2 - y c y - \ep)_+ \precsim_A x$.
\item \label{W.T.R.P.44.nn}
 $\|  c x c  \| > 1 - \ep$. 
\setcounter{TmpEnumi}{\value{enumi}}
\end{enumerate}

To prove (\ref{W.T.R.P.11.nn}), we estimate, using
(\ref{W.T.R.P.11.n})
and
(\ref{Lin.application3})
at the second step, using
(\ref{Item.a_9417_dtlessep.n}) at the third step, and using (\ref{Item.a_9417_dtlessep.n}) at the last step,
\begin{align*}
\| a c_g - c_g a \| 
&\leq \| a \| \cdot \| c_g - f_g \|
+
\| a f_g - f_g a\|
+
\| f_g - c_g\| \cdot \| a \|
\\
&<
2 M \ep' + \dt <  \ep.
\end{align*}

To prove (\ref{W.T.R.P.22.nn}), we estimate, using
(\ref{W.T.R.P.22.n})
and 
(\ref{Lin.application3})
at the second step and using (\ref{Item.a_9417_dtlessep.n}) at the last step,
\begin{align*}
\| \alpha_{g}( c_h ) - c_{gh} \|
&\leq
\| \alpha_{g}( c_h ) - \alpha_{g}( f_h ) \|
+
\| \alpha_{g}( f_h ) - f_{gh} \|
+
\| f_{gh} - c_{gh} \|
\\
&< 2 \ep' + \dt < \ep.
\end{align*}

To prove (\ref{W.T.R.P.33.nn}), first we estimate, using
(\ref{Lin.application3})
at the second step,
\begin{equation}\label{Eq111.05.22.2019}
\| f - c  \| \leq \sum_{g \in G} \| f_g - c_g\| < n  \ep'. 
\end{equation}
Second, we estimate, using (\ref{Eq111.05.22.2019}) and (\ref{Item.a_9417_dtlessep.n}) at the third step,
\begin{align*}
\| (y^2 - y c y) - (y^2 - y f y - \dt)_+\|
&\leq
\| (y^2 - y c y) - (y^2 - y f y) \|
\\
& \quad +
\| (y^2 - y f y) - (y^2 - y f y - \dt)_+  \|
\\
&\leq
\| y \| \cdot \| f - c \| \cdot \| y \| + \dt
< M^2 n \ep' + \tfrac{\ep}{2}< \ep.
\end{align*}
Now, we use this relation and Lemma~\ref{PhiB.Lem.18.4}(\ref{PhiB.Lem.18.4.10.a}) to get
\begin{equation}
(y^2 - y c y - \ep)_+ 
\precsim_A
(y^2 - y f y - \dt)_+.
\end{equation}
Then, by (\ref{W.T.R.P.33.n}), 
\[
(y^2 - y c y - \ep)_+ \precsim_A x.
\]

To prove (\ref{W.T.R.P.44.nn}), first we estimate, using
(\ref{Eq111.05.22.2019})
at the second step,
\begin{equation}\label{Eq13.05.22.2019}
\| f  x f -  c x c \|
\leq
\| f - c  \| \cdot \| x f \| 
+
\| c x  \| \cdot \| f - c \|
<
2 M n \ep'< \tfrac{\ep}{2}.
\end{equation}
Therefore, 
using (\ref{Eq13.05.22.2019}) at the first step,
using (\ref{W.T.R.P.44.n}) at the second step,
and using (\ref{Item.a_9417_dtlessep.n}) at the third step,
we get
\[
\| c x c\| \geq \| f x f  \| - \tfrac{\ep}{2} > 1 - \dt - \tfrac{\ep}{2} > 1 - \ep. 
\]
This completes the proof.
\end{proof}

The following lemma is Theorem~1.3 of \cite{FG20}. 
\begin{lem}
\label{FG.thm.1.3}
 Let $\alpha \colon  G \to  \Aut(A)$ and $\beta \colon G \to\ \Aut(B)$ be actions of a finite
group $G$ on simple \ca{s} $A$ and $B$. If $\alpha$ has the weak tracial Rokhlin property,
then so also does the action $\alpha \otimes \beta \colon G \to  \Aut(A \otimes_{\mathrm{min}} B)$.
\end{lem}
We need the following approximation lemmas. 
\begin{lem}
\label{L:FCF_2019.05.08}
Let $f, g\colon [0, \infty) \to [0, \infty )$ be  \cfn{s}
such that $f (0) =g (0)= 0$, let $M \in (0, \I)$,
and let $\ep > 0$.
Then there is $\dt > 0$ such that whenever
$A$ is a C*-algebra and
$a, b \in A_{+}$ satisfy $\| a \| \leq M$
and $\| a b  - b a \| < \dt$,
then 
\[
\| f (a) g(b) - g(b) f(a) \| < \ep.
\]
\end{lem}
\begin{proof}
Let $\ep>0$.
Set
$C = \max \big( \| f |_{[0, M]} \|_{\I}, \, \| g |_{[0, M]} \|_{\I} \big)$.
Using Lemma~1.6 of \cite{AV21},
 we can choose $\dt_1>0$ such that
\begin{enumerate}
\item\label{Cond.a.05.15.19}
If
$c, d \in A_{\mathrm{+}}$ satisfy $\| c \|, \, \| d \| \leq \max (C, \, M)$
and $\| c d - d c \| < \dt_1$,
then 
\[
\big\| c f(d) - f(d) c \big\| < \ep.
\]
\setcounter{TmpEnumi}{\value{enumi}}
\end{enumerate}
Using Lemma~1.6 of \cite{AV21}, we can choose $\dt_2>0$ such that
\begin{enumerate}
\setcounter{enumi}{\value{TmpEnumi}}
\item\label{Cond.b.05.15.19}
If
$c, d \in A_{\mathrm{+}}$ satisfy $\| c \|, \, \| d \| \leq M$
and $\| c d - d c \| < \dt_2$,
then 
\[
\big\| c g (d) - g(d) c \big\| < \dt_1.
\]
\setcounter{TmpEnumi}{\value{enumi}}
\end{enumerate}
Set $\dt= \dt_2$. 
Let $a, b \in A_+$ satisfy $\| a \| , \| b \| \leq M$ and $\| a  b  - b a \|< \dt$. 
Using (\ref{Cond.b.05.15.19}), we get
\[
\| a g(b) - g(b) a \| < \dt_1.
\]
Then, by (\ref{Cond.a.05.15.19}),
\[
\| f(a) g(b) - g(b) f(a) \| < \ep.
\]
This completes the proof.
\end{proof}

\begin{lem}
\label{L2:FCF_2019.05.08}
Let $f, g \colon [0, \I) \to [0, \I)$
be \cfn{s} such that $f (0) = g (0) = 0$,
let $\ep > 0$,
and let $M \in (0, \I)$.
Then there is $\dt > 0$ such that whenever $A$ is a \ca,
  and
$c, d \in A_{+}$ satisfy $\| c \|, \| d \| \leq M$
and $\| c d   - d c  \| < \dt$,
then 
\[
\left\| 
\Big(f (d )   g (c)  f (d)  \Big)^{1/2}  
- 
f(d)^{1/2} g(c)^{1/2} f(d)^{1/2} 
\right\| 
< \ep.
\]
\end{lem}
\begin{proof}
Let $\ep>0$ and let $M \in (0, \I)$.
Set
$C = \max \big( \| f |_{[0, M]} \|_{\I},
                 \, \| g |_{[0, M]} \|_{\I} \big)+1$.
 Using Lemma 12.4.5 of \cite{GKPT18}, we can choose $\dt_1>0$ such that:
\begin{enumerate}
\item\label{Cond.a.05.14.19}
If
$a, b \in A_{\mathrm{sa}}$ satisfy $\| a \|, \, \| b \| \leq C^3$
and $\| a - b \| < \dt_1$,
then 
$
\big\| a^{1/2} - b^{1/2} \big\| < \ep.
$
\setcounter{TmpEnumi}{\value{enumi}}
\end{enumerate}
Using \Lem{L:FCF_2019.05.08}, we can choose $\dt_2>0$ such that:
\begin{enumerate}
\setcounter{enumi}{\value{TmpEnumi}}
\item\label{Cond.b.05.14.19}
If
$a, b \in A_{+}$ satisfy $\| a \|, \, \| b \| \leq M$
and $\| a b  - b a \| < \dt_2$,
then 
\[
\left\| f^{1/2}(a) g^{1/2}(b) - g^{1/2}(b) f^{1/2}(a) \right\| < \tfrac{\dt_1}{2 C^2}.
\]
\setcounter{TmpEnumi}{\value{enumi}}
\end{enumerate}
Now set $\dt= \min \{\tfrac{\dt_1}{2C^2},\dt_2\}$. 
Let $c, d \in A_+$ satisfy $\| c \| , \| d \| \leq M$ and $\| c d - d c \|< \dt$. 
Using (\ref{Cond.b.05.14.19}) at the third step, we get
\begin{align}
&
\left\| 
f(d) g(c) f(d) 
- 
\left(f(d)^{1/2} g(c)^{1/2} f(d)^{1/2}\right)^2 
\right\|
\\\notag
& \hspace*{1 em} {\mbox{}}=
\left\|
 f(d) g(c) f(d) 
- 
f(d)^{1/2} g(c)^{1/2} f(d) g(c)^{1/2} f(d)^{1/2} 
\right\|
\\\notag
& \hspace*{1 em} {\mbox{}} \leq
\left\| f(d) g(c)^{1/2} \right\| 
\cdot
\left\| g(c)^{1/2} f (d)^{1/2} -  f (d)^{1/2} g(c)^{1/2} \right\|
\cdot
\left\| f (d)^{1/2} \right\|
\\\notag
& \hspace*{4 em} {\mbox{}} +
\left\| f (d)^{1/2} \right\|
\cdot
\left\| f (d)^{1/2} g(c)^{1/2} - g(c)^{1/2} f (d)^{1/2} \right\|
\\\notag
& \hspace*{18 em} {\mbox{}}
\cdot
\left\| f (d)^{1/2} g(c)^{1/2} f (d)^{1/2}  \right\|
\\\notag
& \hspace*{1 em} {\mbox{}} <
C^2  \tfrac{\dt_1}{2 C^2} 
+ 
C^2 \tfrac{\dt_1}{2 C^2}
=
\dt_1.
\end{align}
Then, by (\ref{Cond.a.05.14.19}),
\[
\left\| 
\Big(f (d )   g (c)  f (d)  \Big)^{1/2}  
- 
f(d)^{1/2} g(c)^{1/2} f(d)^{1/2} 
\right\| 
< \ep.
\]
This completes the proof.
\end{proof}
The importance of the following lemma is that it can apply to any non-unital C*-algebra and any positive element with norm one. 
This lemma plays an important role in the proof of Lemma~\ref{Cu.Surjec}. In the case that $A$ is unital and we take $y= 1_A$, then this gives us
 Lemma~2.6 of \cite{AGP19}, which is restricted to unital C*-algebras and their units. 
\begin{lem}\label{Lem.ANP.Dec.18.non}
Let $A$ be a  \ca and let $a \in A_{+}$ with $\| a \| \leq 1$.
 Let $\ep_1 , \ep_2 > 0$. Then for every $\ep>0$, 
 there exists  $\dt>0$ such that 
 for all $y \in A_+$ with $\| y \|=1$ and 
 for all $c \in A_+$ with $\| c \| \leq 1$, 
  the following condition holds:\\

 If $\| y^2 a y^2 - a \|< \dt$ 
 and 
 $\| y c - cy \| < \dt$,
  then
\[
\Big(   a  - ( \ep +\ep_1 + \ep_2) \Big)_+ 
\precsim_{A} 
\Big( c ( y a y ) c - \tfrac{\ep_1}{2}  \Big)_+
\oplus 
\Big( y^2 - y c y - \tfrac{\ep_2}{2} \Big)_+.
\]
\end{lem}

\begin{proof}
We may assume that $a\neq 0$.  
Let $\ep, \ep_1, \ep_2\in (0, \infty)$. 
Set $\ep'= \min \{\ep, \tfrac{\ep_1}{4}\}$. 
Applying  Lemma~\ref{L2:FCF_2019.05.08},
 choose $\dt > 0$ such that the following conditions hold:
\begin{enumerate}
\setcounter{enumi}{\value{TmpEnumi}}
\item\label{Item.a_9417_dtlessep}
$\dt \leq  \ep'$. 
\item\label{Item.b_9417_Prod}
If
$b, d \in A_{+}$ satisfy $\| b \|, \, \| d \| \leq 1$
and $\| b d   - d b \| < \dt$,
then 
\[
\left\| (d^2  b^2 d^2 )^{1/2}  - d b d \right\| < \ep'.
\]
\setcounter{TmpEnumi}{\value{enumi}}
\end{enumerate}
Let $y \in A_+$ with $\| y \|=1$ and let $c \in A_+$ with $\| c\| \leq 1$. Assume that 
\begin{equation}\label{Eq9.2019.05.07}
\| y^2 a y^2 - a \| < \dt
\qquad
\mbox{ and }
\qquad
\| y c  -  c y  \|< \dt.
\end{equation}
It folllows from (\ref{Item.b_9417_Prod}) and 
the second part of (\ref{Eq9.2019.05.07}) that 
\begin{equation}\label{Eq5.2019.05.9}
\left\| (y^2  c^2 y^2 )^{1/2}  - y c y \right\| < \ep'.
\end{equation}
Set \[
h = y^4 - y^2 c^2 y^2 = y^2 (1 -  c^2 ) y^2.
\] 
Since $\| y \|=1$, $\| c\| \leq 1$, and $c, y \in A_+$, it follows that 
$h \geq 0$ and  $y^4 - h = y^2 c^2 y^2 \geq 0$.
Now we claim that:
\begin{enumerate}
\setcounter{enumi}{\value{TmpEnumi}}
\item\label{Claim1.05.14.2019}
$\left(a - (\ep+\ep_1+\ep_2)\right)_+ 
\precsim_A
\left(y^2 a y^2 - (\ep_1+\ep_2)\right)_+$.
\item\label{Claim2.05.14.2019}
$
\left((y^2 c^2 y^2)^{1/2} a (y^2 c^2 y^2)^{1/2} - \ep_1 \right)
\precsim_A
\left( c y a y c   - \tfrac{\ep_1}{2 } \right)_+.
$
\item\label{Claim3.05.14.2019}
$
(h - \ep_2)_+ 
\precsim_A
\left( y^2  - y c y  -\tfrac{\ep_2}{2} \right)_+.
$
\end{enumerate}

To prove (\ref{Claim1.05.14.2019}), we estimate, using the first part of (\ref{Eq9.2019.05.07}) at the second step and using (\ref{Item.a_9417_dtlessep}) at the last step,
\begin{align*}
\left\| a - \left(y^2 a y^2 - (\ep_1+ \ep_2) \right)_+ \right\| 
&\leq
 \| a - y^2 a y^2\| + \left\| y^2 a y^2 - \left(y^2 a y^2 - (\ep_1+ \ep_2)\right)_+\right\|
\\
&< 
\dt + \ep_1+ \ep_2 < \ep + \ep_1+ \ep_2
\end{align*}
Now, we use this and Lemma~\ref{PhiB.Lem.18.4}(\ref{PhiB.Lem.18.4.10.a}) to get  
\begin{equation*}\label{Eq5.2019.05.07}
\left(a - (\ep+\ep_1+\ep_2)\right)_+ 
\precsim_A
\left(y^2 a y^2 - (\ep_1+\ep_2)\right)_+.
\end{equation*}
This completes the proof of (\ref{Claim1.05.14.2019}).

To prove (\ref{Claim2.05.14.2019}), first we estimate, using  (\ref{Eq5.2019.05.9}) at the third step,
\begin{align}\label{Eq10.2019.05.10}
&
\left\| 
(y c y) a (y c y)  
- 
(y^2 c^2 y^2)^{1/2} a (y^2 c^2 y^2)^{1/2}
 \right \|
\\\notag
&\hspace*{10em} {\mbox{}} \leq
\left\| y c y -  (y^2 c^2 y^2)^{1/2} \right\| 
\cdot 
\left\| a y c y \right\| 
\\\notag
&\hspace*{12em} {\mbox{}} +
\left\| (y^2 c^2 y^2)^{1/2} a  \right\| 
\cdot
\left\| y c y -  (y^2 c^2 y^2)^{1/2} \right\|
\\\notag
&\hspace*{10em} {\mbox{}} <
2 \ep' \leq \tfrac{\ep_1}{2}.
\end{align}
Second, we estimate, using (\ref{Eq10.2019.05.10}) at the second step,
\begin{align}\label{Eq7.2019.05.07}
&
\left\| 
\left((y c y) a (y c y ) - \tfrac{\ep_1}{2} \right)_+
 - 
(y^2 c^2  y^2)^{1/2} a (y^2 c^2 y^2)^{1/2}
\right\|
\\\notag
&\hspace*{2 em} {\mbox{}} \leq
\left\| 
\left((y c y) a (y c y) - \tfrac{\ep_1}{2} \right)_+ 
-  
(y c y) a (y c y) 
\right\|
\\\notag
&\hspace*{3 em} {\mbox{}} +
\left\|
 (y c y) a (y c y)
- 
(y^2 c^2 y^2)^{1/2} a (y^2 c^2 y^2)^{1/2}  
\right\|
\\\notag
&\hspace*{2 em} {\mbox{}} <
\tfrac{\ep_1}{2} + \tfrac{\ep_1}{2}= \ep_1.
\end{align}
Using this relation and Lemma~\ref{PhiB.Lem.18.4}(\ref{PhiB.Lem.18.4.10.a}), we get
\begin{equation}\label{Eq8.2019.05.07}
\left((y^2 c^2 y^2)^{1/2} a (y^2 c^2 y^2)^{1/2} - \ep_1 \right)
\precsim_A
\left((y c y) a (y c y)  - \tfrac{\ep_1}{2} \right)_+.
\end{equation}
Therefore, 
using (\ref{Eq8.2019.05.07}) at the first step, 
using Lemma~\ref{PhiB.Lem.18.4}(\ref{LemdivibyNorm})  at the second step, 
and using $\| y \|=1$ at the last step, we get
\begin{align}\label{Eq12.2019.05.07}
\left((y^2 c^2 y^2)^{1/2} a (y^2 c^2 y^2)^{1/2} - \ep_1 \right)
&\precsim_A
\left((y c y) a (y c y)  - \tfrac{\ep_1}{2} \right)_+
\\\notag
&\precsim_A
y \left( c y a y c   - \tfrac{\ep_1}{2 \| y \|^2} \right)_+ y
\\\notag
&\precsim_A
\left( c y a y c   - \tfrac{\ep_1}{2 } \right)_+.
\end{align}
This completes the proof of (\ref{Claim1.05.14.2019}).

To prove (\ref{Claim3.05.14.2019}), we use 
Lemma~\ref{PhiB.Lem.18.4}(\ref{LemdivibyNorm}) at the second and eighth steps, use the fact that $\|c \|\leq 1$ at the fifth step, use Lemma~\ref{PhiB.Lem.18.4}(\ref{PhiB.Lem.18.4.6}) at the seventh and last steps, and use 
$\|1+ c\|\leq 2$ at the ninth step to get
\begin{align*}
(h - \ep_2)_+ 
&=
(y^4 - y^2 c^2 y^2 - \ep_2)_+
\\
&\precsim_A
y \left(y^2  -  y c^2 y   - \tfrac{\ep_2}{\| y \|^2}\right)_+ y
\\
&\precsim_A
\left(y^2 - y c^2 y - \ep_2 \right)_+
\\
&=
\left(y (1 - c^2) y - \ep_2\right)_+ 
\\
&=
 \left( y (1 - c) (1 + c) y  - \ep_2\right)_+ 
\\
&=
 \left( y (1 - c)^{1/2} (1 + c)^{1/2}  (1 + c)^{1/2} (1 - c)^{1/2} y - \ep_2\right)_+ 
\\
&\sim_A
 \left((1 + c)^{1/2} (1 - c)^{1/2} y^2 (1 - c)^{1/2} (1 + c)^{1/2} - \ep_2\right)_+ 
\\
&\precsim_A
(1 + c)^{1/2} \left( (1 - c)^{1/2} y^2 (1 - c)^{1/2}  -\tfrac{\ep_2}{\| 1+ c\|}\right)_+ (1 + c)^{1/2}
\\
&\precsim_A
\left( (1 - c)^{1/2} y^2 (1 - c)^{1/2}  -\tfrac{\ep_2}{2}\right)_+
\\
&\sim_A
\left( y (1 - c) y  -\tfrac{\ep_2}{2} \right)_+
=
\left( y^2  - y c y  -\tfrac{\ep_2}{2} \right)_+.
\end{align*}
This completes the proof of (\ref{Claim3.05.14.2019}).

Therefore, using (\ref{Claim1.05.14.2019}) at the first step, using Lemma~\ref{PhiB.Lem.18.4}(\ref{PhiB.Lem.18.4.6}) at the second, sixth, and seventh steps, using Lemma~\ref{PhiB.Lem.18.4} (\ref{PhiB.Lem.18.4.15}) at the fourth step, 
 using (\ref{Claim2.05.14.2019}) together with Lemma~\ref{PhiB.Lem.18.4}(\ref{LemdivibyNorm}) and the fact that $0<\| a \|\leq 1$ at the seventh step,
and using (\ref{Claim3.05.14.2019}) at the last  step, we get
\begin{align*}
\left( a - (\ep + \ep_1 + \ep_2) \right)_+
&\precsim_A
 \left(y^2 a y^2 - (\ep_1+\ep_2) \right)_+
\\
&\sim_A
\left( a^{1/2} y^4  a^{1/2} - (\ep_1+\ep_2) \right)_+
\\
&=
\left( a^{1/2} (y^4 - h )  a^{1/2} + a^{1/2} h a^{1/2}  - (\ep_1+\ep_2) \right)_+
\\
&\precsim_A
\left(a^{1/2} (y^4 - h )  a^{1/2} - \ep_1 \right)_+
\oplus
\left(a^{1/2} h a^{1/2} - \ep_2 \right)_+
\\
&=
\left(a^{1/2} (y^2 c^2  y^2) a^{1/2} - \ep_1 \right)_+
\oplus
(a^{1/2} h a^{1/2} - \ep_2 )_+
\\
&\sim_A
\left( (y^2 c^2  y^2)^{1/2} a (y^2 c^2  y^2)^{1/2} - \ep_1\right)_+
\oplus
(h^{1/2} a h^{1/2} - \ep_2 )_+
\\
&\precsim_A
\left( c ( y a y ) c   - \tfrac{\ep_1}{2 } \right)_+
\oplus
( h - \ep_2 )_+
\\
&\precsim_A
\left( c ( y a y) c   - \tfrac{\ep_1}{2 } \right)_+
\oplus
\left(y^2 - y c y - \tfrac{\ep_2}{2} \right)_+.
\end{align*}
This completes the proof.
\end{proof}
The proof of the following lemma is as same as the proof of Lemma 3.6 of \cite{AGP19}, 
as the unit of the C*-algebra is not needed there.  
\begin{lem}\label{approx.orthogonal.action}
Let $A$ be a   simple not necessarily unital  \ca, and let
  $\alpha \colon G \to \Aut (A) $ be an action 
of a finite group $G$ on $A$ which has the weak tracial Rokhlin property.
 Then for every $\ep > 0$ and  every positive element
 $b \in A^{\alpha}$ with $\| b \|=1$, 
 there is a positive element $x \in \overline{b A b}$ 
 with $\| x \|=1$ such that 
 $
\| x\alpha_{g} (x) \| < \ep
$
 for all  $g \in G \setminus \{ 1 \}$.
\end{lem}
The following proposition shows that  Cuntz comparison over a (non-unital) C*-algebra $A$ and  Cuntz comparison over the fixed algebra $A^\alpha$  for the purely positive elements of the fixed point algebra are the same, when $\alpha$ has the (non-unital) weak tracial Rokhlin property. 
This is a generalization of Lemma~3.4 of \cite{AGP19} in the setting of non-unital C*-algebras.
The basic idea of the proofs is the same. But,
 in this case, we encounter additional difficulties arising from the fact that $A$ may not be unital, from the definition of the (non-unital) weak tracial Rokhlin property, and from
 the critical role of Lemma~\ref{Lem.ANP.Dec.18.non}.
\begin{prp}\label{W.T.R.P.inject}
Let $A$ be a simple  infinite dimensional \ca{},
 not necessarily unital,
 and let $\alpha\colon G \to A$
be an action of a finite group $G$ on $A$ with the weak tracial Rokhlin property. Let 
$a, b \in (A^{\alpha})_+$ and assume that $0$ is a limit point of $\spec (b)$. Then
 $a \precsim_A b$  if and only if $a \precsim_{A^\alpha} b$.
\end{prp}

\begin{proof}
We prove only the forward implication, as the backward implication is trivial. 
Assume  
  $\| a \| \leq 1$ and $\| b\|=1$.
Let $\ep >0$. By Lemma~\ref{PhiB.Lem.18.4}(\ref{PhiB.Lem.18.4.11.b}), it suffices to show that $(a -  \ep)_+ \precsim_{A^{\alpha}} b$. 
We may assume $\ep < [2 \, \card (G)]^{-2}$
and $G$ is not a trivial group. 
Since $a \precsim _A b$, it follows from Lemma~\ref{PhiB.Lem.18.4}(\ref{PhiB.Lem.18.4.11.c}) that there is $\dt_1 >0$ such that 
$(a- \ep)_+ \precsim_A (b-\dt_1)_+$. 
Set $a'= (a-\ep)_+$ and $b'= (b - \dt_1)_+$. 
So there is $ w \in  A$
such that
\[
\big\| w b' w^* - a'  \big\|
   < \frac{\ep}{60 \, \card (G)}.
\]
Since $ b', a' \in A^{\alpha} $, it follows that
\begin{equation}\label{Eq.12.28.18.9}
\big\| \alpha_{g} (w)  b'  \alpha_{g}(w^*) - a'  \big\|
   <  \frac{\ep}{60 \, \card (G)}
\end{equation}
 for all $g \in G$.
 Choose a continuous function $h \colon [0, \infty) \to [0, \infty)$ such that:
\begin{enumerate}
\item\label{h.func.a}
$h(t) = 0$ for all $t \in  \{0\} \cup [\dt_1, \infty)$.
\item\label{h.func.b}
$h(t) \neq 0$ for all $t \in (0, \dt_1)$ and $\| h \|= 1$.
\setcounter{TmpEnumi}{\value{enumi}}
\end{enumerate}
Since $0$ is a limit point of $\spec (b)$, it follows that
$\spec (b) \cap (0, \dt_1) \neq \varnothing$.
Then the following properties hold:
\begin{enumerate}
\setcounter{enumi}{\value{TmpEnumi}}
\item\label{h_dt.p.a}
$h (b)\neq 0$ and $\| h (b) \|=1$.
\item\label{h_dt.p.b}
$h (b) \perp b'$.
\item\label{h_dt.p.c}
$h (b) + b' \precsim_{A^\alpha} b$.
\setcounter{TmpEnumi}{\value{enumi}}
\end{enumerate}
Applying \Lem{approx.orthogonal.action} 
with $h (b)$ in place of $b$, we find
a positive element $x \in \overline{h (b) A h (b)}$ 
with $\| x \|=1$ such that 
\[
\big \| x  \alpha_{g} (x) \big \| < \frac{\ep^{2}}{16^2}  
\]
for all  $g \in G \setminus \{ 1 \}$.
Now set 
\[
\ep'= \frac{\ep}{240 \big(\| w \| +1 \big)^2
 \big(\| b '\| +1 \big) \big( \card(G)^2+ 1 \big)}.
 \]
Using \Lem{Lem.ANP.Dec.18.non}
with $\tfrac{\ep}{4}$ in place of $\ep_1$, $\ep_2$, and $\ep$, with $a'$ in place of $a$, and 
with $A^{\alpha}$ in place of $A$,
we can choose $\dt_2>0$  such that:
\begin{enumerate}
\setcounter{enumi}{\value{TmpEnumi}}
\item\label{Eq5.05.14.2019}
$\card (G) \dt_2 \leq \ep'$.
\item\label{Eq5.05.13.2019}
 For all $y \in (A^{\alpha})_+ $ with $\| y \|=1$ and 
 for all $c \in (A^{\alpha})_+$ with $\| c \| \leq 1$
 satisfying $\| y^2 a' y^2 - a' \|< \dt_2$ 
 and 
 $\| c y  -  y c \|< \dt_2$,
  then
\[
\left(   a'  - \tfrac{3 \ep}{4} \right)_+ 
\precsim_{A} 
\Big( c ( y a' y ) c - \tfrac{\ep}{8}  \Big)_+
\oplus 
\Big( y^2 - y c y - \tfrac{\ep}{8} \Big)_+.
\]
\setcounter{TmpEnumi}{\value{enumi}}
\end{enumerate}
Let $( u_\lambda )_{\lambda \in \Lambda}$ be an approximate unit for $A^{\alpha}$. 
Choose $\lambda_0$ such that 
\begin{equation}\label{Eq1.05.13.2019}
\| u_{\lambda_0}^{2} a' u_{\lambda_0}^{2} - a' \| <\dt_2.
\end{equation}
We may assume $\| u_{\lambda_0} \|=1$. 
 Set 
\[
\dt_3= \tfrac{\dt_2}{\card (G)}, 
\qquad
F= \{a', b', w, w^*, u_{\lambda_0} \}, 
\qquad
M = \sup_{z \in F} \| z \|,
\quad
\mbox{ and } 
\]
\[
s= \frac{1}{\|(x^2 - \frac{1}{2})_+ \|} (x^2 - \frac{1}{2})_+. 
\]
Applying  \Lem{invariant.contractions} with $F$ as given,
 with $\dt_3$ in place of $\ep$, with $s$ in place of $x$, and with $u_{\lambda_0}$ in place of $y$,
  we get positive contractions  
$f_g \in A$, for $g \in G$, such that,
 with $f= \sum_{g \in G} f_g $,  the following properties hold:
\begin{enumerate}
\setcounter{enumi}{\value{TmpEnumi}}
\item\label{WTRP4}
$\| f_g f_h \| < \dt_3$ for all $g , h\in G$.
\item\label{WTRP1}
$\| z f_g - f_g z \| <  \dt_3$ for all $g \in G $ and all $z \in F$.
\item\label{WTRP2}
$\| \alpha_{g}( f_g ) - f_{gh} \|  <  \dt_3$ for all $g , h\in G$.
\item\label{WTRP3}
$f \in A^{\alpha}$ and $\| f \| = 1$.
\item\label{WTRP5}
$(u_{\lambda_0}^2 - u_{\lambda_0} f u_{\lambda_0} - \dt_3 )_+ \precsim_{A} s$.
\setcounter{TmpEnumi}{\value{enumi}}
\end{enumerate}
Now, we use (\ref{WTRP5}) and the fact that $\dt_3 <\ep' < \frac{\ep}{16}$ to get
\begin{align*}
(u_{\lambda_0}^2 - u_{\lambda_0} f u_{\lambda_0} - \tfrac{\ep}{16}  )_+ 
&\precsim_{A}
 (u_{\lambda_0}^2 - u_{\lambda_0} f u_{\lambda_0} - \ep' )_+ 
 \\
&\precsim_{A} 
(u_{\lambda_0}^2 - u_{\lambda_0} f u_{\lambda_0} - \dt_3 )_+
\\
&
 \precsim_{A} 
 s 
\sim_A 
 (x^2 - \tfrac{1}{2})_+.
\end{align*}
Considering this relation and applying 
Lemma 3.5 of \cite{AGP19}
with $h(b)$ in place of $b$,
 with $(u_{\lambda_0}^2 - u_{\lambda_0} f u_{\lambda_0} - \tfrac{\ep}{16})_+$
in place of $a$, with $x$ as given,
 and with $\frac{\ep^{2}}{16^2}$ in place of $\dt$, we get 
\[
\Bigg( \big(u_{\lambda_0}^2 - u_{\lambda_0} f u_{\lambda_0} -
\tfrac{\ep}{16}\big)_+ 
- \left(\frac{\ep^{2}}{16^2}\right)^{1/2}  
\Bigg)_{+}
\precsim_{A^{\alpha}} 
h(b). 
\]
Therefore, using Lemma~\ref{PhiB.Lem.18.4}(\ref{PhiB.Lem.18.4.8}),
 \begin{equation}\label{Eq.8855}
 \left(u_{\lambda_0}^2 - u_{\lambda_0} f u_{\lambda_0} - \tfrac{\ep}{8}\right)_+
  \precsim_{A^{\alpha}} 
 h (b).
 \end{equation}
Now define $v= \sum_{g} \alpha_{g} (f_1 w)$. Clearly $v \in A^{\alpha}$. We claim that:
\begin{enumerate}
\setcounter{enumi}{\value{TmpEnumi}}
\item\label{Claim.a.05.13.2019}
$
\left\| v b' v^* - f a' f     \right\| < \frac{\ep}{16}$.
\item\label{Claim.b.05.13.2019}
$
\|  f u_{\lambda_0}  - u_{\lambda_0} f   \|< \dt_2.
$
\end{enumerate}

We prove (\ref{Claim.a.05.13.2019}).  
 First, for all $g, h \in G$ with $g \neq h$, we estimate,
using (\ref{WTRP4}) 
and 
(\ref{WTRP1}) at the third step, and using (\ref{Eq5.05.14.2019}) at the last step,
\begin{align}\label{Eq13.05.14.2019}
\left\| f_g a' f_h\right\| 
&\leq
\left\|  f_g a' f_h -  f_g f_h a' \right\|
+
\left\| f_g f_h a' \right\|
\\\notag
&\leq
\left\| f_g \right\| 
\cdot 
\left\| a' f_h -   f_h a'  \right\|
+
\left\| f_g f_h \right\| \cdot \| a' \|
\\\notag
&<
\dt_3 + \ep' < 2 \ep'.
\end{align}
Second, we estimate,
 using (\ref{WTRP4})  at the fifth step,
\begin{align}\label{Eq.12.28.18.5}
\Bigg\| 
f a' f  
-
\sum_{g \in G} f_{g} a' f_{g}
\Bigg\|
&=
\Bigg\| 
\Bigg(\sum_{g \in G} f_{g} \Bigg) a' \Bigg(\sum_{h \in G} f_{h} \Bigg)
-
\sum_{g \in G} f_{g} a' f_{g}
\Bigg\| 
\\\notag
&=
\Bigg\| 
\sum_{g, h \in G} f_g a'  f_h
-
\sum_{g \in G} f_{g} a' f_{g}
\Bigg\|
\\\notag
&\leq
\Bigg\|
\sum_{g \neq h} f_g a'  f_h
\Bigg\|
\\\notag
&<
 2 \,\card (G)^2 \dt_3
 <
 2 \,\card (G)^2 \ep'
 \leq \frac{\ep}{120}.
\end{align}
Third, for all $g\in G$ we estimate, 
using (\ref{WTRP1}) and (\ref{WTRP2}) at the last step,
\begin{align}
\label{Eq.12.27.18.1}
&
\left\| 
\alpha_{g} (f_1 w) 
- 
\alpha_{g} ( w ) f_g  
\right\| 
\\\notag
& \hspace*{3 em} {\mbox{}} \leq 
\left\| 
\alpha_{g} (f_1 w) 
- 
\alpha_{g} (w f_1) 
\right\|
+ 
\left\| 
\alpha_{g} (w f_1) 
- 
\alpha_{g} (w) f_g
\right\| 
\\\notag
& \hspace*{3 em} {\mbox{}} \leq
\left\| 
f_1 w 
- 
w f_1 
\right \| 
+ 
\| w\| 
\cdot
\left\| 
\alpha_{g}(f_1)  
- 
 f_g  
\right\|
\\\notag
& \hspace*{3 em} {\mbox{}} <
(\| w \| + 1) \dt_3 <(\| w \| + 1) \ep'.
\end{align}
Therefore, for all $g \in G$,
\begin{equation}\label{Eq.12.27.18.2}
\left\|
 \alpha_{g} (w^* f_1 ) 
  -  
 f_{g} \alpha_{g} ( w^* ) 
 \right\|  
 <
 (\| w \| + 1) \dt_3
 <
 (\| w \| + 1) \ep'.
\end{equation}
Fourth, for all $g \in G$ we estimate, using 
 (\ref{Eq.12.27.18.1}), (\ref{Eq.12.27.18.2}), and
 (\ref{WTRP2}) at the second step, and using (\ref{Eq5.05.14.2019}) at the last step,
\begin{align}\label{Eq.12.27.18.3}
&
\left\|
 \alpha_{g} ( w ) f_g  b' f_{g} \alpha_{g} (w^* )
 - 
f_g \alpha_{g} ( w)  b' \alpha_{g} (w^*) f_g
\right\|
\\\notag
&\hspace*{9em} {\mbox{}}
\leq
\left\|
 \alpha_{g} ( w ) f_g
 - 
 \alpha_{g} ( f_1 w ) 
 \right\|
 \cdot 
\left\| 
b'  f_g \alpha_{g} (w^* )
\right\|
\\ \notag
& \hspace*{10.2em} {\mbox{}} +
\left\| 
\alpha_{g} ( f_1 w ) b'
\right\| 
\cdot 
\left\|
 f_g \alpha_{g} (w^* ) 
 - 
 \alpha_{g} (w^* f_1 ) 
 \right\|
\\ \notag
&\hspace*{10.2em} {\mbox{}}+
\left\|
 \alpha_{g} (f_1) 
- 
f_g
\right\| 
\cdot 
\left\|
 \alpha_{g} (w ) b'  \alpha_{g} ( w^* f_1  )
 \right\|
\\ \notag
&\hspace*{10.2em} {\mbox{}} +
\left\| 
f_g \alpha_{g} ( w )  b'  \alpha_{g} ( w^* ) 
\right\| 
\cdot 
\left\|
\alpha_{g} ( f_1  ) 
- 
f_g
\right\|
\\ \notag
&\hspace*{9em} {\mbox{}}
<
2 \| b' \| \cdot \| w\| ( \| w \| +1) \ep' + 2 \| b' \| \cdot \| w \|^{2} \dt_3
\\ \notag
&\hspace*{9em} {\mbox{}}
<
\frac{\ep}{60 \, \card (G)}.
\end{align}
Fifth, for all $g, h \in G$ we estimate, using 
 (\ref{Eq.12.27.18.1}), (\ref{Eq.12.27.18.2}), and (\ref{WTRP1}) at the second step, and 
 using (\ref{Eq5.05.14.2019}) at the last step,
 \begin{align}\label{Eq1.2024.12.24}
 &
 \left\| 
 \alpha_{g} ( f_1 w )  b'  \alpha_{h} (w^*f_1) 
 - 
 \alpha_{g} (w) f_g f_h b'  \alpha_{h} (w^*)
 \right\|
\\\notag
&\hspace*{9em} {\mbox{}} \leq
\left\| 
\alpha_{g}(f_1 w) 
- 
\alpha_g (w) f_g
\right\| 
\cdot
\left\| 
b' \alpha_{h} (w^* f_1)
\right\|
\\\notag
&\hspace*{10.2em} {\mbox{}}+ 
\left\| \alpha_{g} (w) f_g b'
\right\| 
\cdot
\left\|
\alpha_{h}(w^* f_1) 
- 
f_h \alpha_{h}(w^*) 
\right\|
\\\notag
&\hspace*{10.2em} {\mbox{}}+
\left\|
 \alpha_{g} (w) f_g
 \right\| 
\cdot 
\left\| 
b' f_h - f_h b' 
\right\|
\cdot 
\|\alpha_{h}(w^*) \|
\\\notag
&\hspace*{9em} {\mbox{}}<
2 \|b'\| \cdot \|w\|(\|w\|+1)  \ep' + \|w\|^{2} \dt_3
\\\notag
&\hspace*{9em} {\mbox{}}
<
\frac{\ep}{80 \, \card (G)^2}.
\end{align}
Sixth, for all $g, h \in G$,  we estimate, using (\ref{Eq.12.27.18.3}) 
at the second step and  using
 (\ref{WTRP1})  at the last step,
 \begin{align}\label{Eq.1.20.2019}
 &
 \left\|
   \alpha_{g}( w )  f_g^2  b' \alpha_h (w^* )  
 - 
 f_g \alpha_{g}(w)  b'   \alpha_{g}(w^*) f_g
 \right\|
 \\\notag
 &\hspace*{4 em} {\mbox{}}\leq
 \left\| 
 \alpha_{g}( w ) f_g
 \right\| 
 \cdot 
 \left\| 
 f_g  b'- b' f_g 
 \right\| 
 \cdot 
 \| \alpha_h (w^* ) \|
\\\notag
&\hspace*{5 em} {\mbox{}}+
\left\|
\alpha_{g}( w )  f_g  b' f_g \alpha_h (w^* ) 
 -  
 f_g \alpha_{g}(w)  b'   \alpha_{g}(w^*) f_g
 \right\|
\\\notag
&\hspace*{4 em} {\mbox{}}<
\| w \|^{2} \cdot \| b' f_g - f_g b' \| 
+
 \frac{\ep}{60 \, \card (G)}
\\\notag
&\hspace*{4 em} {\mbox{}}<
\frac{5 \ep}{240 \, \card (G)}.
\end{align}
Seventh, we estimate, using (\ref{Eq.12.28.18.5}) at the second step, and 
using (\ref{Eq.12.28.18.9}) and (\ref{Eq.1.20.2019}) at the last step,
\begin{align}\label{Eq15.05.14.2019}
&
\Bigg \|
\sum_{g \in G} \alpha_{g}( w ) \, f_g^2  b' \, \alpha_h (w^* ) 
-
f a' f
\Bigg \| 
\\\notag
&
\hspace*{4 em} {\mbox{}}\leq
\Bigg\| 
\sum_{g \in G} \alpha_{g}( w ) \, f_g^2  b' \, \alpha_h (w^* ) 
 - 
\sum_{g \in G} f_g \alpha_{g}(w)  b'   \alpha_{g}(w^*) f_g
\Bigg\|
\\\notag
&\hspace*{5 em} {\mbox{}}+
\Bigg\| \sum_{g \in G} f_g \alpha_{g}(w) b'  \alpha_{g}(w^*)  f_g - \sum_{g} f_g a' f_g   \Bigg\|
\\\notag
& \hspace*{5 em} {\mbox{}} +
\Bigg \| \sum_{g \in G} f_g a' f_g -  f a' f  \Bigg\|
\\\notag
&\hspace*{4 em} {\mbox{}} \leq
\card (G)
 \left\| 
 \alpha_{g}( w )  f_g^2  b'  \alpha_h (w^* )
  - 
 f_g \alpha_{g}( w) b' \alpha_{g}(w^*) f_g
 \right\|
\\\notag
&\hspace*{5 em} {\mbox{}} +
\card (G) \| \alpha_{g} (w) b' \alpha_{g}(w^*) - a'  \| 
+
\frac{\ep}{120}
\\\notag
&\hspace*{4 em} {\mbox{}} <
\frac{5 \ep}{240 }
+
\frac{ \ep}{60}
+
\frac{ \ep}{120}
=
\frac{ 11 \ep}{240}
\end{align}
Using (\ref{Eq15.05.14.2019}) at the second step, and using
 (\ref{WTRP4}), (\ref{Eq5.05.14.2019}), and (\ref{Eq1.2024.12.24}) at the third  step, we get 
\begin{align*}
&
\left\|
 v b' v^* 
-
 f a' f 
\right\| 
\\
&\hspace*{3 em} {\mbox{}} \leq
\Bigg \| \sum_{g, h \in G} \alpha_{g}(f_1 w )  b' \alpha_h (w^* f_1) - 
  \sum_{g, h \in G} \alpha_{g}( w ) f_g f_h b' \alpha_h (w^* )\Bigg\|
\\
&\hspace*{4 em} {\mbox{}} +
\Bigg \| \sum_{g\neq h}  \alpha_{g}( w ) f_g f_h b' \alpha_h (w^* )\Bigg\|
\\
& \hspace*{4 em} {\mbox{}} +
\Bigg \| 
\sum_{g \in G} \alpha_{g}( w ) \, f_g^{2}  b' \, \alpha_h (w^* )  
- 
 f a' f
\Bigg\|
\\
& \hspace*{3 em} {\mbox{}} \leq
\card (G)^2 \left\| \alpha_{g} ( f_1 w )  b'  \alpha_{h} (w^* f_1 )
 - 
 \alpha_{g} (w) f_g f_h b'  \alpha_{h} (w^*)
 \right\|
\\
&\hspace*{4 em} {\mbox{}} +
 \card (G)^2 \| b' \| \cdot  \| w \|^{2} \| f_g f_h \|
+
\frac{11 \ep}{240} 
\\
&\hspace*{3 em} {\mbox{}}<
\frac{\ep}{80} +\frac{\ep}{240}+  \frac{11 \ep}{240} 
=\frac{\ep}{16}.
\end{align*}
This completes the proof of (\ref{Claim.a.05.13.2019}). 

To prove (\ref{Claim.b.05.13.2019}), we estimate,
using (\ref{WTRP1}) at the third step,
\begin{align*}
\|  f u_{\lambda_0}  - u_{\lambda_0}  f\|
&=
 \left\|
 \sum_{g \in G} f_g u_{\lambda_0}  - \sum_{g \in G}  u_{\lambda_0} f_g  
 \right\|
\\\notag
&\leq
\sum_{g \in G} \|    f_g u_{\lambda_0} - u_{\lambda_0} f_g \|  
\\\notag
&<
\card (G) \dt_3  
=
\dt_2.
\end{align*}
This completes the proof of (\ref{Claim.b.05.13.2019}).

By (\ref{Claim.a.05.13.2019}) and Lemma~\ref{PhiB.Lem.18.4}(\ref{PhiB.Lem.18.4.10.a}), we have 
\begin{equation}\label{Eq.7744}
( f a' f - \tfrac{\ep}{16})_+ \precsim_{A^{\alpha}} v b' v^* \precsim_{A^{\alpha}} b'.
\end{equation}
Using  (\ref{Claim.b.05.13.2019}) at the third step and using (\ref{Eq5.05.14.2019}) at the last step, we get
\begin{align*}
\left\| 
f u_{\lambda_0} a' u_{\lambda_0} f 
- 
( u_{\lambda_0} f a' f u_{\lambda_0} - \tfrac{\ep}{16})_+ 
\right\|
&\leq
\| f u_{\lambda_0}  -  u_{\lambda_0} f \|
\cdot
\| a' u_{\lambda_0} f \|
\\
&\hspace*{1 em} {\mbox{}} +
\| u_{\lambda_0} f  a' \|
\cdot
 \| u_{\lambda_0} f -  f u_{\lambda_0} \| 
 \\
 &\hspace*{1 em} {\mbox{}}+
\left\|
u_{\lambda_0}  f a' f u_{\lambda_0}
- 
\left( u_{\lambda_0} f a' f u_{\lambda_0} - \tfrac{\ep}{16} \right)_+ 
 \right\|
\\
& <
2 \dt_2 + \tfrac{\ep}{16} < \tfrac{\ep}{8}.
\end{align*}
We use this relation and Lemma~\ref{PhiB.Lem.18.4}(\ref{PhiB.Lem.18.4.10.a}) to get 
\begin{equation}\label{Eq7.05.14.219}
\left(f u_{\lambda_0} a' u_{\lambda_0} f - \tfrac{\ep}{8} \right)_+
\precsim_{A^\alpha}
\left( u_{\lambda_0} f a' u_{\lambda_0} f - \tfrac{\ep}{16}\right)_+.
\end{equation}
Therefore, 
using (\ref{Eq7.05.14.219}) at the first step,
using Lemma~\ref{PhiB.Lem.18.4}(\ref{LemdivibyNorm})  at the second step,
and using (\ref{Eq.7744}) at the fourth step, we get
\begin{align}\label{Eq3.2019.05.07}
\left(f u_{\lambda_0} a' u_{\lambda_0} f - \tfrac{\ep}{8} \right)_+
&\precsim_{A^\alpha}
\left( u_{\lambda_0} f a'  f u_{\lambda_0} - \tfrac{\ep}{16}\right)_+
\\\notag
&\precsim_{A^\alpha}
u_{\lambda_0} \left(  f a'  f - \tfrac{\ep}{16\| u_{\lambda_0} \|^2}\right)_+ u_{\lambda_0}
\\\notag
&\precsim_{A^\alpha}
\left(  f a'  f - \tfrac{\ep}{16}\right)_+
\precsim_{A^\alpha}
b'.
\end{align}
Using (\ref{Eq5.05.13.2019}) with $f$ in place of $c$ and with $u_{\lambda_0}$ in place of $y$,
 and considering (\ref{Eq1.05.13.2019}) and
(\ref{Claim.b.05.13.2019}), we get
\begin{equation}\label{Eq2.2019.05.07}
\Big(  a'  - \tfrac{3 \ep}{4} \Big)_+ 
\precsim_{A^\alpha} 
\Big( f u_{\lambda_0} a' u_{\lambda_0} f - \tfrac{\ep }{8} \Big)_+
\oplus 
\Big( u_{\lambda_0}^2 - u_{\lambda_0} f u_{\lambda_0} - \tfrac{\ep }{8} \Big)_+.
\end{equation}
Therefore, using (\ref{Eq2.2019.05.07}) at the second step,  using
  (\ref{Eq3.2019.05.07}) and  (\ref{Eq.8855}) at the third step, 
  and using  (\ref{h_dt.p.c}) at the last step, we get
\begin{align*}
(a -  \ep)_+
&= 
(a' - \tfrac{3\ep}{4})_+
\\
&\precsim_{A^\alpha} 
\Big( f u_{\lambda_0} a' u_{\lambda_0} f - \tfrac{\ep}{8} \Big)_+
\oplus 
\Big( u_{\lambda_0}^2 - u_{\lambda_0} f u_{\lambda_0} - \tfrac{\ep}{8} \Big)_+
\\
&\precsim_{A^{\alpha}}  
b' \oplus h (b) \precsim_{A^{\alpha}} b.
\end{align*}
This relation implies that
$(a -  \ep)_+ \precsim_{A^{\alpha}} b$.
\end{proof}

\begin{dfn}\label{D_9421_Pure}
Let $A$ be a simple \ca.
Following the discussion before Corollary 2.24 of~\cite{APT11}
and Definition~3.1 of~\cite{Ph14},
with slight changes in notation,
we define the following:
\begin{itemize}
\item
$A_{++} = \big\{ a \in A_{+} \colon
  \mbox{there is no projection $p \in M_{\infty} (A)$
         such that $[ a ]_A = [ p ]_A$} \big\}$.
\item
$\Cu_{+} (A)
 = \big\{ [ a ]_A \colon a \in (A \otimes \cK)_{++} \}$.
\end{itemize}
The elements of $A_{++}$ are called {\emph{purely positive}}.
Let $G$ be a discrete group 
and let
$\alpha \colon G \to \Aut (A)$ be an action
of $G$ on~$A$.
We set
\[
\Cu_{+} (A)^{\alpha}
 = \big\{ [ a ]_A \colon a \in (A \otimes \cK)_{++}
\mbox{ and } 
 [ \alpha_{g} (a) ]_A =   [ a ]_A
  \mbox{ for all } g \in G \}.
  \]
\end{dfn}
In the above definition, pure positivity  may not make sense if $A$ is not simple. For example, 
let $A$ be a simple C*-algebra. Take $a, b \in A$ such that $a \in A_{++}$ and $b \notin A_{++}$. By the above definition, 
$(a, b) \in A \oplus A$  is purely positive, which is not what we desire. For the non-simple setting, we refer to Definition~3.1 of \cite{ThiVil24}.

Recall that an element $x$ in a $\Cu$-semigroup $S$ is called \emph{compact} if it is compactly contained it itself, i.e., $x \ll x$. (See Corollary 5 of \cite{CowEllIva08CuInv})

The following lemma provides a connection between pure pure positivity and softness of elements in a  stably finite simple (non-unital) C*-algebra. 
\begin{lem}\label{Pure-Soft-Equivalent}
Let $A$ be a stably finite simple \ca. Let $a \in (A \otimes \cK)_+ \setminus \{0\}$.
Then the following properties are equivalent:
\begin{enumerate}
\item \label{Pure-Soft-Equivalent.1}
$a$ is purely positive 
\item \label{Pure-Soft-Equivalent.2}
$a$ is soft as in Definition~3.1 of \cite{ThiVil24}. 
\item \label{Pure-Soft-Equivalent.3}
 $0$ is not an isolated point in $\spec (a)$. 
 \item\label{Pure-Soft-Equivalent.4}
 $[a]_A$ is not compact in $\Cu (A)$. 
\end{enumerate}
\end{lem}
\begin{proof}
It follows from simplicity of $A$ and 
Proposition~3.6 of \cite{ThiVil24}
that 
(\ref{Pure-Soft-Equivalent.2})
 is equivalent to 
 (\ref{Pure-Soft-Equivalent.3}). 
Stable finiteness of 
$A$ is not required for this part.

Since $A$ is  stably finite simple, it follows by
Remark~\ref{FiniteProj}  that all projections in $A$ are finite. So, the proof of Lemma 3.2 of \cite{Ph14} still works if $A$ is not unital. Therefore, (\ref{Pure-Soft-Equivalent.1}) is equivalent to (\ref{Pure-Soft-Equivalent.3}). 

It follows from the definition of pure positivity and
Proposition~5.3.16 of \cite{APT18}
that 
(\ref{Pure-Soft-Equivalent.4})
 is equivalent to 
 (\ref{Pure-Soft-Equivalent.1}). 
\end{proof}
\begin{ntn}
Let $A$ be a C*-algebra and let $a, b \in A_+$. 
We write $a \approx b$ if there exists $x\in A$ such that
$a=x^*x$ and $b=xx^*$.
\end{ntn}
The following definition is equivalent to
the actual definition of the the Global Glimm Property, Definition~4.12 of \cite{KR02}. We use this version for convenience. 
\begin{dfn}\label{GlobalGlimmProp}
Let $A$ be a C*-algebra. We say that 
$A$ has \emph{the Global Glimm Property} if for every $a \in A_+$, for every $\ep > 0$, and for every $n \in \N$, there are
mutually orthogonal positive elements 
$b_1, b_2,\ldots, b_n$ in $\overline{aAa}$ such that:
\begin{enumerate}
\item
$b_1 \approx b_2 \approx \cdots \approx b_n$.
\item
$(a-\ep)_+ \in \overline{\mathrm{Span} (AbA)}$ 
where $b= b_1+ b_2+ \ldots + b_n$. 
\end{enumerate}
\end{dfn}
The following lemma   is quite similar to
Lemma~2.4 of \cite{Ph14}, except that we here have $b_1 \approx b_2 \approx \cdots \approx b_n$
instead of  $b_1 \sim_A b_2 \sim_A \cdots \sim_A b_n$. 
The proof is exactly the same as the proof of Lemma~2.4 of \cite{Ph14} as, in fact, it is easy to check that 
$b_1 \approx b_2 \approx \cdots \approx b_n$ in the setting of \cite{Ph14}. 
\begin{lem}\label{Lem2.4Ph14}
Let $A$ be a simple C*-algebra which is not of type I, let 
$a \in A_+ \setminus \{0\}$, and let 
$n \in \N$. Then there exist 
mutually orthogonal positive elements
$b_1, b_2,\ldots, b_n \in A_+ \setminus \{0\}$
 such that
 $b_1 \approx b_2 \approx \cdots \approx b_n$
 and 
 $b_1+ b_2+ \ldots + b_n \in \overline{aAa}$.
\end{lem}
By Definition~\ref{GlobalGlimmProp},
Lemma~\ref{Lem2.4Ph14}, and simplicity of the C*-algebra, we have the following lemma. 
\begin{lem}\label{SimplNonTypeGlimm}
Every simple C*-algebra which is not of type I has the Global Glimm Property. 
\end{lem}
We refer the reader to \cite{KR02, ThiVil23} for more details on the Global Glimm Property, and, also,  refer to Section~3 of \cite{ATV24}  for the notion of strongly soft Cuntz classes. 
Since every stably finite simple non-type-I C*-algebra  has the Global Glimm Property (see Lemma~\ref{SimplNonTypeGlimm}), it follows from Lemma~\ref{Pure-Soft-Equivalent} and Corollary~3.4 of \cite{ATV24}  that the strongly soft part of $\Cu(A)$ coincides  with
$\Cu_+ (A) \cup \{0\}$. Now, using this and  Proposition~3.6 of \cite{ATV24}, we get the following corollary. 
\begin{cor}\label{SubCuSemiofPurelyPosi}
Let $A$ be a stably finite simple non-type-I \ca. Then $\Cu_+(A) \cup \{0\}$
is a sub-Cu-semigroup of $\Cu (A)$. 
\end{cor}
In the following theorem, we show that the $\Cu$-morphism 
$\Cu (\iota) \colon \Cu (A^{\alpha}) \to \Cu (A)$
 is injective on the purely positive part of 
 $\Cu(A^\alpha)$. 
This generalizes Lemma 3.11 of \cite{AGP19} to the non-unital case. 
\begin{thm}\label{WC_+.injectivity}
Let $A$ be a  simple non-type-I not necessarily unital \ca{} and let $\alpha\colon G \to A$
be an action of a finite group $G$ on $A$ with the weak tracial Rokhlin property. 
Let $\iota \colon A^{\alpha} \to A$ denote the inclusion map. Then
the map 
\[
\Cu (\iota) \colon \Cu (A^{\alpha}) \to \Cu (A)
\]
 is  a $\Cu$-morphism which is injective on $\Cu_+ (A^\alpha)$.
\end{thm}

\begin{proof}
Let $\id_n \colon G \to \Aut (M_n (\mathbb{C}))$ denote the trivial action,
for $n\in \N$. 
It follows from Remark~\ref{Cucpc} that 
$\Cu (\iota)$ is a $\Cu$-morphism. 
To prove the injectivity of $\Cu (\iota)$ 
on $\Cu_+ (A)$, 
let $a, b \in (A \otimes \cK)_{++}$. 
Suppose 
$\Cu (\iota) \left([ a ]_{A^{\alpha}}\right)
=
\Cu (\iota) \left([ b ]_{A^{\alpha}}\right)$.
Then $[ a ]_{A} = [ b ]_{A}$
and  $a \sim_A b$. 
By Lemma \ref{FG.thm.1.3}, 
 $\id_{n}  \otimes \alpha $ has the weak tracial  Rokhlin property for all $n\in \N$.
 Therefore, by Corollary 4.4 of \cite{FG20}, 
$\dirlim \id_{n}  \otimes \alpha$ has the weak tracial  Rokhlin property. 
Using  Proposition \ref{W.T.R.P.inject}, we get $a \sim_{A^{\alpha}} b$ and 
so $[ a ]_{A^{\alpha}} = [ b ]_{A^{\alpha}}$.
 \end{proof}
 For the definition of  the \emph{Rokhlin property} for a finite group action,  
we refer to Definition~2 of \cite{S15}
for  non-unital C*-algebras, 
Definition~3.1 of \cite{Naw16} for $\sigma$-unital C*-algebras, and Defintion~ 3.1 of \cite{Iz04} for unital C*-algebras (see Definition~1.1 of \cite{PhT1} for the tracial version of this).
 In the above theorem, $\Cu (\iota)$ cannot be injective on the entire $\Cu (A^\alpha)$ (see Example 4.7 of \cite{AGP19} for a counterexample). However, it holds if 
 $\alpha$ has the Rokhlin property, not just  the weak tracial Rokhlin property. 
Namely, if $\alpha\colon G \to A$ has the Rokhlin property, then 
$\Cu (\iota)$ is injective on the whole of $\Cu (A^\alpha)$.
This is shown in Theorem~4.10 of \cite{GS17}.

The following lemma is Lemma~5.2 of \cite{AGP19}. 
\begin{lem}\label{Small_fixed_element}
Let $A$ be a simple \ca{} which is not of type I. 
Let $\alpha \colon G \to \Aut (A)$ be an action of a finite group $G$ on $A$
and let $x \in A_+ \setminus \{0\}$. 
Then there exists $z \in (A^{\alpha})_+ \setminus \{0\}$ such that
$z \precsim_A x$.
\end{lem}
We need the following lemma for the  proofs of Lemma~\ref{Cu.Surjec} 
and  Theorem~\ref{TrcaeofProj}. 
\begin{lem}\label{M.A_lemma_05.17.2019}
Let $A$ be a \ca{} and let $\ep>0$. Let $a \in A_+$ with $\| a \| \leq 1$. 
Then, for every $y \in A$,
\[
(y y^* - y a y^* - \ep)_+
\precsim_A
(y y^* - y a^2 y^* - \ep)_+
\precsim_A
\left(y y^* - y a y^* - \tfrac{\ep}{2}\right)_+.
\]
\end{lem}
\begin{proof}
 Applying the functional calculus to $a$ in $A^+$, the unitization of $A$, 
 and considering $\| a \| \leq 1$, 
we get
\begin{equation*}
1 - a \leq 1 - a^2
\qquad
\mbox{and}
\qquad
1 - a^2 \leq 2 (1 - a).
\end{equation*}
Then, for all $y \in A$,
\begin{equation}\label{Eq1.05.17.2019}
0 \leq y y^* - y ay^*\leq y y^* - y a^2 y^*
\qquad
\mbox{and}
\qquad
0 \leq y y^* - y a^2 y^* \leq 2 (y y^* -y a y^*).
\end{equation}
Therefore, 
using the first part of (\ref{Eq1.05.17.2019}) and Lemma~\ref{PhiB.Lem.18.4}(\ref{Item.largesub.lem1.7})  at the first step,
using the second  part of (\ref{Eq1.05.17.2019}) and \Lem{Eq1.05.17.2019} at the second step,
and using Lemma~\ref{PhiB.Lem.18.4}(\ref{LemdivibyNorm})  at the third step, we get
\begin{align*}
(y y^* - y a y^* - \ep)_+
&\precsim_A
(y y^* - y a^2 y^* - \ep)_+
\\
&\precsim_A
\big(2 (y y^* -y a y^*) - \ep \big)_+
\\
&\precsim_A
2 \Big( (y y^* -y a y^*) - \tfrac{\ep}{2} \Big)_+
\\
&\sim_A
\Big( (y y^* -y a y^*) - \tfrac{\ep}{2} \Big)_+. 
\end{align*}
This relation implies that
\[
(y y^* - y a y^* - \ep)_+
\precsim_A
(y y^* - y a^2 y^* - \ep)_+
\precsim_A
\left(y y^* - y a y^* - \tfrac{\ep}{2}\right)_+.
\]
This completes the proof.
\end{proof}
The following lemma is well known. 
\begin{lem}\label{A.C.05.21.2019}
Let $A$ be a \ca{} and 
let $\alpha \colon G \to \Aut(A)$ be an action of a finite group $G$ on $A$.
Then there exists an approximate unit for $A$ which is in $A^{\alpha}$.
\end{lem} 

\begin{proof}
Let $(u_{\lambda})_{\lambda \in \Lambda}$ be an approximate unit for $A$.
Clearly, 
$\left(\alpha_g (u_{\lambda}) \right)_{\lambda \in \Lambda}$ 
is also an approximate unit  for $A$ for all $g \in G$. 
For all $\lambda \in \Lambda$, define 
\[
v_{\lambda}= \frac{1}{\card (G)} \cdot \sum_{g \in G} \alpha_g (u_{\lambda}).
\]
Clearly, $v_{\lambda} \in A^{\alpha}$ for all $\lambda \in \Lambda$.
Now we claim that $(v_{\lambda})_{\lambda \in \Lambda}$ 
is an approximate unit for $A$.
To prove the claim, 
let $\ep>0$, and let $a \in A$.
Choose $\lambda_0$ such that, for all $g \in G$,
\begin{equation}\label{Eq10.05.23.2019}
\| \alpha_g (u_{\lambda_0}) a - a   \| < \ep.
\end{equation} 
Then, by (\ref{Eq10.05.23.2019}) at the last step,
\begin{align*}
\| v_{\lambda_0} a - a   \|
&= 
\frac{1}{\card (G)} \cdot
\Bigg\|  \sum_{g \in G}  \alpha_g (u_{\lambda_0}) a -  \card (G) \, a \Bigg\|
\\
&\leq 
\frac{1}{\card (G)} \cdot \sum_{g \in G} \| \alpha_g (u_{\lambda_0}) a -   a \| < \ep.
\end{align*}
This completes the proof.
\end{proof} 
 To prepare for the proof of  surjectivity of $\Cu (\iota) \colon \Cu_+ (A^{\alpha}) \cup \{ 0 \}\to \Cu_+ (A)^{\alpha} \cup \{ 0 \}$, we need the following proposition. 
 In the setting of unital C*-algebras, this lemma is 
 Lemma~5.3 of \cite{AGP19}.
 \begin{lem}\label{Cu.Surjec}
Let $A$ be a  simple non-type-I  not necessarily unital
 \ca{} and let 
$\alpha \colon G \to \Aut (A)$ be an action 
of a finite group $G$ on $A$ which has the weak tracial  Rokhlin property.
Let $a \in (A \otimes \cK)_+$ satisfy $a \sim_A \alpha_{g} (a)$ for all $g \in G$
and assume that  $0$ is a limit point of $\spec (a)$.
Then for $\ep>0$  there exist  $\delta>0$ 
and  $b \in \left( (A \otimes \cK \right)^{\alpha})_+$ such that
\[
(a- \ep)_+ \precsim_A  b \precsim_A (a - \delta)_+
\qquad
\mbox{ and }
\qquad
[0, 1] \subseteq \spec (b).
\]
\end{lem}

\begin{proof}
With $\id \colon G \to \Aut(\cK)$ the trivial action, 
since $\alpha$ has the weak tracial Rokhlin property, it follows from Lemma~\ref{FG.thm.1.3} that
$\alpha \otimes \id \colon G \to \Aut (A \otimes \mathcal{K})$ also has the weak tracial Rokhlin property.
By abuse of notation, let us write $\alpha$ instead of  $\alpha \otimes \id$. 
We may assume  
$\ep <1$  and $\| a \|=1$.
For all $g \in G$ we have $a \precsim_A \alpha_{g} (a)$, 
and so 
there exists $\delta_g^{(1)} >0$ such that 
\[
\left(a- \tfrac{\ep}{32}\right)_+ 
\precsim_A 
\left(\alpha_g (a) - \delta_g^{(1)} \right)_+.
\]
We may assume $ \card (G)^2 \delta_g^{(1)} < \frac{\ep}{32}$.
Set $\delta_1 = \min \big\{ \delta_g^{(1)} \colon g \in G \big\}$.
 Then
\begin{equation} \label{Eq111.2019.04.16}
\big(a- \tfrac{\ep}{32}\big)_+ 
\precsim_A 
\big(\alpha_g (a) - \delta_1 \big)_+.
\end{equation}
For all $g \in A$, we have  $\alpha_g (a)  \precsim_A a$,
and 
so there exists $\delta_g^{(2)} >0$ such that 
\[
\left(\alpha_g (a) - \tfrac{\delta_1^3}{\card(G)} \right)_+ 
 \precsim_A 
\left(a - \delta_g^{(2)} \right)_+.
\]
We may assume $ \delta_g^{(2)} + \delta_1^3 < \delta_1^2$.
Set $\delta_2 = \min \big\{ \delta_g^{(2)} \colon g \in G \big\}$. Then
\begin{equation} \label{Eq2.2019.04.16}
\left(\alpha_g (a) -\tfrac{\delta_1^3}{\card(G)} \right)_+  
\precsim_A 
\left(a - \delta_2 \right)_+.
\end{equation}
Using (\ref{Eq111.2019.04.16}) and applying
Lemma~2.7 of \cite{AGP19}
with $ \big(a - \tfrac{\ep}{32}\big)_+$ in place of $a$ and 
with $\alpha_g (a)$ in place of $b$,
 for all $g \in G$,
we can find   $v_g \in A \otimes \cK$ such that 
\begin{equation}\label{Eq.2019.04.16}
\left\| 
v_g \alpha_g (a) v_g^* 
- 
\big(a - \tfrac{\ep}{32}\big)_+ 
\right\| 
< 
\tfrac{\delta_2}{20 \, \card(G)} 
\qquad
\mbox{ and } 
\qquad
\| v_g \| \leq \delta_1^{-1/2}.
\end{equation}
By (\ref{Eq2.2019.04.16}), for all $g \in G$,
we can find  $w_g \in A \otimes \cK$ such that 
\begin{equation}\label{Eq1.2019.04.16}
\left\| 
w_g (a - \delta_2)_+  w_g^* - 
\left(\alpha_g (a) - \tfrac{\delta_1^3}{\card(G)} \right)_+  
 \right\| 
< 
 \tfrac{\delta_2}{20 \, \card(G)}.
\end{equation}
Since $0$ is a limit point of $\spec (a)$, 
we can choose $0<\dt_3< \dt_2$ such that 
$\spec (a) \cap (\dt_3, \dt_2) \neq \varnothing$.
 Choose a continuous function $h \colon [0, \infty) \to [0, \infty)$ such that:
\begin{enumerate}
\item\label{h.func.2.a}
$h(t) = 0$ for all $t \in [0, \dt_3] \cup [\dt_2, \infty)$, 
\item\label{h.func.2.b}
$h(t) \neq 0$ for all $t \in (\dt_3, \dt_2)$ and $\| h \|= 1$.
\setcounter{TmpEnumi}{\value{enumi}}
\end{enumerate}
Then we have the following properties:
\begin{enumerate}
\setcounter{enumi}{\value{TmpEnumi}}
\item\label{h_dt.p.a1}
$h (a)\neq 0$ and $\| h (a) \|=1$.
\item\label{h_dt.p.b2}
$h_(a) \perp (a- \dt_2)_+$.
\item\label{h_dt.p.c3}
$h (a) + (a- \dt_2)_+ 
\precsim_{A} 
(a - \dt_3)_{+}$.
\setcounter{TmpEnumi}{\value{enumi}}
\end{enumerate}
Using \Lem{Lem.ANP.Dec.18.non}
with $A \otimes \cK$ in place of $A$,
with $a$ as given,
with $\tfrac{\ep}{4}$ in place of $\ep$, $\ep_2$, 
and
with $\tfrac{\ep}{2}$ in place of  $\ep_1$, 
we can choose $\dt_4>0$  such that:
\begin{enumerate}
\setcounter{enumi}{\value{TmpEnumi}}
\item\label{Eq555.05.16.2019}
$ \dt_4 <  \dt_3$.
\item\label{Eq5.05.16.2019}
 For all $c, y \in \left(A \otimes \cK\right)_+ $ with $\| c \| \leq 1$  and 
  with $\| y \|=1$, the 
  following condition holds:
 
If  $\| y^2 a y^2 - a \|< \dt_4$ 
 and 
 $\| c y  -  y c \|< \dt_4$,
  then
\[
\left(   a  - \ep  \right)_+ 
\precsim_{A} 
\Big( c ( y a y ) c - \tfrac{\ep}{4}  \Big)_+
\oplus 
\Big( y^2 - y c y - \tfrac{\ep}{8} \Big)_+.
\]
\setcounter{TmpEnumi}{\value{enumi}}
\end{enumerate}
We have 
\begin{equation}\label{Eq4.2019.04.16}
0<\dt_4 <\dt_3 < \delta_2 <\delta_2 + \delta_1^3  < \delta_1^2 <  \delta_1 < \card (G)^2 \delta_1 
< \tfrac{\ep}{32} < 1.
\end{equation}
By \Lem{A.C.05.21.2019}, we can find an approximate unit 
$( u_\lambda )_{\lambda \in \Lambda}$ 
for $A \otimes \cK$ which is in  $(A \otimes \cK)^{\alpha}$. 
Choose $\lambda_0$ such that 
\begin{equation}\label{Eq1.05.16.2019}
\| u_{\lambda_0}^{2} a u_{\lambda_0}^{2} - a \| <\dt_4.
\end{equation}
We may assume $\| u_{\lambda_0} \|=1$. 
Set 
\[
\ep'= \tfrac{\delta_4}{60 M^2 \card (G)^3 }, 
\qquad
M= \big(\max \{ \| z \| \colon z \in F\}\big) +1,
\quad
\mbox{and}
\]
\[
F = \big\{ v_{g}, v_{g}^*, w_{g}, w_{g}^* \colon g \in G\big\} 
\cup 
\big\{ a, (a - \tfrac{\ep}{32})_+, (a - \delta_2 )_+, u_{\lambda_0}   \big\}. 
\]
Applying  \Lem{invariant.contractions} with $F$ as given, with $\ep'$ in place of $\ep$,
$u_{\lambda_0}$ in place of $y$, 
and with $h(a)$ in place of $x$, 
  we get positive contractions
$f_g \in A \otimes \cK$, for $g \in G$, such that,
 with $f= \sum_{g \in G} f_g $,  the following properties hold:
\begin{enumerate}
\setcounter{enumi}{\value{TmpEnumi}}
\item\label{WTRP4.a.16}
$\left\| f_g f_h \right\| < \ep'$ for all $g , h\in G$.
\item\label{WTRP1.b.16}
$\left\| z f_g - f_g  z \right\| < \ep'$ for all $g \in G $ and all $z \in F$.
\item\label{WTRP2.c.16}
$\left\| \alpha_{g} ( f_g ) - f_{gh} \right\|  < \ep'$ for all $g , h\in G$.
\item\label{WTRP3.d.16}
$f \in (A\otimes \cK)^{\alpha}$ and $\| f \| = 1$.
\item\label{WTRP5.e.16}
$\left(u_{\lambda_0}^2 - u_{\lambda_0} f u_{\lambda_0} - \ep' \right)_+ \precsim_{A} h(a)$.
\setcounter{TmpEnumi}{\value{enumi}}
\end{enumerate}
Using \Lem{M.A_lemma_05.17.2019} with $\tfrac{\ep}{8}$ 
in place of $\ep$, 
with $f$ in place of $a$, 
and 
with $u_{\lambda_0}$ in place of $y$, 
we get  
 \begin{equation}\label{Eq41.16.04.2019}
 \left(u_{\lambda_0}^2 - u_{\lambda_0} f^2 u_{\lambda_0} - \tfrac{\ep}{8}\right)_+ 
\precsim_{A} 
\left(u_{\lambda_0}^2 - u_{\lambda_0} f u_{\lambda_0} - \tfrac{\ep}{16} \right)_+.
 \end{equation}
 Therefore, 
 using (\ref{Eq41.16.04.2019}) at the first step,
 using (\ref{Eq4.2019.04.16}) at the second,  
and  using (\ref{WTRP5.e.16}) at the third  step, we get 
\begin{align}\label{Eq3.16.04.2019}
\left(u_{\lambda_0}^2 - u_{\lambda_0} f^2 u_{\lambda_0} - \tfrac{\ep}{8}\right)_+ 
&\precsim_{A} 
\left(u_{\lambda_0}^2 - u_{\lambda_0} f u_{\lambda_0} - \tfrac{\ep}{16}\right)_+
\\\notag
&\precsim_{A}
\left(u_{\lambda_0}^2 - u_{\lambda_0} f u_{\lambda_0} - \ep'  \right)_+
\precsim_A
h(a).
\end{align}
Set 
$x= \sum_{g \in G} f_g v_{g} $ 
and 
$y= \sum_{g \in G} f_g w_{g} $. Then, by the second part of  (\ref{Eq.2019.04.16}) at the last step,
\begin{equation}\label{Eq99.2019.04.16}
\| x \|= \Bigg\|  \sum_{g\in G} f_g v_{g}  \Bigg\| 
\leq
\sum_{g \in G} \| f_g \| \cdot \| v_{g} \| 
\leq  \sum_{g\in G} \| v_{g} \| 
\leq \card(G) \delta_1^{-1/2}.
\end{equation}
Now define 
$d_0= \sum_{g \in G} \alpha_g (f_1 a f_1 )$.  
Clearly
 $d_0 \in \left((A \otimes \cK)^\alpha \right)_+ $. 
We claim that 
\begin{enumerate}
\setcounter{enumi}{\value{TmpEnumi}}
\item\label{claim0.16}
$\left\| 
f^2 u_{\lambda_0} a u_{\lambda_0} f^2 
- 
( u_{\lambda_0} f^2 a f^2 u_{\lambda_0} - \tfrac{\ep}{8})_+ 
\right\| < \tfrac{\ep}{4}$,
\item\label{claim1.16}
$\left\| x (d_0 - \delta_1^2)_+  x^* - f^2 a f^2 \right\| <  \tfrac{\ep}{8}$, 
\item\label{claim2.16}
$\left\| y (a - \delta_2)_+ y^* - d_0 \right\| < \delta_1^2$.
\setcounter{TmpEnumi}{\value{enumi}}
\end{enumerate}

To prove (\ref{claim0.16}),   we estimate, 
using the second part of (\ref{WTRP3.d.16}) 
and 
using (\ref{WTRP1.b.16}) at the third  step,
\begin{align}\label{Eq1.05.20.2019}
\left\| f^2 u_{\lambda_0} -  u_{\lambda_0} f^2  \right\|
&\leq
 \| f \| \cdot \| f u_{\lambda_0} - u_{\lambda_0} f\|
+
\| f u_{\lambda_0} - u_{\lambda_0} f\| \cdot \| f\|
\\\notag
&\leq
2 \sum_{g \in G} \| f_g u_{\lambda_0} - u_{\lambda_0} f_g \|
<2\, \card(G) \ep'
< \dt_4.
\end{align}
Therefore, using  (\ref{Eq1.05.20.2019}) at the second  step
and 
using (\ref{Eq4.2019.04.16}) at the third step, we get
\begin{align*}
&
\left\| 
f^2 u_{\lambda_0} a u_{\lambda_0} f^2 
- 
( u_{\lambda_0} f^2 a f^2 u_{\lambda_0} - \tfrac{\ep}{8})_+ 
\right\|
\\
 &\hspace*{2 em} {\mbox{}} \leq
\| f^2 u_{\lambda_0}  -  u_{\lambda_0} f^2 \|
\cdot
\| a u_{\lambda_0} f^2 \|
+
\| u_{\lambda_0} f^2  a\|
\cdot
 \| u_{\lambda_0} f^2   -  f^2 u_{\lambda_0} \| 
\\\notag
&\hspace*{3 em} {\mbox{}}+
\left\|
u_{\lambda_0}  f^2 a f^2 u_{\lambda_0}
- 
\left( u_{\lambda_0} f^2 a f^2 u_{\lambda_0} -  \tfrac{\ep}{8} \right)_+ 
 \right\|
\\
&\hspace*{2 em} {\mbox{}} <
4\, \card(G) \ep' +\tfrac{\ep}{8}
\leq 
\tfrac{4 \ep }{32} + \tfrac{\ep}{8}=
 \tfrac{\ep}{4}. 
\end{align*}
This completes the proof of (\ref{claim0.16}).

To peove (\ref{claim1.16}), first we estimate, 
using (\ref{WTRP4.a.16}) at the fourth step,
\begin{align}\label{Eq1.01.04.201}
\Bigg\| f^2 - \sum_{g \in G} f_g^2 \Bigg\| 
&=
\Bigg\| \Bigg(\sum_{g \in G} f_g \Bigg)\Bigg( \sum_{h\in G} f_h \Bigg) - \sum_{g \in G} f_g^2 \Bigg\|
\\\notag
&=
\Bigg\| \sum_{g, h \in G}  f_g  f_h - \sum_g f_g^2 \Bigg\|
\\\notag
&\leq
 \sum_{g \neq h} \|  f_g  f_h \|
<
\card(G)^2 \ep'.
\end{align}
Second, we estimate, using (\ref{Eq1.01.04.201}) at the second step,
\begin{align}\label{Eq.1.2016.04.16}
&\Bigg\|
f^2 \left(a - \tfrac{\ep}{32}\right)_+ f^2   
 - 
\sum_{g, h \in G} f_{g}^2  \left(a - \tfrac{\ep}{32}\right)_+  f_{h}^2
\Bigg\|
\\\notag
& \hspace*{8 em} {\mbox{}} \leq
\Bigg\|
f^2 - \sum_{g \in G} f_{g}^2
\Bigg\| 
\cdot 
\Bigg\| \left(a - \tfrac{\ep}{32}\right)_+ f^2  \Bigg\|
 \\\notag
 & \hspace*{9 em} {\mbox{}} +
\Bigg \| \sum_{g \in G} f_{g}^2   \left(a - \tfrac{\ep}{32}\right)_+ \Bigg \| 
\cdot 
\Bigg\| f^2  - \sum_{h \in G} f_{h}^2    \Bigg\|
\\\notag
& \hspace*{8 em} {\mbox{}} < 
\card(G)^2 \ep' + \card(G)^3 \ep'
\leq
2 \,  \card(G)^3 \ep'.
\end{align}
Third, we estimate, using (\ref{WTRP4.a.16}) and (\ref{WTRP1.b.16}) at the second step,
\begin{align}\label{Eq.2.2019.04.16}
&\Bigg\|
\sum_{g, h \in G} f_{g}^2  \left(a - \tfrac{\ep}{32}\right)_+  f_{h}^2
-
\sum_{h \in G} f_{h}^2  \left(a - \tfrac{\ep}{32}\right)_+  f_{h}^2
\Bigg\|
\\\notag
& \hspace*{3 em} {\mbox{}} \leq
\sum_{g, h\in G  } 
\Big( 
\left\| f_{g}^2 \right\| 
\cdot 
\left\| \left(a - \tfrac{\ep}{32}\right)_+  f_{h} - f_{h} \left(a - \tfrac{\ep}{32}\right)_+  \right\|
 \cdot
  \big\| f_h \big\| \Big)
\\\notag  
& \hspace*{4 em} {\mbox{}} +
\sum_{g\neq h}
\left(
\left\| f_g \right\| 
\cdot 
\left\| f_{g} f_{h} \right\| 
\cdot
 \left\| \left(a - \tfrac{\ep}{32}\right)_+  f_{h} \right\| 
\right)
\\\notag
&\hspace*{4 em} {\mbox{}} +
\sum_{h \in G}
\left(
\| f_h^2 \| 
\cdot  
\left\| 
f_h \left(a - \tfrac{\ep}{32}\right)_+ 
- 
  \left(a - \tfrac{\ep}{32}\right)_+  f_{h} 
\right\|
\cdot 
\| f_h \|
\right) 
\\\notag
&\hspace*{3 em} {\mbox{}} <
\card(G)^2 \ep'
+
\card(G)^2 \ep'
+
\card(G) \ep'
\leq 3 \, \card(G)^2 \ep'.
\end{align} 
Fourth, we estimate, 
using (\ref{Eq.1.2016.04.16})  and (\ref{Eq.2.2019.04.16})  at the second  step,
\begin{align} \label{2019.04.16.1}
&
\Bigg\|
f^2 \big(a - \tfrac{\ep}{32}\big)_+ f^2   
 - 
\sum_{h \in G} f_{h}^2 \big(a - \tfrac{\ep}{32}\big)_+ f_{h}^2 
\Bigg\|
\\\notag
&\hspace*{7em} {\mbox{}} \leq
\Bigg\|
f^2 \big(a - \tfrac{\ep}{32}\big)_+ f^2   
 - 
\sum_{g, h \in G} f_{g}^2  \big(a - \tfrac{\ep}{32}\big)_+  f_{h}^2
\Bigg\|
\\\notag
&\hspace*{8 em} {\mbox{}} +
\Bigg\|
\sum_{g, h \in G} f_{g}^2  \big(a - \tfrac{\ep}{32}\big)_+  f_{h}^2
-
\sum_{h \in G} f_{h}^2  \big(a - \tfrac{\ep}{32}\big)_+  f_{h}^2
\Bigg\|
\\\notag
&\hspace*{7em} {\mbox{}}<
2 \,  \card(G)^3 \ep'
+
3 \, \card(G)^2 \ep'
\leq
5 \card(G)^3 \ep'.
\end{align}
Fifth, for all $h \in G$ we estimate,
using (\ref{WTRP2.c.16}) at the third step, 
\begin{align}\label{2019.04.16}
\left\|
 \alpha_h (f_1 a f_1) - f_h \alpha_h (a) f_h
\right\|
&
\leq
\|\alpha_h (f_1) - f_h  \| \cdot \| \alpha_h (a f_1) \|
\\\notag
& \qquad +
\| f_h \alpha_h (a ) \| \cdot \|\alpha_h (f_1) - f_h  \|
\\\notag
&
< 2  \ep'.
\end{align}
Sixth, for all $g, h, t \in G$ with $g \neq h$ or $h \neq t$, we estimate,
using  (\ref{WTRP4.a.16}), (\ref{WTRP1.b.16}), and (\ref{2019.04.16}) at the third step, 
\begin{align}\label{2019.04.16.2}
&\| f_g v_{g} \alpha_h (f_1 a f_1) v_{t}^* f_t \| 
\\\notag
&\hspace*{7em} {\mbox{}} \leq 
\| f_g v_{g} \|
 \cdot
  \| \alpha_h (f_1 a f_1) - f_h \alpha_h (a) f_h  \| 
 \cdot 
 \| v_{t}^* f_t \|
\\\notag
& \hspace*{8 em} {\mbox{}}+
\| f_g \| 
\cdot 
\| v_{g} f_h - f_h  v_{g}  \| 
\cdot 
\| \alpha_h (a) f_h v_{t}^* f_t \|
\\\notag
&\hspace*{8 em} {\mbox{}} +
\| f_g f_h v_{g} \alpha_h (a)  \| 
\cdot 
\| f_h v_{t}^* -  v_{t}^* f_h \| 
\cdot
 \| f_t \|
\\\notag
&\hspace*{8 em} {\mbox{}} +
\| f_g f_h \|  
\cdot
 \| v_{g}  \alpha_h (a) v_{t}^* \| 
\cdot 
\| f_h f_t \|
\\\notag
&\hspace*{7em} {\mbox{}} <
2 M^2 \ep' + M \ep' + M \ep' + M^2 \ep'
< 5 M^2 \ep'.
\end{align}
Seventh, for all $h \in G$, we estimate, 
 using  (\ref{WTRP4.a.16}), (\ref{WTRP1.b.16}), and (\ref{2019.04.16}) at the second step,
\begin{align}\label{Eq20.2019.04.16.15}
&\|
 v_{h} \alpha_h (f_1 a f_1) v_{h}^*
- 
f_h v_{h} \alpha_h (a) v_{h}^* f_h
\|
\\\notag
&\hspace*{7em} {\mbox{}}
\leq 
\|  v_{h} \| 
\cdot  
\| \alpha_h (f_1 a f_1)  -    f_h \alpha_h ( a )  f_h \|
\cdot 
\| v_{h}^*\| 
\\\notag
&\hspace*{8 em} {\mbox{}}
+
\|  v_{h} f_h - f_h  v_{h} \|
\cdot 
\|  \alpha_h ( a ) f_h v_{h}^*  \|
\\\notag
&\hspace*{8 em} {\mbox{}}
+
\| f_h  v_{h}  \alpha_h ( a ) \| 
\cdot
 \| f_h v_{h}^*  - v_{h}^* f_h \|
\\\notag
&\hspace*{7em} {\mbox{}}
<
2 M^2 \ep' + 2 M \ep'
< 4 M^2 \ep'.
\end{align}
Eighth, we estimate, using
the first part of (\ref{Eq.2019.04.16}), 
(\ref{2019.04.16.1}), 
and (\ref{Eq20.2019.04.16.15}) 
  at the second step and using (\ref{Eq4.2019.04.16}) at the third step,
\begin{align}\label{Eq808.2019.04.16}
&\left\| 
\sum_{h \in G} f_h v_{h}    \alpha_h (f_1 a f_1  )   v_{h}^* f_h 
-
f^2 a f^2 
\right\|
\\\notag
& \hspace*{2 em} \leq
\left\| 
\sum_{h \in G} f_h v_{h}    \alpha_h (f_1 a f_1  )   v_{h}^* f_h 
-
\sum_{h \in G} f_h^2 v_{h}   \alpha_h ( a )   v_{h}^*  f_h^2
\right\|
\\\notag
&\hspace*{3 em} +
\left\|  \sum_{h \in G} f_h^2 v_{h}   \alpha_h ( a )   v_{h}^*  f_h^2 
- 
\sum_{h \in G} f_h^2  (a - \tfrac{\ep}{32})_+ f_h^2
\right\|
\\\notag
&\hspace*{3 em} +
\left\| \sum_{h \in G} f_h^2  (a - \tfrac{\ep}{32})_+ f_h^2  
 - 
  f^2  (a - \tfrac{\ep}{32})_+  f^2 \right\|
\\\notag
& \hspace*{3 em} +
\left\|    f^2  \left(a - \tfrac{\ep}{32}\right)_+  f^2  -  f^2 a f^2 \right\|
\\\notag
&\hspace*{2 em} < 
 4 \, \card (G) M^2 \ep'  
+ \tfrac{\card (G) \delta_2}{20 \, \card(G)} 
 +
5 \card(G)^3 \ep' + \tfrac{\ep}{32}
<
\tfrac{\delta_2}{5} +\tfrac{\ep}{32}. 
\end{align}
Set $S= \big\{ (g, h, t) \in G^3 \colon g \neq h \mbox{ or } h \neq t \big\}$.
Ninth, we estimate, using
 (\ref{2019.04.16.2})  and  (\ref{Eq808.2019.04.16})
  at the third step,  and
 using (\ref{Eq4.2019.04.16}) at the fifth step,
\begin{align}\label{Eq43.2019.04.16}
&
\left\| x d_0  x^* - f^2 a f^2 \right\|
\\\notag
&\hspace*{2 em} =
\left\| 
\Bigg(\sum_{g \in G} f_g v_{g}  \Bigg) \Bigg(\sum_{h \in G} \alpha_h (f_1 a f_1 ) \Bigg) 
\Bigg(\sum_{t \in G}  v_{t}^* f_t \Bigg) -  f^2 a f^2 
\right\| 
\\\notag
&\hspace*{2 em} \leq
\left\| 
\sum_{(g, h, t) \in S} f_g v_{g}    \alpha_h (f_1 a f_1  )   v_{t}^* f_t 
\right\|
\\\notag
&\hspace*{3 em} +
\left\| 
\sum_{h \in G} f_h v_{h}    \alpha_h (f_1 a f_1  )   v_{h}^* f_h 
-
 f^2 a f^2
\right\|
\\\notag
&\hspace*{2 em} < 
5\, \card (G)^3 M^2 \ep' + \tfrac{\delta_2}{5} +\tfrac{\ep}{32} 
\\\notag
& \hspace*{2 em} <
\delta_2 +\tfrac{\ep}{32} 
< \tfrac{\ep}{32} + \tfrac{\ep}{32} = \tfrac{\ep}{16}.
\end{align}
Tenth, we estimate, using the second part of (\ref{Eq.2019.04.16}) at the second step
and 
using (\ref{Eq4.2019.04.16}) at the last step,
\begin{align}\label{Eq1.16.04.2019}
\| x (d_0 - \delta_1^2)_+  x^* - x d_0 x^*\| 
&\leq 
\| x \| \cdot \| (d_0 - \delta_1^2)_+ - d_0  \| \cdot \| x^* \|
\\\notag
&< 
(\card (G) \dt_1^{-1/2} ) \dt_1 ^2 ( \card (G) \dt_1^{-1/2})
\\\notag
&=
 \card(G)^2 \dt_1  < \tfrac{\ep}{32}.
\end{align}
Therefore, using (\ref{Eq43.2019.04.16}) and  (\ref{Eq1.16.04.2019}) at the second step,
\begin{align*}
\| x (d_0 - \delta_1^2)_+  x^* - f^2 a f^2 \|
&\leq
\| x (d_0 - \delta_1^2)_+  x^* - x d_0 x^*\|
+
\| x d_0 x^* - f^2 a f^2 \| 
\\\notag
&< \tfrac{\ep}{32} + \tfrac{\ep}{16} < \tfrac{\ep}{8}.
\end{align*}
This completes the proof  of (\ref{claim1.16}).

To prove (\ref{claim2.16}), first  for all $g, h \in G$ with $g \neq h$ we estimate,
using  (\ref{WTRP4.a.16}) and (\ref{WTRP1.b.16}) at the second step,
\begin{align}\label{Eq102.2019.04.16}
&\|  
f_g w_{g}  \left(a - \delta_2\right)_+ w_{h}^* f_h 
\|
\\\notag
& \hspace*{6 em} {\mbox{}} \leq 
\| f_g w_{g}  \left(a - \delta_2\right)_+ \| 
\cdot 
\| w_{h}^* f_h - f_h w_{h}^* \|
\\\notag
&\hspace*{7 em} {\mbox{}} +
\| f_g w_{g} \| 
\cdot 
\| \left(a - \delta_2\right)_+ f_h - f_h  \left(a - \delta_2\right)_+ \| 
\cdot
 \| w_{h}^* \|
\\\notag
&\hspace*{7 em} {\mbox{}} +
\|  f_g \|
 \cdot 
\| w_{g} f_h - f_h w_{g}  \| 
\cdot 
\|  \left(a - \delta_2\right)_+  w_{h}^* \|
\\\notag
& \hspace*{7 em} {\mbox{}} +
\| f_g f_h\| 
\cdot
\left\| w_{g} \left(a - \delta_2\right)_+  w_{h}^* \right\|
\\\notag
& \hspace*{6 em} {\mbox{}} < 
M \ep' + M^2 \ep' + M \ep' + M^2 \ep'  < 4 M^2 \ep'.
\end{align}
Second, we estimate, using (\ref{Eq1.2019.04.16})  and (\ref{2019.04.16}) at the second step,
\begin{align}\label{Eq26.2019.04.15}
&\left\|
f_g  w_{g} \left(a - \delta_2\right)_+ w_{g}^* f_g 
-
 \alpha_g (f_1 a f_1)
\right\|
\\\notag
&  \hspace*{5 em} {\mbox{}} \leq
\| f_g \| 
\cdot 
\left\|
 w_{g}  \left(a - \delta_2\right)_+ w_{g}^*   
 -  
 \left(\alpha_g (a) - \tfrac{\delta_1^3}{\card(G)} \right)_+ 
 \right\| 
\cdot
 \| f_g \|
\\\notag
& \hspace*{6 em} {\mbox{}} +
\| f_g \| 
\cdot
\left\|
 \left(\alpha_g (a) - \tfrac{\delta_1^3}{\card(G)} \right)_+ 
  - 
 \alpha_g (a) 
\right\|
\cdot
\| f_g \|
\\\notag
& \hspace*{6 em} {\mbox{}} +
\| 
 f_g \alpha_g (a) f_g  
- 
 \alpha_g (f_1 a f_1)
\|
\\\notag
&  \hspace*{5 em} {\mbox{}} <
   \tfrac{\delta_2}{20 \, \card(G)} +
 \tfrac{  \delta_1^3}{\card(G)} + 2  \ep'.
\end{align}
Third, we compute, 
using (\ref{Eq102.2019.04.16}) and (\ref{Eq26.2019.04.15}) at the third step,
and using (\ref{Eq4.2019.04.16}) at the last step,
\begin{align}
&
\Bigg\| y \left(a - \delta_2\right)_+ y^* - d_0 \Bigg\|
\\\notag
& \hspace*{3 em} {\mbox{}} =
\Bigg\| 
\Bigg(\sum_{g \in G} f_g w_{g}  \Bigg) \left(a_m - \delta_2\right)_+ \Bigg(\sum_{h \in G}  w_{h}^* f_h \Bigg)
- 
\sum_{g \in G} \alpha_g (f_1 a f_1) 
\Bigg\| 
\\\notag
&\hspace*{3 em} {\mbox{}} \leq
\Bigg\| 
\sum_{g \neq h} f_g  w_{g}  \left(a - \delta_2\right)_+ w_{h}^* f_h
\Bigg\|
\\\notag
&\hspace*{4 em} {\mbox{}} +
\Bigg\|
\sum_{{g \in G} } f_g  w_{g}  \left(a - \delta_2\right)_+ w_{g}^* f_g
-
\sum_{g \in G} \alpha_g (f_1 a f_1)
 \Bigg\|
\\\notag
&\hspace*{3 em} {\mbox{}}<
4 \card(G)^2 M^2 \ep' + 
\card (G) \left(
   \tfrac{ \delta_2}{20 \, \card(G)} +
 \tfrac{ \delta_1^3}{\card(G)} + 2  \ep' 
\right)
\\\notag
& \hspace*{3 em} {\mbox{}} <
\delta_2 +\delta_1^3 < \delta_1^2.
\end{align}
This completes the proof of (\ref{claim2.16}). 

Now, we use  (\ref{claim0.16}) and Lemma~\ref{PhiB.Lem.18.4}(\ref{PhiB.Lem.18.4.10.a}) at the first step,  and 
Lemma~\ref{PhiB.Lem.18.4}(\ref{LemdivibyNorm}) and 
the fact that $\| u_{\lambda_0} \|=1$ at the second step, to get
\begin{equation}\label{Eq777.05.14.219}
\left(f^2 u_{\lambda_0} a u_{\lambda_0} f^2 -  \tfrac{\ep}{4}\right)_+
\precsim_{A^\alpha}
\left( u_{\lambda_0} f^2 a  f^2 u_{\lambda_0} - \tfrac{\ep}{8} \right)_+
\precsim_{A^\alpha}
\left(  f^2 a  f^2  -\tfrac{\ep}{8} \right)_+.
\end{equation}
By (\ref{claim1.16}) and Lemma~\ref{PhiB.Lem.18.4}(\ref{PhiB.Lem.18.4.10.a}), we have
\begin{equation}\label{Eq6.2019.04.16}
(f^2 a f^2 -  \tfrac{\ep}{8})_+ 
\precsim_{A} 
 x (d_0 - \delta_1^2)_+  x^* 
\precsim_{A} 
(d_0 - \delta_1^2)_+.
\end{equation}
By (\ref{claim2.16}) and Lemma~\ref{PhiB.Lem.18.4}(\ref{PhiB.Lem.18.4.10.a}), we have 
\begin{equation}\label{Eq7.2019.04.16}
(d_0 - \delta_1^2)_+ 
\precsim_{A} 
y (a - \delta_2)_+  y^* 
\precsim_{A} 
(a - \delta_2)_+.
\end{equation}
Using 
(\ref{Eq1.05.16.2019})
and  
(\ref{Eq1.05.20.2019}), and 
applying  (\ref{Eq5.05.16.2019})
  with $f^2$ in place of $c$
and with $u_{\lambda_0}$ in place of $y$,
we get
\begin{equation}\label{Eq6.16.04.2019}
\left(a -  \ep \right)_+  
\precsim_A 
 (f^2 u_{\lambda_0} a u_{\lambda_0} f^2 -\tfrac{\ep}{4})_+ 
\oplus 
( u_{\lambda_0}^2 - u_{\lambda_0} f^2 u_{\lambda_0} - \tfrac{\ep}{8})_+.
\end{equation}
Therefore, 
using (\ref{Eq6.16.04.2019}) at the first step, 
using (\ref{Eq41.16.04.2019}) and (\ref{Eq777.05.14.219}) at the second step,
using (\ref{Eq6.2019.04.16}) at the third  step, and 
(\ref{Eq7.2019.04.16}) and (\ref{Eq3.16.04.2019})   at the fourth step,
\begin{align}\label{Eq100.2019.04.16}
(a -  \ep)_+
&\precsim_A  
\left(f^2 u_{\lambda_0} a u_{\lambda_0}  f^2 - \tfrac{\ep}{4}\right)_+ 
\oplus
 \left(u_{\lambda_0}^2 - u_{\lambda_0} f^2 u_{\lambda_0} - \tfrac{\ep}{8}\right)_+ 
\\\notag
&\precsim_A  
\left(f^2  a  f^2 - \tfrac{\ep}{8}\right)_+ 
\oplus
 \left(u_{\lambda_0}^2 - u_{\lambda_0} f u_{\lambda_0} - \tfrac{\ep}{16}\right)_+
\\\notag
&\precsim_A (d_0 - \delta_1^2)_+ 
\oplus
 \left(u_{\lambda_0}^2 - u_{\lambda_0} f u_{\lambda_0} - \tfrac{\ep}{16}\right)_+
\\\notag
&\precsim_A
(a - \dt_2)_+ \oplus h(a)
\\\notag
&\precsim_A
 (a - \dt_3)_+.
\end{align}
Now set 
\[
b_0 
=
(d_0 - \delta_1^2)_+ 
\oplus 
(u_{\lambda_0}^2 - u_{\lambda_0}  f u_{\lambda_0} - \tfrac{\ep}{16})_+.\]  
Since $d_0, f, u_{\lambda_0} \in \big( (A \otimes \cK)^{\alpha} \big)_+$, 
it follows that $b_0 \in \big( (A \otimes \cK)^{\alpha} \big)_+$. 
Therefore, using (\ref{Eq100.2019.04.16}),
\begin{equation}\label{Eq1.2019.04.20}
(a -  \ep)_+ \precsim_{A}  b_0 \precsim_{A} (a - \dt_3)_+.
\end{equation}
Since $0$ is a limit point of $\spec(a)$, we can choose 
$\dt_5 \in (0, \dt_4)$ such that $\spec (a) \cap (\dt_5, \dt_4) \neq \varnothing$.
Choose a continuous function $g \colon [0, \infty) \to [0, \infty)$ such that:
\begin{enumerate}
\setcounter{enumi}{\value{TmpEnumi}}
\item\label{Cont.func.g.a.19}
$g(t) = 0$ for all $t \in [0, \dt_5] \cup [\dt_4, \infty)$.
\item\label{Cont.func.g.b.19}
$g(t) \neq 0$ for all $t \in (\dt_5, \dt_4)$.
\end{enumerate}
Since $\spec (a) \cap (\dt_5, \dt_4) \neq \varnothing$, it follows that $g(a) \neq 0$. 
Then, by (\ref{Cont.func.g.a.19}) and  (\ref{Cont.func.g.b.19}),
\begin{equation}\label{Eq205.2019.04.20}
g(a) \oplus (a - \dt_4)_+ \precsim_A (a - \dt_5)_+.
\end{equation}
Since
$A$ is a simple non-type-I C*-algebra
and
 $\alpha$ has the weak tracial Rokhlin property,  also
$A \otimes \cK$
and
$(A \otimes \cK)^{\alpha}$   are  simple and non-type-I (see Theorem~4.1 of \cite{Rffl} and Corollary~3.3 of \cite{FG20}). 
Using \Lem{Small_fixed_element} with $g(a)$ in place of $x$,
 we can find $z \in \big( (A \otimes \cK)^\alpha\big)_+ \setminus \{0\}$ such that 
\begin{equation}\label{Eq3.05.17.2019}
z \precsim_A g(a).
\end{equation}
It follows from  Remark~\ref{Rmk2024.08.27} that $A$ is not of type I and by Theorem~4.1 of \cite{Rffl}, neither is
$(A \otimes \cK) ^{\alpha}$.
Now, using Lemma~2.1 of \cite{Ph14} with 
$\overline{z  (A \otimes \cK) ^{\alpha} z}$ in place of $A$, 
we can find a positive element 
 $d_1 \in \overline{z  (A \otimes \cK) ^{\alpha} z}$ such that $\spec (d_1) = [0, 1]$.
Using $d_1 \in \overline{z  (A \otimes \cK) ^{\alpha} z}$ at the first step and 
using (\ref{Eq3.05.17.2019}) at the second step, we have 
\begin{equation}\label{Eq201.2019.04.19}
 d_1 \precsim_A z \precsim_A g(a)
 \qquad
 \mbox{and}
 \qquad
\spec(d_1)= [0, 1].
\end{equation}
Therefore, using (\ref{Eq1.2019.04.20}) at the first and third step, 
using the first part of (\ref{Eq201.2019.04.19}) at the third step,
using (\ref{Eq555.05.16.2019}) at the fourth step,
and using (\ref{Eq205.2019.04.20}) at the fifth step, we get
\begin{align}\label{Eq4.2019.04.20}
(a - \ep)_+ 
\precsim_A 
b_0 
\precsim_A 
d_1 \oplus b_0 
&\precsim_A
 g(a) \oplus (a - \dt_3)_+ 
\\\notag
&\precsim_A
 g(a) \oplus (a - \dt_4)_+ 
\precsim_A (a - \dt_5)_+.
\end{align}
Now set 
\[
b = d_1 \oplus b_0=d_1 
\oplus 
(d_0 - \delta_1^2)_+ 
\oplus 
(u_{\lambda_0}^2 - u_{\lambda_0}  f u_{\lambda_0} - \tfrac{\ep}{8})_+
\quad
\mbox{and}
\quad
  \dt= \dt_5.
 \] 
Since $b_0, d_1 \in \big( (A \otimes \cK)^\alpha\big)_+$, 
we have $b \in \big( (A \otimes \cK)^\alpha\big)_+$.
Since $\spec (d_1) = [0, 1]$, we have $[0, 1] \subseteq \spec (b)$.
Therefore, using (\ref{Eq4.2019.04.20}),
\[
(a- \ep)_+ \precsim_A  b \precsim_A (a - \delta)_+
\qquad
\mbox{ and }
\qquad
[0, 1] \subseteq \spec (b).
\]
This completes the proof.
\end{proof}

We need the following lemma for the proof of Theorem~\ref{Thm.05.22.2019}. 
\begin{lem}\label{Cu.W.T.R.p.Surjec}
Let $A$ be a  stably finite simple non-type-I  not necessarily unital  \ca{} and let 
$\alpha \colon G \to \Aut (A)$ be an action 
of a finite group $G$ on $A$ which has the weak tracial Rokhlin property.
Let $a \in (A \otimes \cK)_{++}$ satisfy $a \sim_A \alpha_{g} (a)$ for all $g \in G$. 
Then there exists    $b \in (K \otimes A^{\alpha})_{++}$ such that 
$[ a ]_A = [ b ]_A$.
\end{lem}

\begin{proof}
With $\id \colon G \to \Aut(\cK)$ the trivial action,  
since $\alpha$ has the weak tracial Rokhlin property, it follows from Lemma~\ref{FG.thm.1.3} that
$\alpha \otimes \id \colon G \to \Aut (A \otimes \mathcal{K})$ also has the weak tracial Rokhlin property.
By abuse of notation, let us write $\alpha$ instead of  $\alpha \otimes \id$. 
For all $n \in \N$, set $\ep_n = \frac{1}{n}$. 
Using Lemma~\ref{Cu.Surjec} with $\ep_1$ in place of $\ep$, choose
$\dt_1>0$  and  $b _1 \in\left( (A \otimes \cK)^{\alpha} \right)_+$ such that 
\begin{equation}\label{Eq15.w.2019.04.18}
(a - \ep_1)_+ \precsim_A  b_1 \precsim_A (a - \dt_1)_+
\qquad
\mbox{ and }
\qquad
[0, 1] \subseteq \spec (b_1).
\end{equation}
Set $\eta_1= \ep_1$ and $\eta_2= \min \{\ep_2, \dt_1\}$.
Using Lemma~\ref{Cu.Surjec} with $\eta_2$ in place of $\ep$, choose
$\dt_2>0$ and  $b _2 \in \left( ( A \otimes \cK)^{\alpha} \right)_+$ such that 
\begin{equation}\label{Eq16.w.2019.04.18}
(a - \eta_2)_+ \precsim_A  b_2 \precsim_A (a - \dt_2)_+
\qquad
\mbox{ and }
\qquad
[0, 1] \subseteq \spec (b_2).
\end{equation}
Set  $\eta_3= \min \{\ep_3, \dt_2\}$.
Using Lemma~\ref{Cu.Surjec} with $\eta_3$ in place of $\ep$, choose
$\dt_3>0$ and  $b _3 \in \left( ( A \otimes \cK)^{\alpha} \right)_+$ such that 
\begin{equation}\label{Eq17.w.2019.04.18}
(a - \eta_3)_+ \precsim_A  b_3 \precsim_A (a - \dt_3)_+
\qquad
\mbox{ and }
\qquad
[0, 1] \subseteq \spec (b_3).
\end{equation}
By induction for all $k=1, 2,3, \ldots$,  we have
$\dt_{k}>0$ and  $b _k \in \left( ( A \otimes \cK)^{\alpha} \right)_+$ such that 
\begin{equation}\label{Eq18.w.2019.04.18}
(a - \eta_k)_+ \precsim_A  b_k \precsim_A (a - \dt_k)_+
\qquad
\mbox{ and }
\qquad
[0, 1] \subseteq \spec (b_k),
\end{equation}
where $\eta'_k= \min \{\ep_k, \dt_{k-1}\}$. 
By (\ref{Eq15.w.2019.04.18}), (\ref{Eq16.w.2019.04.18}), (\ref{Eq17.w.2019.04.18}), and 
(\ref{Eq17.w.2019.04.18}), it is clear that
\begin{align*}
(a - \eta_1)_+ 
 \precsim_{A} 
 b_1 &\precsim_{A} (a - \dt_1)_+ 
\\
&\precsim_{A} (a - \eta_2)_+  \precsim_{A} b_2 \precsim_{A} (a - \delta_2)_+
\\
& \precsim_{A} (a - \eta_3)_+ \precsim_{A}  \ldots
\precsim_{A} (a - \delta_{n-1})_+ \precsim_{A}  b_n \precsim_{A} (a - \delta_{n})_+ 
\\
&\precsim_{A}  (a - \eta_{n+1})_+ \precsim_A b_{n+1}  \precsim_{A} (a - \delta_{n+1})_+ \precsim_{A} a.
\end{align*}
Thus, we have constructed a sequence 
$( b_n )_{n \in \N}$ in $\left( ( A \otimes \cK)^{\alpha} \right)_+ $ such that, for all $n \in \N$, 
\begin{equation}\label{Eq26.2019.04.24}
b_n \precsim_A b_{n+1}, 
\quad
[0, 1] \subseteq \spec (b_n),
\quad
\mbox{and}
\quad
(a - \tfrac{1}{n})_+ \precsim_A  ( a - \eta_{n})_+ \precsim_A b_n \precsim_A a.
\end{equation}
Let $\id_n \colon G \to \Aut (M_n (\mathbb{C}))$ denote the trivial action
for  $n\in \Nz$. 
By Lemma~\ref{FG.thm.1.3}, 
 $\id_{n}  \otimes \alpha $ has the weak tracial  Rokhlin property for $n\in \N$.
 Then, by Proposition~4.3 of \cite{FG20},
$\dirlim \id_{n}  \otimes \alpha$ also has  the weak tracial  Rokhlin property.
Using Proposition \ref{W.T.R.P.inject}, we get
\begin{equation*}
b_1 \precsim_{A^\alpha} b_2 \precsim_{A^\alpha} b_3 
\precsim_{A^\alpha}   \ldots,
\end{equation*}
i.e.,
\begin{equation}\label{Eq22.w.2019.04.24}
 [ b_1 ]_{A^{\alpha}} \leq [ b_2 ]_{A^{\alpha}} 
 \leq [ b_3 ]_{A^{\alpha}} \leq [ b_4 ]_{A^{\alpha}} 
 \leq \ldots.
\end{equation}
By Corollary~\ref{SubCuSemiofPurelyPosi} and O1, 
there exists $b_{\infty} \in (A^{\alpha} \otimes \cK)_{++}$ 
such that 
\begin{equation*}
[ b_{\infty} ]_{A^{\alpha}} = \sup_{n} [ b_n ]_{A^{\alpha}}.
\end{equation*}
Now, we use this relation at the second step and 
 Remark~\ref{Cucpc} at the third step to get
\begin{align}\label{Eq13.2019.04.24}
[ b_{\infty} ]_{A} 
= 
\Cu(\iota) \big([ b_{\infty} ]_{A^{\alpha}}\big) 
&=
\Cu(\iota) \Big( \sup_{n} [ b_n ]_{A^{\alpha}} \Big)
\\\notag
&=
 \sup_n  \Big(\Cu(\iota) ([ b_n ]_{A^{\alpha}})  \Big) 
=
 \sup_n  [ b_n ]_{A}.
\end{align}
By the second part of (\ref{Eq26.2019.04.24}), we have, for all $n \in \N$,
\begin{equation}\label{Eq29.2019.04.24}
[ (a - \tfrac{1}{n})_+ ]_A \leq [  b_n ]_{A} \leq [ b_{\infty} ]_{A}
\qquad
\mbox{and}
\qquad
[ b_n ]_A \leq [ a ]_{A}.
\end{equation}
Therefore, 
using Lemma 1.25(1) of \cite{Ph14}
 at the first step,
using the first part of (\ref{Eq29.2019.04.24}) at the second step,
using (\ref{Eq13.2019.04.24}) at the third step, 
and using the second part of (\ref{Eq29.2019.04.24}) at the fourth step, we get
\begin{equation}\label{Eq31.2019.04.24}
[ a ]_{A}
=
\sup_n [ (a - \tfrac{1}{n})_+ ]_A \leq [ b_{\infty} ]_{A}
=
\sup_n [ b_n ]_A 
\leq 
[ a ]_{A}.
\end{equation}
With $b = b_{\infty}$, by (\ref{Eq31.2019.04.24}) we have
$[ b ]_{A} = [ a ]_{A}$, as desired.
\end{proof}

Recall the definition of $\Cu_+(A)$ (Definition \ref{D_9421_Pure}). 
As promised, we show in the following theorem that the map
$\Cu (\iota) \colon \Cu_+ (A^{\alpha}) \cup \{0\} \to \Cu_+ (A)^\alpha \cup \{0\}$  is an isomorphism of ordered semigroups. While this theorem is known for unital C*-algebras (see Theorem~5.5 of \cite{AGP19}),
our overarching objective is to prove this result for 
  non-unital C*-algebras.  
\begin{thm}\label{Thm.05.22.2019}
 Let $A$ be a  stably finite  simple non-type-I
 (not necessarily unital) \ca{} and let 
$\alpha \colon G \to \Aut (A)$ be an action 
of a finite group $G$ on $A$ which has the weak tracial Rokhlin property. Let $\iota \colon A^{\alpha} \to A$
denote the inclusion map. Then $\iota$ induces an isomorphism of ordered semigroups 
\[
\Cu (\iota) \colon \Cu_+ (A^{\alpha}) \cup \{0\} 
\to\Cu_+ (A)^{\alpha}  \cup  \{0\}.
\]
\end{thm}

\begin{proof}
 It follows from 
 Theorem~\ref{WC_+.injectivity} that the map 
 $\Cu (\iota) \colon \Cu (A^{\alpha}) \to \Cu (A)$ is injective on $\Cu_+ (A^\alpha)$.
Let us show that
 \[
 \Cu (\iota) \left(\Cu_+ (A^{\alpha})\right) = \Cu_+ (A)^{\alpha}.
 \]
Clearly,  $\Cu (\iota) \left(\Cu_+ (A^{\alpha})\right) \subseteq \Cu_+ (A)^{\alpha}$. 
So it suffices  to prove that 
$\Cu_+ (A)^{\alpha} \subseteq \Cu (\iota) \left(\Cu_+ (A^{\alpha})\right)$.
Let $\mu \in \Cu_+ (A)^{\alpha}$. Then there is 
$a \in (A \otimes \cK)_{++}$ with 
  $[ a ]_A = [ \alpha_{g} (a) ]_A$ for all $g \in G$ and
  $\mu = [ a ]_A$. 
  Since $\alpha$ has the weak tracial Rokhlin property and 
  $a \sim_A  \alpha_{g} (a)$ for all $g \in G$,
  it follows from \Lem{Cu.W.T.R.p.Surjec} that
   there exists 
  $b \in (K \otimes A^{\alpha})_{++}$ such that 
$[ a ]_A = [ b ]_A$.
Now set $\omega = [ b ]_{A^{\alpha}}$. Then
\[
\Cu (\iota) (\omega) = \Cu (\iota) ( [ b ]_{A^{\alpha}} ) 
= [ b ]_{A} = [ a ]_A = \mu.
\]
This completes the proof of the surjectivity of $\Cu (\iota) \colon \Cu_+ (A^{\alpha}) \cup \{0\} 
\to\Cu_+ (A)^{\alpha}  \cup  \{0\}$.
Now, let $x_1, x_2 \in \Cu_+ (A)^{\alpha}$ with
$x_1\leq x_2$. By surjectivity, there exist $a, b \in (A^\alpha \otimes \cK)_{++}$ such that 
$\Cu(\iota) ([ a ]_{A^{\alpha}}) = x_1$ and $\Cu(\iota) ([ b ]_{A^{\alpha}}) = x_1$. Since 
$x_1\leq x_2$, it follows that  
$[ a ]_{A} \leq [ b ]_{A}$. So, $a \precsim_A b$. Using Proposition~\ref{W.T.R.P.inject}, we get $a \precsim_{A^\alpha} b$. Therefore, 
$[ a ]_{A^\alpha} \leq [ b ]_{A^\alpha}$. 
This shows that the inverse of $\Cu (\iota)$ preserves the order as well.
This completes the proof.
\end{proof}
\begin{rmk}
In the above theorem,
$\Cu_+ (A^{\alpha})  \cup  \{0\}$
is a sub-$\Cu$-semigroup  of $\Cu (A)$
by
Corollary~\ref{SubCuSemiofPurelyPosi}. 
Since 
$\Cu (\iota) \colon \Cu_+ (A^{\alpha}) \cup \{0\} 
\to\Cu_+ (A)^{\alpha}  \cup  \{0\}$ is an
isomorphism of ordered semigroups, it follows that 
$\Cu_+ (A)^{\alpha}  \cup  \{0\}$ is also a $\Cu$-semigroup. Furthermore, of course, 
$\Cu (\iota) \colon \Cu_+ (A^{\alpha}) \cup \{0\} 
\to\Cu_+ (A)^{\alpha}  \cup  \{0\}$ is a $\Cu$-morphism.
\end{rmk}  
%
Our next goal is to find precise relationships between 
the Cuntz semigroups of the fixed point algebra and the crossed product, and their purely positive parts. To do this, we need Lemma~\ref{CornerLem1} and 
Lemma~\ref{CornerLem2} to prove Proposition~\ref{prcFullCu}.
\begin{lem}\label{CornerLem1}
Let $A$ be a C*-algebra, let $p \in \cM(A)$ be a full projection, and let 
$a \in (A \otimes \cK)_+$. Then
for $\ep> 0$, there exist $m \in \N$ and 
$b \in M_{m} (pAp)_+$ such that 
$(a - \ep)_+ \precsim_A b$. 
\end{lem}
\begin{proof}
Let $\ep>0$. We use Lemma~1.9 of \cite{Ph14} to choose $t \in \N$ and $c \in M_t (A)_+$ such that $(a - \ep)_+ \sim_A c$. Now, find $\delta >0$ such that 
$(a - \tfrac{\ep}{2})_+ \precsim_A (c -\dt)_+$. Since $p$ is full, it follows that 
$\overline{\mathrm{Span} \left( M_t (A) (e_{1, 1} \otimes p) M_t (A) \right)} = M_t (A)$. 
So,  there exist $s \in \N$ and
 $d_1, \ldots , d_s, f_1, \ldots , f_s \in M_t (A)$ such that 
 \[
\left\| c - \sum_{j=1}^{s} d_{j}  (e_{1, 1} \otimes p) f_{j}\right\| < \delta.
 \]
 It is clear that 
 \[
\left\| c - \sum_{j=1}^{s} f^*_{j}  (e_{1, 1} \otimes p) d^*_{j}\right\| < \delta.
 \]
 Putting these together, we get
 \[
 \left\|  c - \tfrac{1}{2} \cdot \sum_{j=1}^{s} 
 \left[
 d_{j}  (e_{1, 1} \otimes p) f_{j} - f^*_{j}  (e_{1, 1} \otimes p) d^*_{j}
 \right]
 \right\| < \delta.
 \]
Applying Proposition~II.3.1.9 of \cite{BlaBo}
with $x= (e_{1, 1} \otimes p) d_{j}^{*}$ and $y=(e_{1, 1} \otimes p)
f_j$, we get
\begin{equation}\label{Eq.1.16.2019.2}
d_j (e_{1, 1} \otimes p) f_j + f_{j}^{*} (e_{1, 1} \otimes p) d_{j}^{*} 
\leq d_j (e_{1, 1} \otimes p) d_{j}^{*} + f_{j}^{*} (e_{1, 1} \otimes p) f_{j}
\end{equation}
for all $j\in \{1, \ldots , m\}$.
We use 
 (\ref{Eq.1.16.2019.2}) and Lemma~\ref{PhiB.Lem.18.4}(\ref{PhiB.Lem.18.4.10.a})  to get 
\begin{align*}
(c - \delta)_+ 
& \precsim_A 
\frac{1}{2} \sum_{j=1}^{s} ( d_{j}  (e_{1, 1} \otimes p) f_{j} + f^*_{j}  (e_{1, 1} \otimes p) d^*_{j} )
\leq
\frac{1}{2} \cdot \sum_{j=1}^{s} (  d_j (e_{1, 1} \otimes p) d_{j}^{*} + f_{j}^{*} p f_{j}).
\end{align*}
Now set 
\[
x_j=
\begin{cases}
\frac{1}{\sqrt{2}} f_j       &  1 \leq j\leq s, \\
\frac{1}{\sqrt{2}} d_{j-m}^*       &  s < j\leq 2s.
\end{cases} 
\]
Then
\[
(c - \delta)_+  \precsim_A \sum_{j=1}^{2m} x_j (e_{1, 1} \otimes p)  x_{j}^* \precsim \bigoplus_{j=1}^{2s} x_j (e_{1, 1} \otimes p) x_{j}^* 
\sim_A
\bigoplus_{j=1}^{2s} (e_{1, 1} \otimes p) x^*_j x_{j} (e_{1, 1} \otimes p).
\]
Now, set $n = 2s$ and $b= \bigoplus_{j=1}^{2s} (e_{1, 1} \otimes p) x^*_j x_{j} (e_{1, 1} \otimes p)$. 
 Then
\[
(a - \ep)_+ \precsim_A (c - \delta)_+ \precsim_A b.  
\]
This completes the proof of the lemma. 
\end{proof}
\begin{lem}\label{CornerLem2}
Let $A$ be a C*-algebra, let $p \in \cM(A)$ be a full projection, and let 
$a \in (A \otimes \cK)_+$. 
Suppose that 
for any $\ep> 0$, there exists 
$c \in (\mathcal{K} \otimes pAp)_+$ such that 
$(a - \ep)_+ \precsim_A c$. Then
$[ a ] \in \Cu (pAp)$. 
\end{lem}
\begin{proof}
By induction on~$n$,
we construct sequences $(\ep_n)_{n \in \Nz}$ in $(0, \infty)$,
$(d_n)_{n \in \Nz}$ in $(\cK \otimes pAp)_{+}$
such that $\lim_{n \to \I} \ep_n = 0$,
$d_0 \precsim_A (a - \ep_{0})_{+}$,
and, for all $n \in \Nz$,  we have
\begin{enumerate}
\item
$\ep_{n + 1} < \ep_{n}$,
\item
$d_n \in  (\cK \otimes p A p)_{+}$,
\item
$
(a - \ep_{n})_{+}
  \precsim_A d_{n + 1} \precsim_A (a - \ep_{n + 1})_{+}$.
\end{enumerate}

To begin,
set $\ep_0 = 1$.
Given $\ep_n$ with $n \in \Nz$,
apply \Lem{CornerLem1}
with $\ep_n$ in place of $\ep$,
getting
 $m (n) \in \N$,
and $b_{n} \in ( \cK \otimes pAp)_{+}$
such that
\begin{equation}\label{Eq15_2019_05_14}
(a - \tfrac{\ep_{n}}{2})_{+}
  \precsim_A b_{n}.
\end{equation}
Choose $c_n \in \cK \otimes A$
such that 
\[
\|(a - \tfrac{\ep_{n}}{2})_{+} - c_n b_n c_n^* \| < \tfrac{\ep_n}{4}.
\]
So,  Lemma~\ref{PhiB.Lem.18.4}(\ref{Item_9420_LgSb_1_6}),  we have 
\begin{equation}\label{EQ2024.12.15}
(a - \ep_n)_{+} \precsim 
(c_n b_n c_n^* - \tfrac{\ep_n}{4})_+
\precsim 
(a - \tfrac{\ep_{n}}{2})_{+}.
\end{equation}
By Lemma~\ref{PhiB.Lem.18.4}(\ref{PhiB.Lem.18.4.6}), we have 
$
(c_n b_n c_n^* - \tfrac{\ep_n}{4})_+
\sim 
( b_n^{1/2} c_n^* c_n  b_n^{1/2} - \tfrac{\ep_n}{4})_+$. Now, set 
$d_{n+1} = \left( b_n^{1/2} c_n^* c_n  b_n^{1/2} - \tfrac{\ep_n}{4}\right)_+$.
It is clear that 
$d_{n+1} \in \cK \otimes pAp$ and, then, by (\ref{EQ2024.12.15}), 
\[
(a - \ep_n)_{+} \precsim 
d_{n+1}
\precsim 
(a - \tfrac{\ep_{n}}{2})_{+}.
\]
Then set $\ep_{n + 1} = \frac{\ep_n}{2}$.
The induction is complete.
We now have
$d_0
  \precsim_A d_{1}
  \precsim_A d_2
  \precsim_A \cdots$.
  Since
$pAp$ is a hereditary C*-subalgebra of $A$,
it follows from Lemma~\ref{PhiB.Lem.18.4}(\ref{KR00.lem.22}) that
\begin{equation}\label{Eq21_2019_05_15}
d_0
 \precsim_{pAp} d_1
 \precsim_{p A p} d_2
 \precsim_{pAp} \cdots.
\end{equation}
By O1,
there exists $d \in (\cK \otimes pAp)_{+}$ such that
$[ d ]_{pAp}
 = \sup_{n} \, [ d_n ]_{p A p}$.
Moreover,
for all $n \in \Nz$,
we have
$( a - \ep_{n})_{+} \precsim_A d_{n + 1} \precsim_A a$.
Since $\lim_{n \to \I} \ep_n = 0$,
it follows from Lemma 1.25(1) of~\cite{Ph14} that
$\sup_{n} \, [ d_n ]_A = [ a ]_A$.
So $[ d ]_{A} = [ a ]_{A}$. This completes the proof. 
\end{proof}
Putting
Lemma \ref{CornerLem1} and Lemma \ref{CornerLem2} together, we have the following proposition. 
\begin{prp}\label{prcFullCu}
Let $A$ be a C*-algebra and let $p \in \cM(A)$ be a full projection. Then  $\xi \colon pAp \to A$ be the inclusion map. 
Then the map $\Cu(\xi) \colon \Cu (pAp) \to \Cu(A)$ induced by the inclusion map is a $\Cu$-isomorphism.
\end{prp}

  We conclude this section by proving the following proposition. Recall the definition of $\Cu_+(A)$ (Definition \ref{D_9421_Pure}). 
We refer to Definition 2.8 of \cite{GS17}
for the definition of 
$\Cu(A)^{\Cu(\alpha)}$. It follows from 
Lemma 2.9 of \cite{GS17} that $\Cu(A)^{\Cu(\alpha)}$ is always a sub-$\Cu$-semigroup of $\Cu(A)$.
\begin{prp}\label{CuCrosFix}
Let $A$ be a stably finite simple non-type-I C*-algebra and let $G$ be a finite group.
Let $\alpha \colon G \to \Aut (A)$ be an action of $G$ on $A$ with the weak tracial Rokhlin property. Then, with respect to the natural maps: 
\begin{enumerate}
\item\label{CuCrosFix.1}
$\Cu(A\rtimes_{\alpha} G) 
\cong
 \Cu(A^{\alpha})$. 
\item\label{CuCrosFix.2}
$\Cu(A)^{\alpha}=\Cu(A)^{\Cu(\alpha)}$, and, therefore, $\Cu(A)^{\alpha}$ is a  sub-$\Cu$-semigroup of $\Cu (A)$. 
\item\label{CuCrosFix.3}
$
\Cu_+ (A\rtimes_{\alpha} G) \cup \{0\} 
\cong 
\Cu_+ (A^{\alpha}) \cup \{0\}  
\cong
\Cu_+ (A)^{\alpha} \cup \{0\}  
$.
\end{enumerate}
\end{prp}
\begin{proof}
Part (\ref{CuCrosFix.1}) follows from 
Proposition~\ref{prcFullCu} 
and Lemma~\ref{Fixedpoint_corner}(\ref{Fixedpoint_corner_d}). 
The proof of part (\ref{CuCrosFix.2}) is the same as the proof of Proposition~7.8 of \cite{ATV24} except that we should now use Lemma~\ref{Cu.W.T.R.p.Surjec} instead of Lemma~4.5 of \cite{AGP19}. 
Part (\ref{CuCrosFix.3}) follows from
Theorem~\ref{Thm.05.22.2019} and (\ref{CuCrosFix.1}).
\end{proof}
The above proposition generalizes Proposition~7.8 of of \cite{ATV24} to non-unital C*-algebras. 
\section{The relative radius of comparison of the crossed product}
\label{ReRc}
In this section, our main goal is to obtain some relationships between the relative radius of comparison of the crossed product of a non-unital C*-algebra by a finite group action, the relative radius of comparison of the fixed point algebra, and that of the original C*-algebra,  when the action has 
 the (non-unital) weak tracial Rokhlin property.
\begin{ntn}
Let $A$ be a C*-algebra, let $G$ be a discrete group, and let $\alpha \colon G \to \Aut (A)$ be an action of $G$ on $A$. We define
\[
\iota \colon A^\alpha \to A
\quad
\mbox{ and  }
\quad
\kappa \colon A \to A\rtimes_{\alpha} G
\]
to be the inclusion maps. 
 By abuse of notation, we often denote the amplification of these maps to matrices or to the stabilization also by the same symbol. 
\end{ntn}

The following proposition is Theorem 5.11 of \cite{ATV24}.
\begin{prp}\label{softGlobalGlimm}
Let $A$ be a stable C*-algebra with the Global Glimm Property. Then, for any $a\in A_+$, there exists a soft element 
$b\in A_+$ with $b \precsim_A a$ and such that
$d_\tau (a) = d_\tau (b)$
for every $\tau \in \EQT_2(A)$.
\end{prp}
It follows from Lemma~\ref{SimplNonTypeGlimm} that the above proposition can apply to all $A \otimes \K$ when $A$ is a simple non-type-I C*-algebra.  

The following theorem establishes a relationship between the relative radius of comparison of the fixed point algebra and that of the original C*-algebra.
\begin{thm}\label{RcofFixedPoint}
Let $A$ be  a stably finite simple  non-type-I  C*-algebra and
let $\alpha \colon G \to \Aut (A)$ be an action of a  finite group $G$ on $A$
which has the weak tracial Rokhlin property.
Let $a \in (A^{\alpha} \otimes \K)_{+} \setminus \{0\}$. Then 
\[
\rc \big(\Cu (A^\alpha), \ [ a ] \big)   \leq \rc \big(\Cu (A), \ [ \iota ( a ) ] \big).
\] 
\end{thm}
\begin{proof}
We set 
\[
\Gamma =
\Big\{ r \in (0, \infty) \colon \Cu (A) \mbox{ has $r$-comparison relative to } [ \iota ( a ) ]_A \Big\}
\]
and 
\[
\Lambda = \Big\{ r \in (0, \infty) \colon \Cu (A^\alpha)  \mbox{ has $r$-comparison relative to } [  a ]_{A^\alpha} \Big\}. 
\]
It is enough to show that $\Gamma \subseteq \Lambda$. 
So, let $r> 0$ in $\Gamma$. 
Let 
$x, y \in (A^{\alpha} \otimes \K)_{+}$ satisfy
\begin{equation}
\label{2024.01.16.Q1}
d_{\tau} ( x )
+ r \cdot
d_{\tau} ( a )
\leq
d_{\tau} (y)
\end{equation}
for all $\tau \in \EQT_{2} (A^\alpha)$. If 
$y=0$, then, by (\ref{2024.01.16.Q1}), we have $x=0$. Therefore, $x \precsim_{A^\alpha} y$. Now assume that 
$y \neq 0$. We may assume that $\| y \|\leq 1$.
Since $A^\alpha$ is simple non-type-I C*-algebra, it follows from Lemma~\ref{SimplNonTypeGlimm} that it has the  Global Glimm Property. Now we use 
Theorem~\ref{softGlobalGlimm} with 
$A^{\alpha} \otimes \K$ in place of $A$ and $y$ in place of  $a$ to get a soft and non-zero
 $y' \in (A^{\alpha} \otimes \K)_+$
 such that:
 \begin{enumerate}
 \item \label{E2024.03.05.a}
 $d_{\tau} (y) = d_{\tau} (y')$
 for all $\tau \in \EQT_2 (A^\alpha)$.
 \item \label{E2024.03.05.b}
 $y' \precsim_{A^\alpha} y$.
 \end{enumerate}
We know that for all $\tau \in \EQT_2 (A)$, we have 
$\tau \circ \iota
 \in \EQT_2 (A^\alpha)$. Using (\ref{2024.01.16.Q1}) and 
(\ref{E2024.03.05.a}), we get
\begin{equation}
\label{E2.2024.03.05}
d_{\tau} ( \iota (x) )
+ r \cdot
d_{\tau} ( \iota (a) )
\leq
d_{\tau} ( \iota (y'))
\end{equation}
for all $\tau \in \EQT_{2} (A)$.
Since $r \in \Gamma$, it follows that
$ \iota ( x ) \precsim_{A} \iota (y') $. Since  $y'$ is soft and non-zero, it follows from 
Lemma~\ref{Pure-Soft-Equivalent} 
that $0$ is a limit point of $y'$. Thus, by Proposition~\ref{W.T.R.P.inject},  $x \precsim_{A^\alpha} y'$. Now, using this and (\ref{E2024.03.05.b}), we get
$x \precsim_{A^\alpha} y$. This implies that 
$r \in \Lambda$. 
Taking the infimum of $\Gamma$ and $\Lambda$, we get the result. 
\end{proof}
We require the following lemmas to prove Theorem~\ref{TrcaeofProj}. 
\begin{lem}\label{Lem2.4Gen}
Let $A$ be a simple  non-type-I  C*-algebra. Let 
$\tau \in \EQT_2 (A) \setminus
 \{\tau_{\infty}\}$. Let $a$ be a positive element in 
  $\Ped(A\otimes \mathcal{K}) \setminus \{0\}$. Then, for every $\ep>0$, there exists a positive non-zero element $b$ in $\overline{a( A \otimes \mathcal{K}) a}$ such that $d_{\tau} (b) < \ep$. 
\end{lem}
\begin{proof}
Let $\tau \in \EQT_2 (A)$. It is clear that 
$\Dom_{\tfrac{1}{2}} (\tau)=\{0\}$ if and only if 
$\tau=\tau_{\infty}$. (This is true if $A$ is not simple.) 
Since $A$ is simple, it follows from Lemma~\ref{DomLem}(\ref{DomLem.4}) that $\Ped (A \otimes \mathcal{K}) \subseteq \Dom_{\tfrac{1}{2}} (\tau)$ for all $\tau \in \EQT_2 (A) \setminus \{\tau_\infty\}$. 
This implies that $\tau(a)$ and $d_{\tau} (a)$ both are finite. So, there exists $\eta \in [0, \infty)$ such that
  $d_{\tau} (a) =\eta$. Let $\ep>0$. 
  We choose $n \in \N$ such that 
  $\tfrac{\eta}{n} < \ep$. Since $A$ is a non-type-I C*-algebra, so also is $A \otimes \mathcal{K}$. Now, 
  we use Lemma \ref{Lem2.4Ph14} to find
   $b_1, b_2, \ldots, b_n \in (A \otimes \mathcal{K})_+ \setminus \{0\}$ such that:
   \begin{enumerate}
   \item\label{Lem2.4.a}
   $b_j b_k =0$ for all $j, k$ with $j\neq k$. 
   \item\label{Lem2.4.b}
   $b_1 \sim_A b_2 \sim_A \ldots \sim_A b_n$. 
   \item\label{Lem2.4.c}
   $b_1 + b_2 + \ldots + b_n \in \overline{a(A \otimes \mathcal{K})a}$.
\end{enumerate} 
(\ref{Lem2.4.c}) implies that
$b_1 + b_2 + \ldots + b_n \precsim_A a$. Take $d_\tau$ from both side of it and use (\ref{Lem2.4.a}) and (\ref{Lem2.4.b}), we have
\[
\sum_{j=1}^{n} d_{\tau} (b_j) = 
d_{\tau} \left(\sum_{j=1}^{n} b_j \right) \leq d_{\tau} (a)= \eta. 
\]
Since $d_{\tau} (b_1) = d_{\tau} (b_j)$ for all $j$, it follows that 
$n \cdot d_{\tau} (b_1) \leq \eta$. This implies that 
$d_{\tau} (b_1) < \ep$. Now, it is enough to take $b=b_1$. 
\end{proof}
The following lemma is known. 
\begin{lem}\label{Fixedpoint_corner}
let $A$ be a \ca,
let $G$ be a finite group,
 and
let $\alpha \colon G \to \Aut (A)$  be an action of $G$ on~$A$.
Let
$p = \frac{1}{\card (G)} \sum_{g \in G} u_g \in \cM(A\rtimes_{\alpha} G)$.
Then:
\begin{enumerate}
\item\label{Fixedpoint_corner_a}
$p$ is a projection in $\cM(A\rtimes_{\alpha} G)$.
\item\label{Fixedpoint_corner_c}
If $a \in A^{\alpha}$, then $p a p = ap$.
\item\label{Fixedpoint_corner_d}
The map $\varphi \colon A^{\alpha} \to p (A\rtimes_{\alpha} G) p$ given by
$a \mapsto \big( (\kappa \circ \iota ) (a )\big)  p$ is an isomorphism.
\end{enumerate}
\end{lem}

\begin{proof}
 See~\cite{Ros79} and Lemma 4.3 of \cite{AGP19}.
\end{proof}
The following corollary follows immediately from Proposition \ref{prcFullCu}. 
\begin{cor}\label{RcofFullCorner}
Let $A$ be a C*-algebra, let $p \in \cM(A)$ be a full projection. Let $\xi \colon pAp \to A$ denote the inclusion map and let 
$c$ be a full positive element in $pAp \otimes \mathcal{K}$.  
Then 
\[
\rc \big(\Cu (pAp), \ [ c ] \big)
=
\rc \big(\Cu (A), \ [ \xi ( c ) ] \big).
\]
\end{cor}
The following lemma is  crucial for us, as one of  the key points to compute the relative radius of comparison of the crossed product is that we want to know the value of extended traces when $a$ is replaced by $p\cdot a$. 
\begin{thm}\label{TrcaeofProj}
Let $A$ be a simple infinite dimensional C*-algebra. Let 
$G$ be a finite group and let 
$\alpha \colon G \to \Aut(A)$ be an action of $G$ on $A$ with the weak tracial Rokhlin property. Set 
 $p= \frac{1}{\card (G)} \sum_{g} u_g$ in 
 $\cM \big(  (A \otimes \mathcal{K}) \rtimes_{\alpha \otimes \id} G \big)$. 
 Let $a$ be a positive element of
  $\Ped(A^\alpha \otimes \mathcal{K}) \setminus \{0\}$ and let 
  $\tau \in \ET( A\rtimes_{\alpha} G))$. Then:
\begin{enumerate}
\item\label{TrcaeofProj.1}
 $\tau\left( \big((\kappa \circ \iota) (a)\big) p \right) = \frac{1}{\card (G)} \cdot \tau ((\kappa \circ \iota) (a))$
\item\label{TrcaeofProj.2}
$d_{\tau} \left( \big((\kappa \circ \iota) (a) \big) p \right)  =
 \frac{1}{\card (G)} \cdot d_{\tau} \left((\kappa \circ \iota) (a) \right)$.
\end{enumerate}
\end{thm}
\begin{proof}
We only prove the first part as the second part follows from the definition of $d_{\tau}$ and the fact that $C^*(a) \subseteq \Ped(A^\alpha \otimes \mathcal{K})$. 
So, to prove (\ref{TrcaeofProj.1}), 
we may assume that $\| a \|\leq 1$.
By abuse of notation, we will also write 
$a$ instead of $(\kappa \circ \iota) (a)$. Let $\id \colon G \to \Aut(\cK)$ denote the trivial action. 
Since $\alpha$ has the weak tracial Rokhlin property, it follows from Lemma~\ref{FG.thm.1.3} that
$\alpha \otimes \id \colon G \to \Aut (A \otimes \mathcal{K})$ also has the weak tracial Rokhlin property. We set $\beta = \alpha \otimes \id$. 
Let $\ep \in (0, \infty)$ be arbitrary. It suffices to show that 
\[
\left| \tau (pa) - \frac{1}{\card (G)} \cdot \tau (a) \right | < \ep. 
\]
Since $\Ped (A^\alpha \otimes \mathcal{K} ) \subseteq 
\Ped ( (A\rtimes_{\alpha} G) \otimes \mathcal{K})$ (see Proposition~4.5 of \cite{IvKu19}), we may use Lemma \ref{Lem2.4Gen}
with $\frac{\ep}{6}$ and $(A\rtimes_{\alpha} G) \otimes \mathcal{K}$ to get a non-zero positive element $b$ in $\overline{a ((A\rtimes_{\alpha} G) \otimes \mathcal{K}) a}$ such that 
\begin{equation}\label{EQ1.2024.06.17}
d_{\tau} (b) < \frac{\ep}{6}.
\end{equation}
We may assume $\| b \|=1$. Since  $a \in \Ped (A^\alpha \otimes \mathcal{K}) \subseteq 
\Ped ( (A\rtimes_{\alpha} G) \otimes \mathcal{K})$, it follows from Lemma~\ref{BonPed} that $\tau$ is bounded  on 
$\overline{a \big((A\rtimes_{\alpha} G) \otimes \mathcal{K} \big) a}$. 
 We set
  $\eta \in [0, \infty)$ to be the norm of the restriction of $\tau$ to 
 $\overline{a \big((A\rtimes_{\alpha} G) \otimes \mathcal{K}\big) a}$. If $\eta=0$, then the theorem is trivial. So, we assume that $\eta \neq 0$.
Set
\[
\ep'= \frac{\ep}{24 \eta \card (G)^3}
\qquad
\mbox{and}
\qquad
F= \left\{ a^{\tfrac{1}{4}}, a^{\tfrac{3}{4}} \right\}.
\] 
  Now, we use Lemma \ref{invariant.contractions} with $a^{\tfrac{1}{2}}$ in place of $y$, $b$ in place of $x$, $\ep'$ in place of $\ep$, and $F$ as given to get positive contractions 
  $f_g \in G$ in $A$ for $g \in G$ such that, 
with $f = \sum_{g \in G} f_g$ , the following properties hold:
 \begin{enumerate}
 \item \label{2-W.T.R.P.66}
$\|f_{g} f_{h}\| <  \ep'$ for all $g, h \in G$ with $g \neq h$.
\item \label{2-W.T.R.P.11}
$\| c f_g - f_g c \| < \ep'$ for all $g \in G $ and all $c \in F$.
\item \label{2-W.T.R.P.22}
$\| \beta_{g}( f_h ) - f_{gh} \|  < \ep'$ for all $g , h\in G$.
\item \label{2-W.T.R.P.33}
$(a - a^{\tfrac{1}{2}} f a^{\tfrac{1}{2}} - \ep')_+ \precsim_A b$.
\item \label{2-W.T.R.P.55}
 $f \in A^{\alpha} \otimes \mathcal{K}$ and $\| f \| = 1$.
\setcounter{TmpEnumi}{\value{enumi}}
\end{enumerate}
 We use Lemma \ref{M.A_lemma_05.17.2019} with $2\ep'$ in place of $\ep$ and $f$ in place of $a$
 to get
 \begin{equation*}
 (a - a^{\tfrac{1}{2}} f a^{\tfrac{1}{2}} - 2 \ep')_+ \precsim_A
 (a - a^{\tfrac{1}{2}} f^2 a^{\tfrac{1}{2}} - 2 \ep')_+ \precsim_A
 (a - a^{\tfrac{1}{2}} f a^{\tfrac{1}{2}} -  \ep')_+. 
\end{equation*}
Now, using this relation at the first step, using 
(\ref{2-W.T.R.P.33}) at the second step, and (\ref{EQ1.2024.06.17}) at the last step, we get
\begin{align}\label{EQ3.2024.06.17}
d_{\tau} \left((a - a^{\tfrac{1}{2}} f^2 a^{\tfrac{1}{2}} - 2 \ep')_+\right)
\leq 
d_{\tau} \left((a - a^{\tfrac{1}{2}} f a^{\tfrac{1}{2}} -  \ep')_+\right)
\leq
d_{\tau} (b) < \frac{\ep}{6}. 
\end{align}
It is clear that 
\begin{equation*}
\left\|
\left(a-a^{\frac{1}{2}} f^{2} a^{\frac{1}{2}}\right)
-
\left(a-a^{\frac{1}{2}} f^2 a^{\frac{1}{2}} - 2\ep'\right)_+\right\|
< 2\ep'.
\end{equation*}
Using this relation and (\ref{EQ3.2024.06.17})  at the second step, we get
\begin{align}\label{EQ8.2024.06.17}
\left|
\tau \left( a-a^{\frac{1}{2}} f^{2} a^{\frac{1}{2}} \right) 
\right|
&\leq
  \left|
  \tau \left( (a-a^{\frac{1}{2}} f^{2} a^{\frac{1}{2}} \right)
  -
  \tau \left( (a-a^{\frac{1}{2}} f^{2} a^{\frac{1}{2}} - 2\ep')_+ \right)
  \right|
   \\& \quad + \notag
 \left|
 \tau \left( (a-a^{\frac{1}{2}} f^{2} a^{\frac{1}{2}} - 2\ep')_+ \right)
 \right|
\\ \notag &<
2\eta \ep' + \frac{\ep}{6}
<
 \frac{\ep}{3}.
\end{align}
Now we claim that
\begin{equation}\label{EQ9.2024.06.17}
\left|
\tau(p a p) - \tau \left(p a^{\frac{1}{2}} f^{2} a^{\frac{1}{2}} p\right)
\right|<\frac{\ep}{3}. 
\end{equation}

To prove (\ref{EQ9.2024.06.17}), we  use
the fact that 
$a-a^{\frac{1}{2}} f^{2} a^{\frac{1}{2}} \in A^{\alpha}$ and  (\ref{EQ8.2024.06.17})
to
get
\begin{align*}
\label{EQ10.2024.06.17}
\tau \left(p (a-a^{\frac{1}{2}} f^{2} a^{\frac{1}{2}}) p\right)
&=
\tau \left( (a-a^{\frac{1}{2}} f^{2} a^{\frac{1}{2}})^{\tfrac{1}{2}} p (a-a^{\frac{1}{2}} f^{2} a^{\frac{1}{2}})^{\tfrac{1}{2}} \right)
\\&\leq \notag 
\| p \| \cdot \tau \left( (a-a^{\frac{1}{2}} f^{2} a^{\frac{1}{2}}) \right)
<
\frac{\ep}{3}. 
\end{align*}
Now, we claim that
\begin{equation}
\label{EQ11.2024.06.17}
\left|
\tau \left( p a^{\tfrac{1}{2}} f^2 a^{\tfrac{1}{2}}  p \right)
-
\frac{1}{\card (G)} \cdot \tau \left( a^{\tfrac{1}{2}} f^2 a^{\tfrac{1}{2}} \right) 
\right|
<
\frac{\ep}{3}. 
\end{equation}

To prove (\ref{EQ11.2024.06.17}), we first use 
(\ref{2-W.T.R.P.66}) and (\ref{2-W.T.R.P.22}) to
estimate, for $g\in G \setminus \{1\}$, 
\begin{align}\label{EQ12.2024.06.17}
\left\|
a^{\frac{1}{8}} f_h u_{g} f_{h} a^{\frac{7}{8}}
\right\|
=
\left\| 
a^{\frac{1}{8}} f_h \beta_{g} (f_{h}) u_g a^{\tfrac{7}{8}}
\right\| 
& \leq
\|a^{\frac{1}{8}} f_h  \| \cdot \| \beta_{g} (f_{h}) - f_{gh}\| \cdot \|u_g a^{\tfrac{7}{8}} \|
 \\ \notag
& \quad +
\|a^{\frac{1}{8}} \|  \cdot \| f_{h} f_{gh}\| \cdot \|u_g a^{\tfrac{7}{8}} \| 
\\ \notag
& < 2\ep'.
\end{align}
Second, by (\ref{2-W.T.R.P.11}), we  have 
\begin{equation}\label{EQ1.2024.06.20}
\left\|
a^{\tfrac{1}{2}} {f_t}^{2} u_g a^{\tfrac{1}{2}} 
- 
a^{\tfrac{1}{4}} f_{t} a^{\tfrac{1}{4}} f_t u_g a^{\frac{1}{2}}
\right\| 
=
\|
a^{\tfrac{1}{4}} 
\| \cdot 
\| a^{\tfrac{1}{4}}  {f_t} 
-
{f_t} a^{\tfrac{1}{4}}  \|
\cdot 
\|
{f_t}  u_g a^{\tfrac{1}{2}} 
\|
<\ep',
\end{equation}
and 
\begin{equation}
\label{EQ2.2024.06.20}
\left\|
a^{\frac{1}{8}} f_t u_g a^{\frac{3}{4}} f_t a^{\frac{1}{8}}
-
a^{\frac{1}{8}} f_{t} u_g f_{t}  a^{\frac{7}{8}}
\right\|
=
\|a^{\frac{1}{8}} f_t u_g \| \cdot
 \|
  a^{\frac{3}{4}} f_t 
 - 
 f_t a^{\frac{3}{4}} \|
 \cdot \| a^{\frac{1}{8}}\|
< \ep'.
\end{equation}
Third, for every $t\in G$ and $g\in G \setminus \{1\}$, we use the fact that $\tau$ is tracial at the first step and use
(\ref{EQ12.2024.06.17}),
(\ref{EQ1.2024.06.20}), and (\ref{EQ2.2024.06.20}) at the second step to get
\begin{align}
\label{EQ3.2024.06.20}
\left|
\tau \left(a^{\tfrac{1}{2}} f_{t}^{2} u_g a^{\tfrac{1}{2}}\right)
\right| 
&\leq
\left|
\tau \left( a^{\tfrac{1}{2}} f_{t}^{2} u_g a^{\frac{1}{2}} \right)
-
\tau \left( a^{\tfrac{1}{4}} f_{t} a^{\frac{1}{4}} f_t u_g a^{\frac{1}{2}} \right)\right| 
\\ &+ \notag
\left|
\tau \left( a^{\tfrac{1}{8}} f_{t} u_g a^{\tfrac{1}{2}} a^{\tfrac{1}{4}} f_{t} a^{\tfrac{1}{8}} \right)
-
\tau \left( a^{\tfrac{1}{8}} f_{t} u_g f_{t} a^{\frac{7}{8}} \right)
\right| 
\\&+ \notag
\left|
\tau \left(a^{\tfrac{1}{8}} f_{t} u_g f_t a^{\tfrac{7}{8}}\right)
\right| 
\\ &< \notag
\eta \ep' + \eta \ep' + 2 \ep' \leq \frac{\ep}{6}. 
\end{align}
Fourth, using 
(\ref{2-W.T.R.P.66}) and (\ref{2-W.T.R.P.22}) at the fourth step, we get
\begin{align}\label{EQ5.2024.06.20}
\left\|
f u_{g} f - \sum_{t \in G} f_{t}^{2} u_{g}
\right\|
&=
\left\|
\sum_{t, h} f_{t} u_{g} f_{h} - \sum_{t} f_{t}^{2} u_{g}
\right\| 
\\& = \notag
\left\| 
\sum_{t, h} f_{t} \beta_g (f_h) u_{g} - \sum_{t} f_{t}^{2} u_{g} 
\right\|
 \\& \leq \notag
\left\|
\sum_{t, h} f_{t} \beta_g (f_h) u_{g} 
-
\sum_{t, h} f_{t} f_{g h} u_{g}
\right\| 
\\&  \quad+ \notag
\left\|
\sum_{t \neq gh} f_{t} f_{g h} u_{g}
\right\| 
\\&< \notag
2 \ \card (G)^2 \ep'.
\end{align}
Fifth , we use the first part of
(\ref{2-W.T.R.P.55}) at the first step and use (\ref{EQ3.2024.06.20}) and (\ref{EQ5.2024.06.20}) at the second step to get, for every $g \in G$, 
\begin{align}\label{EQ11.2024.06.20}
 \left|
 \tau\left( a^{\tfrac{1}{2}} f^2 u_{g}  a^{\tfrac{1}{2}} \right)
 \right| 
 &=
 \left|
 \tau\left( a^{\tfrac{1}{2}} f u_{g} f a^{\tfrac{1}{2}} \right)
 \right| 
\\ &\leq \notag
 \left|
  \tau\left( a^{\tfrac{1}{2}} f u_{g} f a^{\tfrac{1}{2}} \right)
   -
  \sum_{t\in G} \tau \left( a^{\tfrac{1}{2}} f_t^{2} u_g a^{\tfrac{1}{2}} \right)
  \right| 
   \\ & \quad+ \notag
   \left|
   \sum_{t\in G} \tau \left( a^{\tfrac{1}{2}} f_t^{2} u_g a^{\tfrac{1}{2}} \right)
   \right| 
   \\& < \notag
   2 \eta \, \card (G)^2 \ep' +  \frac{\ep}{6} < \frac{\ep}{3}. 
\end{align}
A simple computation shows that
\begin{equation}\label{EQ7.2024.06.20}
p a^{\tfrac{1}{2}} f^2 a^{\tfrac{1}{2}} p = 
\frac{1}{\card (G)} \cdot a^{\tfrac{1}{2}} f^2 a^{\tfrac{1}{2}} u_1 + 
\frac{1}{\card (G)} \cdot \sum_{g \in G \setminus \{1\}}  a^{\tfrac{1}{2}} f^2 a^{\tfrac{1}{2}} u_g.
\end{equation}
Finally, using (\ref{EQ7.2024.06.20}) at the first step and using (\ref{EQ11.2024.06.20}) at the second step, we get
\begin{align*}
\left|
\tau \left(p a^{\tfrac{1}{2}} f^{2} a^{\tfrac{1}{2}} p\right)
-
\frac{1}{\card (G)} \cdot \tau \left( a^{\frac{1}{2}} f^{2} a^{\frac{1}{2}} \right)
\right| 
& = 
\left|
\frac{1}{\card (G)} \cdot 
\sum_{g \in G \setminus \{1\}} 
\tau \left(
 a^{\tfrac{1}{2}} f^2 a^{\tfrac{1}{2}} u_g
 \right)
\right|
 \\&< \notag
 2 \eta \, \card (G)^2 \ep' + 4 \, \card (G) \ep' < \frac{\ep}{3}. 
\end{align*}
This completes the proof of (\ref{EQ11.2024.06.17}). 
Now, we use Lemma~\ref{Fixedpoint_corner}(\ref{Fixedpoint_corner_c}) at the fist step and use
(\ref{EQ8.2024.06.17}),  (\ref{EQ9.2024.06.17}),
and (\ref{EQ11.2024.06.17}) at the second step to get
\begin{align}
\left|
\tau (pa) - \frac{1}{\card (G)} \cdot \tau (a)
\right| 
&\leq 
\left|
\tau (pap) - \tau \left( p a^{\tfrac{1}{2}} f^2 a^{\tfrac{1}{2}} p \right)
\right| 
\\&+ \notag
\left|
\tau \left( p a^{\tfrac{1}{2}} f^2 a^{\tfrac{1}{2}} p \right)
- 
\frac{1}{\card (G)} \cdot \tau \left(  a^{\tfrac{1}{2}} f^2 a^{\tfrac{1}{2}}  \right)
\right| 
\\&+ \notag
 \frac{1}{\card (G)} \cdot
  \left|
  \tau ( a^{\frac{1}{2}} f^2 a^{\frac{1}{2}}) - \tau (a)
  \right| 
  \\&< \notag
  \frac{\ep}{3} + \frac{\ep}{3} + \frac{\ep}{3}= \ep.
\end{align}
This completes the proof. 
\end{proof}
\begin{qst}
Does Theorem~\ref{TrcaeofProj} hold without  the exactness of the C*-algebra $A$?
\end{qst}
The following proposition provides with us a relationship between the radius of comparison of the crossed product relative to $a$ and  to $p\cdot a$. 
\begin{prp}\label{RCandCardinal}
Let $A$ be a simple exact non-type-I C*-algebra. Let 
$G$ be a finite group and let 
$\alpha \colon G \to \Aut(A)$ be an action of $G$ on $A$ with the weak tracial Rokhlin property.  Let
 $\id \colon G \to \Aut (\cK)$ denote the trivial action and 
let 
 $p= \frac{1}{\card (G)} \sum_{g} u_g$ in 
 $\cM \big((A \otimes \mathcal{K}) \rtimes_{\alpha \otimes \id} G\big)$. 
 Let $a$ be a positive element in 
 $\Ped(A^\alpha \otimes \mathcal{K}) \setminus \{0\}$.  Then
  \[
  \rc \Big(\Cu \left(A \rtimes_{\alpha} G) \right), \ [
   (\kappa \circ \iota) (a) ] \Big)  
  =
  \frac{1}{\card (G)}  \cdot 
  \rc \Big(
  \Cu (A \rtimes_{\alpha} G),\  \left[ p \cdot (\kappa \circ \iota) (a) \right]
  \Big).
  \]
\end{prp}
\begin{proof}
Let $r\in (0, \infty)$. It is enough to show that 
$\Cu \left(A \rtimes_{\alpha} G \right)$ has $r$-comparison relative to 
   $[(\kappa \circ \iota) (a) ]$
    if and only if 
$\Cu (A \rtimes_{\alpha} G)$ has $(\card (G)\cdot r)$-comparison relative to $\left[ p \cdot (\kappa \circ \iota) (a) \right]$.    
  Let us assume that 
$\Cu (A \rtimes_{\alpha} G))$  has $r$-comparison relative to $[
   (\kappa \circ \iota) (a)]$.
   Let 
$x, y \in ((A \rtimes_{\alpha} G) \otimes \K)_{+}$ satisfy
\begin{equation}
\label{2024.07.16.Q1}
d_{\tau} ( x )
+ \left(\card (G) \cdot r \right) \cdot
d_{\tau} \left(  p \cdot (\kappa \circ \iota) (a) \right)
\leq
d_{\tau} (y)
\end{equation}
for all $\tau \in \QT_{2} (A \rtimes_{\alpha} G)$.
Since $A$ is exact, it follows that $A \rtimes_{\alpha} G$ is also exact, and therefore, Remark  4.5 of \cite{ERS11}, 
$\EQT_{2} (A \rtimes_{\alpha} G) = \ET (A \rtimes_{\alpha} G)$. Using (\ref{2024.07.16.Q1}) and 
using Theorem (\ref{TrcaeofProj}), we get 
\begin{equation}
\label{2024.07.16.Q2}
d_{\tau} ( x )
+  r \cdot
d_{\tau} ( (\kappa \circ \iota) (a) )
\leq
d_{\tau} (y)
\end{equation}
for all $\tau \in \ET (A \rtimes_{\alpha} G)$. Since $\Cu (A \rtimes_{\alpha} G))$  has $r$-comparison relative to $[
   (\kappa \circ \iota) (a)] \big)$, it follows that 
   $x \precsim_{A \rtimes_{\alpha} G} y$. This shows that $\Cu (A \rtimes_{\alpha} G)$ has $(\card (G)\cdot r)$-comparison relative to $\left[ p \cdot (\kappa \circ \iota) (a) \right]$.    
  The backward implication can be proved similarly. 
\end{proof}
\begin{rmk}
In the  proposition above, if the C*-algebra $A$ is not necessarily exact and $a \in (A^\alpha \otimes \mathcal{K}) \setminus \{0\}$ not necessarily in $\Ped(A^\alpha \otimes \mathcal{K}) \setminus \{0\}$, then we
 have 
 \[
 \rc \Big(\Cu \left(A \rtimes_{\alpha} G) \right),\
  [(\kappa \circ \iota) (a) ]\Big)  
  \leq 
  \rc \Big(
  \Cu (A \rtimes_{\alpha} G), \ \left[ p \cdot (\kappa \circ \iota) (a)  \right]
  \Big).
  \]
  To see this, let $r\in (0, \infty)$. It is enough to show that if
$\Cu (A \rtimes_{\alpha} G)$ has $r$-comparison relative to $\left[ p \cdot  (\kappa \circ \iota) (a) \right]$, then   
$\Cu (A \rtimes_{\alpha} G)$ has $r$-comparison relative to $[ (\kappa \circ \iota) (a) ])$.  
  Assume that 
$\Cu (A \rtimes_{\alpha} G)$ has $r$-comparison relative to $\left[ p \cdot (\kappa \circ \iota) (a) \right]$. 
   Let 
$x, y \in ((A \rtimes_{\alpha} G) \otimes \K)_{+}$ satisfy
\begin{equation}
\label{2024.07.16.Q111}
d_{\tau} ( x )
+
d_{\tau} ( (\kappa \circ \iota) (a))
\leq
d_{\tau} (y)
\end{equation}
for all $\tau \in \QT_{2} (A \rtimes_{\alpha} G)$.
We use Lemma~\ref{Fixedpoint_corner}(\ref{Fixedpoint_corner_c}) at the first step to get 
\[
p \cdot (\kappa \circ \iota) (a) = p \cdot (\kappa \circ \iota) (a)  \cdot p \precsim_{A \rtimes_{\alpha} G} (\kappa \circ \iota) (a).
\]
This implies that 
$d_\tau \left( p\cdot \big((\kappa \circ \iota) (a)\big) \right) \leq d_{\tau} ((\kappa \circ \iota) (a))$
for all $\tau \in \EQT_{2} (A \rtimes_{\alpha} G)$. 
Using this and (\ref{2024.07.16.Q111}), we get
\begin{equation}
\label{2024.07.16.Q222}
d_{\tau} ( x )
+ r \cdot
d_{\tau} \left( p \cdot  (\kappa \circ \iota) (a) \right)
\leq
d_{\tau} (y)
\end{equation}
for all $\tau \in \EQT_2 (A \rtimes_{\alpha} G)$. 
Using this and the fact that $\Cu (A \rtimes_{\alpha} G)$ has $r$-comparison relative to $\left[ p \cdot  (\kappa \circ \iota) (a) \right]$, it follows that 
   $x \precsim_{A \rtimes_{\alpha} G} y$. This shows that $\Cu (A \rtimes_{\alpha} G)$ has $r$-comparison relative to $[ (\kappa \circ \iota) (a) ]$.      
  \end{rmk}
In the following theorem, as promised, we obtained the relationships between the relative radius of comparison of the crossed product by finite group actions with (non-unital) weak tracial Rokhlin property,
the relative radius of comparison of the fixed point algebra, and the relative radius of comparison of the underlying C*-algebra. 
\begin{thm}\label{MainThmRC}
Let $A$ be a stably finite simple  non-type-I  exact  (not necessarily unital) C*-algebra. Let 
$G$ be finite group and let 
$\alpha \colon G \to \Aut(A)$ be an action of $G$ on $A$ with the weak tracial Rokhlin property. 
 Let $a$ be a positive element in 
  $\Ped(A^\alpha \otimes \mathcal{K}) \setminus \{0\}$. Then: 
  \begin{enumerate}
  \item\label{MainThmRC.1}
  $\rc \Big(\Cu \left(A \rtimes_{\alpha} G\right),\ [(\kappa \circ \iota) (a) ]\Big)  
  = 
  \frac{1}{\card (G)} \cdot
  \rc \big(\Cu (A^\alpha), \ [ a ] \big).
  $
\item\label{MainThmRC.2}
$\rc \Big(\Cu \left(A \rtimes_{\alpha} G \right), \
 [(\kappa \circ \iota) (a) ] \Big)  
  \leq 
  \frac{1}{\card (G)}  \cdot 
  \rc \big(
  \Cu (A), \ [ \iota ( a ) ]
  \big)$.
   \end{enumerate}
\end{thm}
\begin{proof}
We prove (\ref{MainThmRC.1}) and (\ref{MainThmRC.2}) together. Let $\id \colon G \to \Aut (\cK)$ denote the trivial action and 
let  
 $p= \frac{1}{\card (G)} \sum_{g} u_g$ in $\cM((A \otimes \mathcal{K}) \rtimes_{\alpha \otimes \id} G)$. 
We use 
 Proposition \ref{RCandCardinal} at the first step, use Corollary \ref{RcofFullCorner} at the third step, 
use Lemma \ref{Fixedpoint_corner} with $A\otimes \mathcal{K}$ in place of $A$ and $\alpha \otimes \id$ in place of $\alpha$ at the fifth step, and use Theorem \ref{RcofFixedPoint} at the last step to get
\begin{align*}
& \card (G) \cdot 
     \rc \Big(
     \Cu \left(  
     A \rtimes_{\alpha} G 
     \right),  \
     [(\kappa \circ \iota) (a) ]
  \Big) 
  \\ 
  & \hspace*{12 em} {\mbox{}}=
 \rc \Big(
 \Cu \left(
  A \rtimes_{\alpha} G 
  \right), \
 [p \cdot (\kappa \circ \iota) (a)  ]
   \Big) 
  \\& \hspace*{12 em} {\mbox{}}=  
   \rc \Big( 
   \Cu \left(
   \left( A \rtimes_{\alpha} G \right) \otimes \mathcal{K}  \right), \
   [   p \cdot (\kappa \circ \iota) (a) ]
   \Big)
  \\&\hspace*{12 em} {\mbox{}}=
   \rc \Big( 
   \Cu \left(p 
   \left( 
   (A \rtimes_{\alpha} G)
    \otimes \mathcal{K} \right) 
   p \right), \
   [ p \cdot  (\kappa \circ \iota) (a) ]
   \Big) 
  \\&\hspace*{12 em} {\mbox{}}=
  \rc \Big(
  \Cu \left(
  p \left( 
  (A \rtimes_{\alpha} G)
   \otimes \mathcal{K} 
  \right) p 
  \right), \
   [ \varphi (a) ]
  \Big) 
  \\&\hspace*{12 em} {\mbox{}}=
  \rc \big(\Cu (A^\alpha \otimes \mathcal{K}),\ [ a ] \big) 
  \\&\hspace*{12 em} {\mbox{}}=
  \rc \big(\Cu (A^\alpha),\ [ a ]\big) 
  \\& \hspace*{12 em} {\mbox{}} \leq
\rc \big(\Cu (A),\ [\iota ( a ) ]\big).
\end{align*}
This completes the proof of the theorem. 
\end{proof}
In the above theorem, if $A$ is unital and if we take $a$ to be $1_A$, then the result immediately follows by Proposition~3.2.3 of \cite{BRTW12}) and Theorem~4.6 of \cite{AGP19}. 
\begin{rmk}
In above theorem, if the C*-algebra $A$ is not necessarily exact and $a \in (A^\alpha \otimes \mathcal{K}) \setminus \{0\}$ not necessarily in $\Ped(A^\alpha \otimes \mathcal{K}) \setminus \{0\}$, then we have 
\[
\rc \big(\Cu \left(A \rtimes_{\alpha} G \right), \
[(\kappa \circ \iota) (a) ] \big) 
  \leq 
  \rc (\Cu (A^\alpha), \ [ a ]) 
  \leq
\rc (\Cu (A), \ [\iota ( a ) ]).
\]
\end{rmk}

\section{Examples}\label{Example}
In this section, we provide some examples that satisfy our main results to demonstrate that our theorems are meaningful.
 
Let us start with one of the most important non-unital C*-algebras. We denote by $\cW$ the Razak-Jacelon algebra. The algebra
$\cW$ is a simple stably finite separable amenable stably projectionless $\cZ$-stable C*-algebra. It is known that
$\Cu (\cW) \cong [0, \infty]$. Also, every element in $\cW$ is a limit of products of two nilpotents  (see Theorem~1.1 of \cite{Rob16}).
For more details about  $\cW$, we refer to \cite{Bi13, Raz02} and, for Rokhlin actions of finite groups on $\cW$, we refer to \cite{Naw21}.
\begin{eme}\label{ExampW}
Let $A$ be a stably finite simple  non-type-I  (not necessarily unital) C*-algebra and 
let $\alpha \colon G \to \Aut(A)$ be an action of a finite group $G$ on $A$.
Let $\beta \colon G \to \Aut(\cW)$ be an action of  $G$ on $\cW$ 
with the weak tracial Rokhlin property. 
Assume $A \otimes \cW \cong \cW$. 
It follows from Lemma~\ref{FG.thm.1.3} that $\alpha \otimes \beta$ has the (non-unital) weak tracial Rokhlin property. 
\begin{enumerate}
\item 
We use Theorem~\ref{Thm.05.22.2019} at the first step to get
\[
\Cu_+  \Big( (A \otimes \cW)^{\alpha \otimes \beta}\Big)
\cup \{0\}
\cong
\Cu_+ (\cW)^{\alpha \otimes \beta} \cup \{0\}
\cong
[0, \infty].
\]
\item
We use Proposition~\ref{CuCrosFix}(\ref{CuCrosFix.3}) at the second step and Theorem~\ref{Thm.05.22.2019} at the third step to get
\begin{align*}
\Cu_+ \Big( (A \otimes \cW) \rtimes_{\alpha \otimes \beta}  G\Big)
\cup \{0\}
&\cong
\Cu_+ \Big( \cW \rtimes_{\beta} G  \Big) 
\\&\cong 
\Cu_+ (\cW^{\beta })  \cup \{0\}
\cong
\Cu_+ (\cW)^{\beta }  \cup \{0\}
\cong [0, \infty].
\end{align*}
\end{enumerate}
\end{eme}
The following two examples are more concrete and are  special cases of
 Example~\ref{ExampW}. 
 \begin{eme}\label{E_9526_JacelonFlip}
 Let $n \in \{ 2, 3, \ldots \}$
 and let $A = \cW^{\otimes n}$,
 the tensor product of $n$ copies of~$\cW$.
 Let  $S_n$ be the symmetric group on the set $\{1,...,n\}$
 and let
  $\alpha \colon S_n \to \Aut (A)$ denote the permutation action given by
\[
\alpha_{\theta}(a_1 \otimes a_2\otimes \ldots  \otimes a_n) = a_{\theta^{-1}(1)} \otimes a_{\theta^{-1}(2)} \otimes \ldots \otimes a_{\theta^{-1}(n)}.
\]
It follows from  Theorem~3.10 of 
 \cite{AGJP17} that
 $\alpha$
 has the (non-unital) weak tracial Rokhlin property. Using Proposition~\ref{CuCrosFix}(\ref{CuCrosFix.3}) at the first step and using Theorem~\ref{Thm.05.22.2019} at the second step, we get
\begin{align*}
 \Cu_+ \big( A \rtimes_{\alpha} S_n  \big) \cup \{0\}
&\cong
\Cu_+ ( A^{\alpha}  ) \cup \{0\}
\\&\cong 
\Cu_+ (A)^{\alpha} \cup \{0\}
= \Cu_+ (\cW)^{\alpha} \cup \{0\}
\cong [0, \infty].
\end{align*}
 \end{eme}
\begin{eme}
Let $A$ be the  $2^\infty$ UHF algebra and 
let $\alpha \colon \mathbb{Z} / 2 \mathbb{Z} \to \Aut(A)$ be an action 
with the  Rokhlin property.
Let $\beta \colon \mathbb{Z} / 2 \mathbb{Z} \to \Aut(\cW)$  
have the weak tracial Rokhlin property. Clearly, $A \otimes \cW \cong \cW$. Then
\[
\Cu_+  \Big( (A \otimes \cW)^{\alpha \otimes \beta}\Big) \cup \{0\}
\cong
[0, \infty]
\quad
\mbox{ and } 
\quad
\Cu_+ \Big( (A \otimes \cW) \rtimes_{\alpha \otimes \beta} \mathbb{Z} / 2 \mathbb{Z}) \Big) \cup \{0\}
\cong [0, \infty].
\]
\end{eme} 
  \begin{eme}\label{E_9526_Flip}
 Let $A$ be a stable finite simple unital \ca{}
 which is covered by the Elliott classification program,
 and let $B \subset A$ be a non-unital \hsa.
 Then, by Theorem~3.10 of 
 \cite{AGJP17}, the tensor flip action $\alpha$ on $B \otimes B$
 has the (non-unital) weak tracial Rokhlin property.
 Then, by Theorem~\ref{Thm.05.22.2019} and Proposition~\ref{CuCrosFix}(\ref{CuCrosFix.3}),
 \[
 \Cu_{+} \big((B \otimes B)^\alpha \big) \cup \{0\} \cong \Cu_+ (B \otimes B)^{\alpha}  \cup \{0\} 
 \cong 
 \Cu_+ \Big( (B \otimes B) \rtimes_{\alpha} \mathbb{Z} / 2 \mathbb{Z}) \Big) \cup \{0\}. 
 \]
 \end{eme}
\begin{eme}\label{E_9526_UHFTensorB}
 Let $A = \bigotimes_{k = 1} ^{\infty} M_{3}$
 denote the UHF~algebra of type $3^{\infty}$ and
 let $B$ be a non-unital simple C*-algebra.
 Then $A \otimes B$ is $\mathcal{Z}$-absorbing, and 
 so is tracially $\mathcal{Z}$-absorbing.
 Consider the action
 $\alpha \colon {\mathbb{Z}}_{2} \to \mathrm{Aut} (A)$ generated by
 the automorphism
 \[
 \bigotimes_{k = 1}^{\infty}
   \mathrm{Ad}
    \begin{pmatrix}1&0&0 \cr 0&1&0 \cr 0&0& - 1 \end{pmatrix}
  \in \mathrm{Aut} (A),
 \]
 and let $\beta \colon {\mathbb{Z}}_{2} \to \mathrm{Aut} (B)$
 be an arbitrary action.
 The action $\alpha$ has the tracial Rokhlin property
 (for unital C*-algebras and using projections),
 so
 $\alpha \otimes \beta
  \colon {\mathbb{Z}}_{2} \to \mathrm{Aut} (A \otimes B)$
 has the (non-unital) weak tracial Rokhlin property by Lemma~\ref{FG.thm.1.3}.
 Therefore,
 \[
 \Cu_{+} \big((A \otimes B)^{\alpha \otimes \beta} \big) \cup \{0\} \cong \Cu_+ (A \otimes B)^{\alpha \otimes \beta}  \cup \{0\} 
 \cong 
 \Cu_+ \Big( (A \otimes B) \rtimes_{{\alpha \otimes \beta}} \mathbb{Z} / 2 \mathbb{Z}) \Big) \cup \{0\}. 
 \]
 \end{eme}
 \begin{eme}\label{E_9526_HSA}
 Let $A$ be a stably finite simple unital \ca,
 let $G$ be a finite group,
 and let $\af \colon G \to \Aut (A)$
 be an action of $G$ on~$A$
 which has the weak tracial Rokhlin property.
 Let $B \subset A$ be a non-unital $\af$-invariant \hsa ~and 
 let $\beta \colon G \to \Aut (B)$ denote the restriction of $\alpha$ to $B$. 
 Then it follows from Proposition~4.2 of \cite{FG20} that $\beta$
 has the (non-unital) weak tracial Rokhlin property. 
 Now, let $b \in (B^{\beta} \otimes \K)_{+} \setminus \{0\}$.
 So, by Theorem~\ref{RcofFixedPoint},  we have
\[
\rc \big(\Cu (B^\alpha), \ [ b ]_B \big)   \leq \rc \big(\Cu (B), \ [\iota ( b ) ]_B \big).
\]  
If further, $A$ is exact, then
\[
\rc \Big(\Cu \left(B \rtimes_{\beta} G\right), \
[ (\kappa \circ \iota) (b) ] \Big)  
  = 
  \frac{1}{\card (G)} \cdot
  \rc \big(\Cu (B^\beta), \ [ b ]\big),
  \]
  and 
\[
\rc \Big(\Cu \left(B \rtimes_{\beta} G \right), \
[(\kappa \circ \iota) (b) ] \Big)  
  \leq 
  \frac{1}{\card (G)}  \cdot 
  \rc \big(
  \Cu (A), \ [ \iota ( b ) ]
  \big),
   \]
   where
$\iota \colon B^\beta \to B$
 and  
$\kappa \colon B \to B\rtimes_{\beta} G$
are the inclusion maps. 
 \end{eme}
The following two examples are  specific, more concrete, cases of the above example. 
 \begin{eme}\label{E_9526_MainHer}
 Let $A$ and $\alpha \colon {\mathbb{Z}}_{2} \to \mathrm{Aut} (A)$
 be as in the main example in Section~6~of~\cite{AGP19}.
 Let $B \subset A$ be a non-unital $\af$-invariant \hsa.
 Then the restriction of $\af$ to~$B$
 has the (non-unital) weak tracial Rokhlin property.
 \end{eme}
 \begin{eme}\label{E_9526_MainHer2}
 Let $A$ and $\alpha \colon {\mathbb{Z}}_{2} \to \mathrm{Aut} (A)$
 be as in the main example in Section~6~of~\cite{AGP19}.
 Let $a \in (A^{\alpha})_{+} \setminus \{0\}$.
 Using
 Theorem~6.21 of \cite{AGP19} at the first step, using the fact that $A$ is a stably finite simple unital C*-algebra together with Proposition~3.2.3 of \cite{BRTW12} at the second step, using the fact that $[ \iota (a) ]_A \leq [ 1_A ]_A$ at the third step,  using Theorem~\ref{MainThmRC}(\ref{MainThmRC.1}) at the fourth step, and using Theorem~\ref{MainThmRC}(\ref{MainThmRC.2}) at the last step, we get
 \begin{align*}
0 <
 \operatorname{rc} (A \rtimes_{\alpha} {\mathbb{Z}}_{2}) 
 &=
 \rc \Big(\Cu \left(A \rtimes_{\alpha} {\mathbb{Z}}_{2} \right), \ 
  [(\kappa \circ \iota) (1_A) ] \Big)  
 \\& \leq 
  \rc \Big(
  \Cu \left(A \rtimes_{\alpha} {\mathbb{Z}}_{2} \right), \
   [(\kappa \circ \iota) (a) ] \Big)  
\\&=  \frac{1}{2} \cdot
  \rc \big(\Cu (A^\alpha), \ [ a ]\big)
  \leq 
  \frac{1}{2} \cdot
  \rc \big(\Cu (A), [ \iota (a) ]\big).
 \end{align*}
 \end{eme}
 \begin{qst}
 In Example \ref{E_9526_MainHer2}, can we prove that
 \[
 \rc (\Cu (A^\alpha), \ [ a ])
 =
  \rc (\Cu (A), \ [ \iota (a) ])?
 \]
 \end{qst}

%

\subsection*{Acknowledgements}
 Both authors would like to especially thank N. Christopher Phillips for the many discussions they had over the course of writing this paper.  
 They would like to thank Bruce Blackadar, Ilan Hirshberg,  Zhuang Niu,  Hannes Thiel, and Andrew S. Toms for their comments on the relative radius of comparison. Also, 
 the first author is grateful to 
 Feodor Kogan for his helpful questions on this work during the Fields Operator Algebra Seminar. 
 
Part of this work was completed while the first author was at the Fields Institute for Research in Mathematical Sciences, as a visiting member, during the period July 2023 to December 2024. He is very grateful for the support and hospitality provided by the Fields Institute, with special thanks to  Miriam Schoeman. 

\end{document}